\documentclass[12pt,reqno]{amsart}

\usepackage[dvipsnames]{xcolor}
\usepackage{xspace}

\usepackage{fullpage}
\usepackage{amscd}
\usepackage{tikz}
\usepackage{tikz-cd}
\usepackage{tikz}
\usetikzlibrary{decorations.pathreplacing}
\usetikzlibrary{decorations.markings}
\usetikzlibrary{calc}
\usepackage[matrix,arrow,curve,frame]{xy}
\usepackage{xy}

\usepackage{hyperref}
\usepackage{amsmath,amsthm,amssymb,mathrsfs}

\newcommand\myshade{85}
\colorlet{mylinkcolor}{Aquamarine}
\colorlet{mycitecolor}{YellowOrange}
\colorlet{myurlcolor}{violet}

\hypersetup{
	linkcolor  = Mahogany,
	citecolor  = magenta,
	urlcolor   = myurlcolor!\myshade!black,
	colorlinks = true,
}

\newcommand\doi[2]{\href{http://dx.doi.org/#1}{#2}}

\newtheorem{thm}{Theorem}[section]
\newtheorem{mthm}{Main Theorem}
\newtheorem{cor}[thm]{Corollary}
\newtheorem{lemma}[thm]{Lemma}
\newtheorem{keylemma}[thm]{Key Lemma}
\newtheorem{introkeylemma}{Key Lemma}

\newtheorem{prop}[thm]{Proposition}
\newtheorem{conj}[thm]{Conjecture}
\theoremstyle{definition}
\newtheorem{defn}[thm]{Definition}
\newtheorem{remark}[thm]{Remark}

\newtheorem{rema}[thm]{Remark}

\newcommand{\cA}{\mathcal{A}}
\newcommand{\cB}{\mathcal{B}}
\newcommand{\cC}{\mathcal{C}}
\newcommand{\cD}{\mathcal{D}}
\newcommand{\cF}{\mathcal{F}}
\newcommand{\cG}{\mathcal{G}}

\newcommand{\cI}{\mathcal{I}}

\newcommand{\cM}{\mathcal{M}}
\newcommand{\cP}{\mathcal{P}}

\newcommand{\cR}{\mathcal{R}}
\newcommand {\cU}{\mathcal{U}}
\newcommand{\cV}{\mathcal{V}}
\newcommand {\cW}{\mathcal{V}}
\newcommand{\cY}{\mathcal{Y}}
\newcommand{\cZ}{\mathcal{Z}}

\newcommand {\cVec}{\mathcal{V}ec}
\newcommand{\HHom}{\mathrm{Hom}}
\newcommand{\intHom}{\underline{\HHom}}
\newcommand {\cCop}{\mathcal{C}^{\mathrm{op}}}
\newcommand {\cCrev}{\mathcal{C}^{\mathrm{rev}}}
\newcommand{\obj}{\mathrm{Obj}}
\newcommand{\irr}{\mathrm{Irr}}

\newcommand{\vac}{\mathbf{1}}

\newcommand{\Hom}[2]{\text{Hom}_{#1}\left( #2 \right)}

\newcommand{\NN}{\mathbb{N}}
\newcommand{\ZZ}{\mathbb{Z}}
\newcommand{\CC}{\mathbb{C}}
\newcommand{\one}{\mathbf{1}}
\newcommand{\Id}{\mathrm{Id}}

\DeclareMathOperator{\LHS}{LHS}
\DeclareMathOperator{\RHS}{RHS}

\DeclareMathOperator{\fin}{fin}

\DeclareMathOperator{\rep}{Rep}
\DeclareMathOperator{\rev}{rev}

\newcommand{\voa}{vertex operator algebra\xspace}
\newcommand{\voas}{vertex operator algebras\xspace}
\newcommand{\fus}{\boxtimes}
\newcommand{\pfus}[1]{\boxtimes_{P( #1 )}}
\newcommand{\alg}[1]{\mathsf{#1}}
\newcommand{\aA}{{\alg{A}}}
\newcommand{\aU}{{\alg{U}}}
\newcommand{\aW}{{\alg{V}}}
\newcommand{\aC}{{\alg{C}}}

\newcommand{\Mod}[1]{\mathsf{#1}}

\newcommand{\mI}{{\Mod{I}}}
\newcommand{\mM}{{\Mod{M}}}
\newcommand{\mN}{{\Mod{N}}}
\newcommand{\mU}{{\Mod{U}}}
\newcommand{\mW}{{\Mod{V}}}
\newcommand{\mX}{{\Mod{X}}}
\newcommand{\mY}{{\Mod{Y}}}
\newcommand{\mZ}{{\Mod{Z}}}
\newcommand{\aV}{{\Mod{V}}}

\newcommand{\repA}{{\rep\aA}}
\usepackage{bbm}

\numberwithin{equation}{section}

\title{Gluing Vertex Algebras}
\author{Thomas Creutzig, Shashank Kanade and Robert McRae}
\date{}

\address{(TC) Department of Mathematical and Statistical Sciences, University of Alberta, Edmonton, Alberta T6G 2G1, Canada}

 \email{creutzig@ualberta.ca}

 \address{(SK) Department of Mathematics, University of Denver, Denver, CO 80208, USA}
\email{\texttt{shashank.kanade@du.edu}}
 
 \address{(RM) Department of Mathematics,
Vanderbilt University, 
Nashville, Tennessee 37240, USA}
\curraddr{Yau Mathematical Sciences Center, Tsinghua University, Beijing 100084, China}
  \email{rhmcrae@tsinghua.edu.cn}

\allowdisplaybreaks

\begin{document}

\begin{abstract}
  We relate commutative algebras in braided tensor categories to braid-reversed tensor equivalences, motivated by vertex algebra representation theory.
  First, for $\mathcal{C}$ a braided tensor category, we give a detailed account of the canonical algebra construction in the Deligne product $\mathcal{C}\boxtimes\mathcal{C}^{\mathrm{rev}}$. Especially, we show that if $\mathcal{C}$ 
  is semisimple but not necessarily finite or rigid, then 
  $\bigoplus_{\mX\in\mathrm{Irr}(\mathcal{C})} \mX'\boxtimes\mX$ is a commutative algebra, where $\mX'$ is a representing object 
  for the functor $\mathrm{Hom}_{\mathcal{C}}(\bullet\otimes_{\cC}\mX,\mathbf{1}_\cC)$ (assuming $\mX'$ exists) and the sum runs over all inequivalent simple objects of $\mathcal{U}$. 
  Conversely, let $\aA=\bigoplus_{i\in I} 
  \mU_i\boxtimes\mW_i$ be a simple commutative algebra in a Deligne product $\mathcal{U}\boxtimes\mathcal{V}$ 
  with $\mathcal{U}$ semisimple and rigid but not necessarily finite, and $\mathcal{V}$ 
  rigid but not necessarily semisimple. We show that if the unit objects $\mathbf{1}_{\mathcal{U}}$ and $\mathbf{1}_{\mathcal{V}}$ form a commuting pair in $\aA$ in a suitable sense, then 
  there is a braid-reversed equivalence between (sub)categories of $\mathcal{U}$ 
  and $\mathcal{V}$ that sends $\mU_i$ to $\mW_i^*$.

  These results apply when $\mathcal{U}$ and $\mathcal{V}$ are braided (vertex) tensor categories of modules for
  simple vertex operator algebras $\mU$ and $\mW$, respectively: Given $\tau:\mathrm{Irr}(\mathcal{U})\rightarrow\mathrm{Obj}(\mathcal{V})$ such that $\tau(\mU)=\mW$, we glue $\mU$ and $\mW$ along $\mathcal{U}\boxtimes\mathcal{V}$ via $\tau$ to create $\aA=\bigoplus_{\mX\in\mathrm{Irr}(\mathcal{U})} \mX'\otimes\tau(\mX)$. Then under certain conditions, $\tau$ extends to a braid-reversed equivalence between $\mathcal{U}$ and $\mathcal{V}$ if and only if $\aA$ has a simple conformal vertex algebra structure that (conformally) extends $\mU\otimes \mW$. 
  As examples, we glue suitable Kazhdan-Lusztig 
  categories at generic levels to construct new vertex algebras extending the tensor product of two affine vertex subalgebras, and we prove braid-reversed equivalences between certain module subcategories for affine vertex algebras and $W$-algebras at admissible levels.
 %
%
%
 %
 %
 \end{abstract}

\maketitle

\tableofcontents

\section{Introduction}

We study the relation between certain types of commutative associative algebra objects in braided tensor categories and braid-reversed equivalences of tensor categories, motivated by  \voa{} theory and its applications to geometry and physics. Commutative associative algebra objects in tensor categories of modules for a vertex operator algebra are the same as \voa{} extensions \cite{HKL}, and such extensions together with equivalences of vertex tensor subcategories are crucial in the context of $S$-duality for four-dimensional supersymmetric $GL$-twisted gauge theories \cite{CG} and the quantum geometric Langlands correspondence \cite{AFO}. In gauge theory, \voas{} are associated to two-dimensional intersections of three-dimensional topological boundary conditions, while categories of \voa{} modules are associated to line defects ending on these boundary conditions. Boundary conditions can be concatenated to form new types of boundary conditions, and the resulting \voas{} are precisely the type of extensions studied in this work. Categories of \voa{} modules appearing in these problems are usually not finite and are often, but not necessarily, semisimple. Thus we derive results in a setting general enough for these applications, especially allowing braided tensor categories to have infinitely many inequivalent simple objects. 

We will now describe our categorical results, followed by \voa{} applications and comments regarding \voa{} theory and existing literature.

\subsection{Tensor category results}

Let $\cC$ and $\cD$ be braided tensor categories and $\tau$ a map from simple objects in $\cC$ to objects in $\cD$. Then we consider objects in a direct sum completion $(\cC \boxtimes \cD)_\oplus$ of the Deligne product $\cC \boxtimes \cD$ of the form
\[
	\aA = \bigoplus_{\mX \in \irr(\cC)} \mX \boxtimes \tau(\mX) \ \in  \ \obj\left((\cC \boxtimes \cD)_\oplus\right).
\]
We aim to prove under suitable conditions on $\cC$ and $\cD$ that a commutative associative algebra structure on $\aA$ is equivalent to a  braid-reversed tensor equivalence between $\cC$ and $\cD$. Thus we ask two questions: under which conditions on $\cC$ and $\cD$ does a commutative algebra object imply a braid-reversed equivalence, and conversely what do we need to assume so that a braid-reversed equivalence yields a commutative associative algebra?

\subsubsection{From braid-reversed equivalences to commutative algebra objects}${}$\vspace{1mm}

It is known  (see for example \cite[Exercise 7.9.9]{EGNO}) that associative algebras can be constructed from module categories; as we could not find a complete proof in the literature, we give this construction in Theorem \ref{thm:assoc}. Essentially, for $\cC$ a (multi)tensor category and $\cM$ a $\cC$-module category such that internal Homs exist, the internal End of an object $\mM$ in $\cM$ can be given the structure of an associative algebra. 
Moreover, when $\cC$ is a braided fusion category, \cite[Lemma 3.5]{DMNO} shows that this algebra is also commutative as long as induction from $\cC$ to $\cM$ is a central functor, that is, it factors through the center of the module category. In Theorem \ref{thm:comm}, we prove this statement in detail without assuming finiteness or semisimplicity.

A useful property in showing that internal Homs exist is rigidity, that is, existence of duals. Unfortunately, braided tensor categories of vertex operator algebra modules are often not known to be rigid. However, we find that internal Homs can exist under the weaker assumption that $\cC$ has a contragredient functor, that is, a contravariant endofunctor $\mX\mapsto\mX'$ such that there is a natural isomorphism
\begin{equation*}
 \HHom_\cC(\mX\otimes\mY,\vac)\cong\HHom_\cC(\mX,\mY')
\end{equation*}
for objects $\mX$, $\mY$ in $\cC$. For vertex algebraic tensor categories, such a functor arises from the contragredient modules of \cite{FHL}, provided that the vertex operator algebra itself is self-contragredient.

With these preparations, we can state our first main result; for precise notation we refer to Section \ref{sec:canonical}. The algebra of this theorem is called the {canonical algebra} in $\cC\boxtimes \cCrev$, where $\cCrev=\cC$ as a tensor category but has reversed braidings. 
\begin{mthm}\label{mthm:canonicalalgebra}
Let $\cC$ be a (not necessarily finite) semisimple braided tensor category with a contragredient functor. Then 
\[
\aA = \bigoplus_{\mX \in \irr(\cC)} \mX' \boxtimes \mX
\]
is a commutative associative algebra in $(\cC\boxtimes \cCrev)_{\oplus}$. If $\cC$ is rigid, then $\aA$ is simple and for simple objects $\mX, \mY, \mZ$ of $\cC$, the multiplication rules are given by $M_{\mX^* \boxtimes \mX, \mY^* \boxtimes \mY^*}^{\mZ^* \boxtimes \mZ}=1$
if and only if $\mZ$ is a summand of $\mX\otimes\mY$.
\end{mthm}

 Since commutative algebras are preserved by braided tensor equivalences, we can restate Main Theorem \ref{mthm:canonicalalgebra} as follows. Let $\cC$ be a semisimple braided tensor category with a contragredient functor, and suppose $\tau:\cC\rightarrow\cD$ is a braid-reversed tensor equivalence (so that $\tau: \cCrev\rightarrow\cD$ is a braided equivalence). Then
 \begin{equation*}
  \aA=\bigoplus_{\mX\in\irr(\cC)} \mX'\fus\tau(\mX)
 \end{equation*}
is a commutative associative algebra in $(\cC\fus\cD)_\oplus$, and if $\cC$ is rigid, then $\aA$ is simple.

\subsubsection{From commutative algebra objects to braid-reversed equivalences} ${}$ \vspace{1mm}

For the converse question, we work in the following setting: 
\begin{enumerate}
\item $\cU$ is a (not necessarily finite) 
semisimple ribbon category, and $\lbrace\mU_i\rbrace_{i\in I}$ is a subset of distinct simple objects in $\cU$ that includes $\mU_0=\vac_\cU$.

\item $\cW$ is a ribbon category. In particular, both $\cU$ and $\cW$ are rigid.
\item We have a simple (commutative, associative, unital) algebra 
\begin{align*}
\aA  
= \bigoplus_{i\in I} \mU_{i}\boxtimes \mW_i.
\end{align*}
in $\cC=\cU\fus\cW$, or $\cC_\oplus$ if $I$ is infinite, where the $\mW_i$ are objects of $\cW$, not assumed to be simple except for $\mW_0=\vac_\mW$.

\item  The tensor units $\one_\cU=\mU_0, \one_\cW=\mW_0$ form a mutually commuting (or dual) pair in $\aA$, in the sense that
\begin{equation*}
\dim\Hom{\cU}{\Mod{U}_0, \Mod{U}_i} =\delta_{i,0} = \dim\Hom{\cW}{\Mod{V}_0, \Mod{V}_i}.
\end{equation*}

\item There is a partition $I=I^0\sqcup I^1$ of the index set such that $0\in I^0$ and for each $i\in I^j$, $j=0,1$, the twist satisfies $\theta_\aA\vert_{\mU_i\fus\mW_i} =(-1)^j\Id_{\mU_i\fus\mW_i}$. In particular, $\theta_{\aA}^2=\Id_\aA$.

\end{enumerate}
Under these conditions, we define $\cU_\aA\subseteq\cU$ and $\cW_\aA\subseteq\cW$ to be the full subcategories whose objects are isomorphic to direct sums of the $\mU_i$ and $\mW_i$, respectively, and prove in Proposition \ref{prop:UA_WA_ribbon} and Theorem \ref{thm:main_bequiv}:
\begin{mthm}
In the setting of this section,
\begin{enumerate}
 \item The categories $\cU_\aA\subseteq\cU$ and $\cW_\aA\subseteq\cW$ are ribbon subcategories. Moreover, $\cW_\aA$ is semisimple with distinct simple objects $\lbrace\mW_i\rbrace_{i\in I}$.
 \item There is a braid-reversed tensor equivalence $\tau:\cU_\aA\rightarrow\cW_\aA$ such that $\tau(\mU_i)\cong\mW_i^*$ for $i\in I$.
\end{enumerate}
\end{mthm}
This theorem relies on the following Key Lemma; here $\cF$ is the the induction functor from $\cC$ to the category $\repA$ of left $\aA$-modules in $\cC$:
\begin{introkeylemma}\label{lem:introkey}
	For all $i\in I$, $\cF(\mU_i\fus\vac_\cW)\cong\cF((\vac_\cU\fus\mW_i)^\ast)$
	in $\repA$.
\end{introkeylemma}

\subsection{Applications to vertex operator algebras}

Our categorical results translate into the following theorem for \voas; see Theorem \ref{thm:VOA} of the main text:
\begin{mthm}\label{mthm:voa}
 Let $\cU$ and $\cW$ be locally finite module categories for simple and self-contragredient vertex operator algebras $\mU$ and $\mW$, respectively, that are closed under contragredients and admit vertex tensor category structure as in \cite{HLZ1}-\cite{HLZ8} and thus also braided tensor category structure. Assume moreover that $\cU$ is semisimple and $\cW$ is closed under submodules and quotients.
 \begin{enumerate}
  \item Suppose $\lbrace\mU_i\rbrace_{i\in I}$ is a set of representatives of equivalence classes of  simple modules in $\cU$ with $\mU_0=\mU$ and $\tau: \cU\rightarrow\cW$ is a braid-reversed tensor equivalence with twists satisfying $\theta_{\tau(\mU_i)} =\pm\tau(\theta_{\mU_i}^{-1})$ for $i\in I$. Then
  \begin{equation*}
   \aA=\bigoplus_{i\in I} \mU_i'\otimes\tau(\mU_i)
  \end{equation*}
is a $\frac{1}{2}\mathbb{Z}$-graded conformal vertex algebra extension of $\mU\otimes\mW$. Moreover, if $\cU$ is rigid, then $\aA$ is simple and the multiplication rules of $\aA$ satisfy $M_{\mU_i'\otimes\tau(\mU_i), \mU_j'\otimes\tau(\mU_j)}^{\mU_k'\otimes\tau(\mU_k)}=1$ if and only if $\mU_k$ occurs as a submodule of $\mU_i\fus\mU_j$.

\item Conversely, suppose $\cU$ and $\cW$ are both ribbon categories, $\lbrace\mU_i\rbrace_{i\in I}$ is a set of distinct simple modules in $\cU$ with $\mU_0=\mU$, and
\begin{equation*}
 \aA=\bigoplus_{i\in I} \mU_i\otimes\mW_i
\end{equation*}
is a simple $\frac{1}{2}\mathbb{Z}$-graded conformal vertex algebra extension of $\mU\otimes\mW$, where the $\mW_i$ are objects of $\cW$  satisfying
\begin{equation*}
 \dim\HHom_\cW(\mW,\mW_i)=\delta_{i,0}
\end{equation*}
and there is a partition $I=I^0\sqcup I^1$ of the index set with $0\in I^0$ and
\begin{equation*}
 \bigoplus_{i\in I^j} \mU_i\otimes\mW_i =\bigoplus_{n\in\frac{j}{2}+\ZZ} \aA_{(n)}
\end{equation*}
for $j=0,1$. Let $\cU_\aA\subseteq\cU$, respectively $\cW_\aA\subseteq\cW$, be the full subcategories whose objects are isomorphic to direct sums of the $\mU_i$, respectively of the $\mW_i$. Then:
\begin{enumerate}
\item $\cU_\aA$ and $\cW_\aA$ are ribbon subcategories of $\cU$ and $\cW$ respectively. Moreover, $\cW_\aA$ is semisimple with distinct simple objects $\lbrace\mW_i\rbrace_{i\in I}$.
\item There is a braid-reversed equivalence $\tau: \cU_\aA\rightarrow\cW_\aA$ such that $\tau(\mU_i)\cong\mW_i'$ for all $i\in I$.
 \end{enumerate}
 \end{enumerate}
\end{mthm}

Conformal vertex algebra extensions as in the first part of the theorem have previously been constructed for certain affine Lie algebra \cite{FS, Zh} and Virasoro \cite{FZ} vertex operator algebras. Also, a closely-related construction due to Huang and Kong \cite{HK, Ko}, starting from braid-equivalent modular tensor categories of representations for vertex operator algebras, yields a conformal full field algebra in the sense of \cite{HK}.
In fact, \cite{Ko} shows that if $\cU$ and $\cW$ are braid-equivalent tensor categories of representations for strongly rational vertex operator algebras $\mU$ and $\mW$, respectively, then conformal full field algebra extensions of $\mU\otimes\mW$ with nondegenerate invariant bilinear form are equivalent to commutative Frobenius algebras in $\cU\fus\cW^{\rev}$ with trivial twist.

The second part of Main Theorem \ref{mthm:voa} (in the case that $\aA$ is $\mathbb{Z}$-graded) has been stated in \cite[Theorem 3.3]{lin} under the assumption that $\mU$ and $\mW$ are strongly rational vertex operator algebras (in particular $I$ is finite in this setting). The proof in \cite{lin} uses semisimplicity of the category $\repA$ of left $\aA$-modules in $\cC$, citing \cite{KO} for this result. However, \cite[Theorem 3.3]{KO} assumes additionally that $\dim_\cC \aA\neq 0$ to prove this semisimplicity, whereas even a modular tensor category can have objects with dimension zero. Relaxing the condition $\dim_\cC\aA \neq 0$ is the work of our Key Lemma \ref{lem:introkey}, so we have in particular filled a gap in \cite{lin}; moreover, we recover the semisimplicity of the category of $\aA$-modules as a consequence, as we now discuss. 

A \voa{} is strongly rational (using terminology from \cite{CGan}) if it is simple, self-contragredient, CFT-type, $C_2$-cofinite, and rational; for such  a \voa{}, the full category of grading-restricted generalized modules is a (semisimple) modular tensor category \cite{Hu4}.
Let $\aV$ be a strongly rational \voa{} and $\aA$ a simple CFT-type \voa{} extension of $\aV$; then $\aA$ is believed to have a modular tensor category of grading-restricted, generalized modules as well.  This indeed follows from Lemma 1.20, Theorem 3.3, and Theorem 4.5 of \cite{KO} as well as \cite[Theorem 3.5]{HKL} (see also \cite[Corollary 3.30]{DMNO}) provided the dimension of $\aA$ as a $\aV$-module is non-zero. Moreover, our previous work \cite[Theorem 3.65]{CKM} shows that this modular tensor category structure is the natural one for module categories of a \voa{}.

Using \cite{Hu4, DJX}, the dimension of $\aA$ in the modular tensor category $\cC$ of $\aV$-modules is strictly positive if all irreducible $\aV$-modules are non-negatively graded, with a non-zero conformal weight $0$ space occurring only in $\aV$ itself. This condition together with the rationality and $C_2$-cofiniteness of $\aV$ ensures that the categorical dimensions of $\aV$-modules are realized by strictly-positive ``quantum dimensions'' defined in terms of characters. But now, we can use the braid-reversed equivalence of Main Theorem \ref{mthm:voa} and \cite[Theorem 2.3]{ENO} to show $\dim_\cC \aA$ is a positive real number without grading-positivity assumptions (see Corollary \ref{cor:dim}):
\begin{cor}
In the setting of Main Theorem \ref{mthm:voa}, assume in addition that $\cU$ and $\cW$ are strongly rational and $\aA$ is a simple CFT-type ($\mathbb{Z}$-graded) \voa{}. Then $\dim_\cC \aA >0$ and $\aA$ is strongly rational; in particular its category of grading-restricted, generalized modules is a semisimple modular tensor category.
\end{cor}

In fact, the calculation of $\dim_\cC \aA$ does not require rationality or $C_2$-cofiniteness for either $\aU$ or $\aV$, so we prove more generally that if the categories $\cU$ and $\cW$ in Main Theorem \ref{mthm:voa} are braided fusion categories, then the category of grading-restricted generalized  $\aA$-modules in $\cC$ is also braided fusion; see Theorem \ref{thm:dim} in the main text for details. In our setting, this removes the grading-positivity assumptions from \cite[Theorem 3.5]{HKL}.

A second corollary of Main Theorem \ref{mthm:voa} relates to the multiplication rules of the extension $\aU\otimes\aV\subseteq\aA$. In general, if $\aA= \bigoplus_{i\in I} \aV_i$ is an extension of a  \voa{} $\aV=\aV_0$ by indecomposable $\aV$-modules $\aV_i$, define the multiplication rule $M_{i, j}^k$ to be $1$ if $\aV_k$ is contained in the operator product algebra of fields of $\aV_i$ with fields of $\aV_j$; otherwise the multiplication rule is $0$. It is clear that $M_{i, j}^k=0$ if the fusion rule $N_{i, j}^k=0$; a question raised in private communication by Chongying Dong is for which extensions the converse is also true. In our setting, this is precisely the content of the statements on multiplication rules in Main Theorems \ref{mthm:canonicalalgebra} and \ref{mthm:voa}. Thus we can rephrase these statements as follows:
\begin{cor}
In the setting of part (1) of Main Theorem \ref{mthm:voa}, and assuming $\cU$ is rigid, the multiplication rules for the canonical algebra
\begin{equation*}
 \aA=\bigoplus_{i\in I} \mU_i^*\otimes\tau(\mU_i)
\end{equation*}
are $0$ if and only if the corresponding fusion rules are $0$.
\end{cor}

\subsection{Examples}

We illustrate our results in two examples; the first illustrates both parts of Main Theorem \ref{mthm:voa} and shows the importance of allowing categories with infinitely many simple objects. 

\subsubsection{Vertex algebras for \texorpdfstring{$S$}{S}-duality and Kazhdan-Lusztig categories at generic level}

The conjecture below is the physics prediction of \voas associated to the intersection of so-called Dirichlet boundary conditions and their general $S$-duals. These \voas{} are claimed in \cite{CG} to play a role as quantum geometric Langlands kernel \voas.

Let $\mathfrak{g}$ be a simple Lie algebra with dominant integral weights $P_+$, and let $\widehat{\mathfrak{g}}$ be the affine Lie algebra associated to $\mathfrak{g}$. For $\lambda\in P_+$ and a level $k\in\CC$, let $\mathbb{L}_{k}(\lambda)$ denote the irreducible highest-weight $\widehat{\mathfrak{g}}$-module with highest weight $\lambda$ at level $k$. For irrational levels $k$, the modules $\mathbb{L}_k(\lambda)$ form the simple objects of the Kazhdan-Lusztig category KL$_k$ of finite-length modules for the simple affine vertex operator algebra $\mathbb{L}_k(\mathfrak{g})$. By \cite{KL1}-\cite{KL4} and \cite{Zha} (see \cite{H3} for a review), KL$_k$ is a semisimple rigid vertex tensor category. Now the following conjecture is known to be true for $n=0$ as we will explain in a moment and also for $\mathfrak g=\mathfrak{sl}_2$ and $n=1$ \cite{CG} and $n=2$ \cite{CGL}:
\begin{conj}\cite[Conjecture 1.1]{CG}
Let $n$ be a non-negative integer, let $P^+_n$ be the set of dominant weights $\lambda$ such that $n\lambda^2$ is an integer, and let $\psi, \psi'$ be generic complex numbers satisfying 
\[
\frac{1}{\psi} +\frac{1}{\psi'} =n.
\]
Then the object 
\[
A^n[\mathfrak g, \psi] := \bigoplus_{\lambda\in P_n^+} \mathbb L_{\psi-h^\vee}(\lambda) \otimes  \mathbb L_{\psi'-h^\vee}(\lambda)
\]
can be given the structure of a simple vertex operator superalgebra. 
\end{conj}

Let $\kappa$ be an irrational number. 
The algebra of chiral differential operators of a compact Lie group $G$ \cite{GMS1, GMS2, AG} with Lie algebra $\mathfrak g$ at level $\kappa$ has the form 
 \[
\mathcal D_\kappa^{\text{ch}}(G) \cong  \bigoplus_{\lambda\in P_+} \mathbb L_{\kappa-h^\vee}(\lambda) \otimes  \mathbb L_{-\kappa-h^\vee}(\lambda^*)
 \]
 as an $\mathbb L_{\kappa-h^\vee}(\mathfrak g) \otimes  \mathbb L_{-\kappa-h^\vee}(\mathfrak g)$-module 
and is a simple \voa (\cite{FS, Zh} and \cite[Proposition 3.15]{Cheung}). Here $\lambda^*=-\omega(\lambda)$, where $\omega$ is the longest element of the Weyl group. 
 Thus $\mathcal D_\kappa^{\text{ch}}(G)$ is a commutative algebra in the direct sum completion of the Deligne product  KL$_{\kappa-h^\vee}$ $\boxtimes$ KL$_{-\kappa-h^\vee}$, so by part (2) of Main Theorem \ref{mthm:voa}, the categories KL$_{\kappa-h^\vee}$ and KL$_{-\kappa-h^\vee}$ are braid-reversed equivalent with $\mathbb L_{\kappa-h^\vee}(\lambda)$ sent to $\mathbb L_{-\kappa-h^\vee}(\lambda)$ under the equivalence. On the other hand, KL$_{\kappa-h^\vee}$ is also braided tensor equivalent to the category of weight modules (representations of type I) for the Lusztig quantum group $U_q(\mathfrak g)$ for $q=\exp\left(\frac{\pi i}{r^\vee\kappa}\right)$ \cite{KL1}-\cite{KL4}. Here $r^\vee=1$ for $\mathfrak{g}$ in types $A$, $D$, and $E$; $r^\vee=2$ in types $B$, $C$, and $F$; and $r^\vee=3$ in type $G$
(this is sometimes called the lacing number of $\mathfrak{g}$, for example in \cite{Ara}).

Let $N$ be the level of the weight lattice $P$ of $\mathfrak g$, that is the smallest positive integer such that $NP$ is integral. The representation category of the rational form of the quantum group is  over $\mathbb Q(s)$ with $s=\exp\left(\frac{\pi i}{Nr^\vee\kappa}\right)$ (\cite{Lus}; see also \cite[Section 1.3]{BK}), so KL$_{\kappa-h^\vee}$ and KL$_{\ell-h^\vee}$ are equivalent if $\frac{1}{\kappa} = \frac{1}{\ell} \mod r^\vee N$. Combining with the braid-reversed equivalence, this means KL$_{\kappa-h^\vee}$ and KL$_{\ell-h^\vee}$ are braid-reversed equivalent if
\[
\frac{1}{\kappa} + \frac{1}{\ell} = mr^\vee N
\]
for some $m\in\mathbb Z$, so that by part (1) of Main Theorem \ref{mthm:voa},
\[
\bigoplus_{\lambda\in P_+} \mathbb L_{\kappa-h^\vee}(\lambda) \otimes  \mathbb L_{\ell-h^\vee}(\lambda^*)
\]
for such $\kappa$ and $\ell$ has the structure of a simple \voa. We can change the simple root system for the second factor by $-\omega$ (for $\omega$ the longest element of the Weyl group) so that 
 \[
A^{mr^\vee N}[\mathfrak g, \kappa] =  \bigoplus_{\lambda\in P_+} \mathbb L_{\kappa-h^\vee}(\lambda) \otimes  \mathbb L_{\ell-h^\vee}(\lambda)
 \]
 also can be given the structure of a simple \voa{}. This proves Conjecture 1.1 of \cite{CG} for $n=mr^\vee N$, that is,
\begin{cor}
Let $N$ be the level of the weight lattice $P$ of the simple Lie algebra $\mathfrak g$. Then \cite[Conjecture 1.1]{CG} is true for $n\in Nr^\vee \mathbb{N}$.
\end{cor}

\subsubsection{Equivalences between affine \voas{} and $W$-algebras at admissible level}

There are also interesting equivalences of vertex tensor categories at admissible levels: let $\mathfrak g$ be a simple simply-laced Lie algebra,  $h^\vee$ the dual Coxeter number, and $k$ an admissible level of $\mathfrak g$. Let $P^m_+$ be the set of weights $\lambda$ such that the irreducible highest-weight representation $ \mathbb L_m(\lambda)$ is a module of the simple affine \voa $ \mathbb L_m(\mathfrak g)$ at positive integer level $m$. We parameterize $k=-h^\vee+\frac{u}{v}$ so that the simple ordinary modules of the simple affine \voa{} $ \mathbb L_k(\mathfrak g)$ of $\mathfrak g$ at level $k$ are the irreducible highest-weight modules $ \mathbb L_k(\lambda)$ of highest weight $\lambda$ at level $k$ with $\lambda \in P^{u-h^\vee}_+$. These modules are called ordinary and the category of ordinary modules at admissible level for simply-laced $\mathfrak g$ is semisimple \cite{Ara}, vertex tensor \cite{CHY}, and rigid \cite{C}. Let $\mathcal W_{k}(\mathfrak g)$ be the simple principal $W$-algebra of $\mathfrak g$ at level $k$. It is strongly rational if $k$ is non-degenerate admissible \cite{Ara2}, that is, $u, v \geq h^\vee$ in the simply-laced case. We denote the image of $ \mathbb L_k(\lambda)$ under quantum Hamiltonian reduction by $\mathcal W_k(\lambda)$. We denote the subcategory of ordinary $\mathbb{L}_{k}(\mathfrak{g})$-modules whose weights lie in the root lattice $Q$ by $\mathcal O^Q_{k, \text{ord}}(\mathfrak g)$, and we use $\mathcal C^Q_{k, \text{ord}}(\mathcal W(\mathfrak g))$ to denote the category of $\mathcal{W}_k(\mathfrak{g})$-modules whose objects are the images of modules in $\mathcal O^Q_{k, \text{ord}}(\mathfrak g)$ under quantum Hamiltonian reduction. That is, $\mathcal C^Q_{k, \text{ord}}(\mathcal{W}(\mathfrak g))$ is the semisimple category of $\mathcal{W}_k(\mathfrak{g})$-modules with simple objects $\mathcal{W}_k(\lambda)$ for $\lambda\in P^{u-h^\vee}_+\cap Q$.

Now let $\ell$ satisfy
\[
\frac{1}{k+1+ h^\vee} +\frac{1}{\ell+h^\vee} =1.
\]
With this notation, a special case of \cite[Main Theorem 3 (1)]{ACL} says that
\begin{align*}
 \mathbb L_k(\mathfrak g)\otimes  \mathbb L_1(\mathfrak g)\cong \bigoplus_{\lambda\in P^{u+v-h^{\vee}}_+\cap \, Q}
 \mathbb L_{k+1}(\lambda)\otimes \mathcal W_\ell(\lambda)
\end{align*}
as $ \mathbb L_{k+1}(\mathfrak g)\otimes \mathcal W_{\ell}(\mathfrak g)$-modules. Applying our Main Theorem \ref{mthm:voa} we have
\begin{cor}
Let $k$ be admissible and $\mathfrak g$ simply-laced, and define $\ell$ by
\[
\frac{1}{k+1+ h^\vee} +\frac{1}{\ell+h^\vee} =1.
\]
Then there is a braid-reversed equivalence between $\mathcal O^Q_{k+1, \mathrm{ord}}(\mathfrak g)$ and  $\mathcal C^Q_{\ell, \mathrm{ord}}(\mathcal W(\mathfrak g))$ sending $ \mathbb L_{k+1}(\lambda)$ to  $W_\ell(\lambda)^*$.
\end{cor}

\subsection{Outlook}

Our research program aims to understand representation categories of \voas using techniques 
of the theory of tensor categories. In this paper, we have proven that \voa extensions of 
suitable tensor products of two \voas{} are possible if and only if 
certain subcategories of modules of the two \voas{} are braid-reversed equivalent 
tensor categories. These results fall into the area of coset vertex algebras, 
since the commutant of a vertex subalgebra $\aV \subseteq \aA$ is called the coset $\aC$ of $\aV$ 
in $\aA$; see \cite[Section 5]{FZ0} for the original mathematical definition of commutant vertex subalgebras and \cite{GKO} for earlier work on coset models in conformal field theory. Often one is dealing with problems where one knows $\aA$ and $\aV$ fairly 
well and would like to study the coset \voa $\aC$. The very first statement we then need is the existence 
of vertex tensor category structure on suitable categories of $\aC$-modules. 
Such an existence result has recently been obtained in the related context of orbifold \voas \cite{McRae}, 
and we hope to extend these results to the more complicated setting of cosets. 

Ultimately, one of the deepest problems in the area is the conjecture that the coset \voa{} of a strongly rational \voa{} $\aA$ by a strongly rational vertex subalgebra $\aV$ is itself strongly rational. Theorem 7.6 of \cite{FFRS} is a good guide as it gives a relation between the category of modules for the coset \voa{} and the categories of $\aA$- and $\aV$-modules without assuming semisimplicity of the coset module category. It is however still proven under very strong assumptions, such as separability (traces of idempotents are non-zero). We now have techniques to prove statements avoiding assumptions like separability, and we hope to use them to come closer to the rationality conjecture for coset \voas.

A second application is that interesting affine vertex operator superalgebras can be realized as extensions of affine \voas and $W$-algebras, and then using our theory in \cite{CKM} one can study the representation theory of the superalgebras. This has for example been succesfully applied to $L_k(\mathfrak{osp}(1|2))$ at admissible level as an extension of $L_k(\mathfrak{sl}_2)$ times rational Virasoro algebras \cite{CFK, CKLiuR}. Thus another further goal is to extend the results of this work to superalgebras and construct many more interesting vertex superalgebras. With this in mind, we have already proved Key Lemma \ref{lemma:keyiso} and Proposition \ref{prop:UA_WA_ribbon} for (supercommutative) superalgebras.

\subsection*{Acknowledgements}
TC thanks Fedor Malikov for discussions on the algebra of chiral differential operators, and we all thank Yi-Zhi Huang for discussions and comments, and the referee for comments and suggestions. 
 TC is supported by NSERC $\#$RES0020460. SK is supported by a start-up grant provided by University of Denver.

\section{Algebras in braided tensor categories}
\label{sec:basic}

In this section, we review basic definitions and properties of (braided) tensor categories and (commutative, associative) algebra objects.

\subsection{Braided tensor categories}

Here we recall some definitions and structures in tensor 
categories.

\begin{defn}
Let $\cC$ be a category 
with a distinguished object $\one_{\cC}$ and bifunctor $\otimes:\cC\times\cC\rightarrow\cC$.
We say that $\cC$ is a \textbf{tensor category} if 
\begin{enumerate}
\item For any object $\mX$, there are natural isomorphisms $l_{\mX}:\one_{\cC}\otimes\mX \xrightarrow{\cong} \mX$ (left unit isomorphism) and 
$r_{\mX}:\mX\otimes \one_{\cC} \xrightarrow{\cong} \mX$ (right unit isomorphism),
\item For any triple of objects $\mX,\mY,\mZ$, there is a natural associativity isomorphism
$\cA_{\mX,\mY,\mZ}:\mX\otimes(\mY\otimes\mZ) \xrightarrow{\cong} (\mX\otimes\mY)\otimes\mZ$,
\item The isomorphisms $l$, $r$, $\cA$ satisfy the triangle axiom and the associativity $\cA$ satisfies the pentagon axiom.
\end{enumerate}
A tensor category $\cC$ is a \textbf{braided tensor category} if additionally:
\begin{enumerate}
	\item For all pairs of objects $\mX,\mY$, there is a natural braiding isomorphism $\cR_{\mX,\mY}:\mX\otimes\mY\rightarrow\mY\otimes\mX$
	\item The isomorphisms $\cR$ satisfy the hexagon axioms.
\end{enumerate}
A tensor category $\cC$ is \textbf{rigid} if
every object has a left and a right dual.
We shall only need notation for the left dual: for any object $\mX$ we denote the evaluation map by $e_{\mX}: \mX^\ast\otimes\mX\rightarrow\one$
and the coevaluation map by $i_{\mX}: \one\rightarrow \mX\otimes \mX^\ast$.

A rigid braided tensor category $\cC$ is \textbf{ribbon} if there is a natural isomorphism $\theta:\Id_\cC\rightarrow\Id_\cC$, called the \textbf{twist}, satisfying:
\begin{enumerate}
	\item $\theta_{\vac_\cC}=\Id_{\vac_\cC}$,
	\item $\theta_{\mX^\ast}=(\theta_{\mX})^\ast$, and
	\item The \textbf{balancing axiom}: $\theta_{\mX\otimes \mY}=\cR_{\mY,\mX}\circ\cR_{\mX,\mY}\circ(\theta_{\mX}\otimes\theta_{\mY})$.
\end{enumerate}
\end{defn}

We will sometimes need to consider tensor categories that are not or are not known to be rigid, since they may lack coevaluations. In these cases, however, some consequences of rigidity still hold when $\cC$ has a weaker structure (called weak rigidity in \cite{KL4}): we say that a contravariant functor $\cC\rightarrow\cC$, which we shall denote by $\mX\mapsto\mX'$, $f\mapsto f'$, is a \textbf{contragredient functor} if it permutes the simple objects of $\cC$ and there are natural isomorphisms
\begin{equation*}
 \Gamma_{\mX,\mY}: \HHom_{\cC}(\mX\otimes\mY,\vac)\rightarrow\HHom_{\cC}(\mX,\mY'),
\end{equation*}
natural in the sense that for morphisms $f: \mX_1\rightarrow\mX_2$ and $g:\mY_1\rightarrow\mY_2$ in $\cC$,
the diagram
\begin{equation}\label{Cont_Nat}
 \xymatrixcolsep{4pc}
 \xymatrix{
 \HHom_{\cC}(\mX_2\otimes\mY_2,\vac) \ar[r]^{\Gamma_{\mX_2,\mY_2}} \ar[d]^{F\mapsto F\circ(f\otimes g)} & \HHom_{\cC}(\mX_2,\mY_2') \ar[d]^{G\mapsto g'\circ G\circ f} \\
 \HHom_\cC(\mX_1\otimes\mY_1,\vac) \ar[r]^{\Gamma_{\mX_1,\mY_1}} & \HHom_{\cC}(\mX_1,\mY_1') \\
 }
\end{equation}
commutes. Given a contragredient functor, we shall denote the morphism $\Gamma_{\mX'\otimes\mX}^{-1}(\Id_{\mX'}): \mX'\otimes\mX\rightarrow\vac$ by $e_\mX$ and call it the evaluation for $\mX$. Note that if $\cC$ is rigid and braided, the duality functor $*$ is a contragredient functor with the natural isomorphisms $\Gamma_{\mX,\mY}$ obtained using $i_{\mY}$. An example of a non-rigid category with a contragredient functor is the category of all vector spaces over a field.

\begin{rema}\label{rema:deltaX}
 If $\cC$ is braided, we have a natural transformation $\psi_\mX: \mX\rightarrow(\mX')'$ given by
 \begin{equation*}
  \psi_{\mX} =\Gamma_{\mX,\mX'}(e_\mX\circ\cR_{\mX,\mX'}).
 \end{equation*}
If $\cC$ is rigid, then $\psi$ is a natural isomorphism (see for instance \cite[Section 2.2]{BK}), and it follows automatically that the contragredient functor permutes the simple objects of $\cC$. If $\cC$ is a ribbon category, then the natural isomorphisms $\delta_\mX: \mX\rightarrow\mX^{**}$ defined by
\begin{equation*}
 \delta_{\mX} =\Gamma_{\mX,\mX^*}(e_\mX\circ\cR_{\mX,\mX^*}\circ(\theta_\mX\otimes\Id_{\mX^*}))
\end{equation*}
have better properties (again see \cite[Section 2.2]{BK}).
\end{rema}

\begin{defn}
Let $\cC$ be a tensor category.
A triple $(\aA,\mu_\aA,\iota_\aA)$ with $\aA$ an object of $\cC$ and $\mu_\aA:\aA\otimes\aA\rightarrow \aA$, $\iota_\aA:\one_{\cC}\rightarrow \aA$ morphisms in $\cC$ is called an \textbf{associative algebra} if:
\begin{enumerate}
	\item Multiplication is associative: $\mu_\aA\circ (\Id_\aA \otimes \mu_\aA) = \mu_\aA\circ (\mu_\aA\otimes \Id_\aA)\circ \cA_{\aA,\aA,\aA}:\aA\otimes (\aA\otimes\aA) \rightarrow \aA$
	\item Multiplication is unital: $\mu_\aA\circ(\iota_\aA\otimes \Id_\aA) =l_\aA:\one_{\cC}\otimes\aA\rightarrow \aA$ and $\mu_\aA\circ(\Id_\aA\otimes\iota_\aA)= r_\aA: \aA\otimes\vac_\cC\rightarrow\aA$.
\end{enumerate}
If $\cC$ is braided, we say that $(\aA,\mu_\aA,\iota_\aA)$ is a \textbf{commutative algebra} if additionally:
\begin{enumerate}
	\item[(3)] Multiplication is commutative: $\mu_\aA\circ\cR_{\aA,\aA}=\mu_\aA:\aA\otimes\aA\rightarrow \aA$.
\end{enumerate}
We will sometimes drop the qualifiers ``associative'' and ``commutative'' when the context is clear.
\end{defn}

\begin{rema}
 In a commutative associative algebra, the right unit property $\mu_\aA\circ(\Id_\aA\otimes\iota_\aA)=r_\aA$ is a consequence of the left unit property and the commutativity.
\end{rema}

We shall need the definition of ``multiplication rules'' and the corresponding multiplication algebra:

\begin{defn}
	Let $(\aA,\mu_\aA,\iota_\aA)$ be an algebra in a tensor category $\cC$ and suppose $\aA$ is completely reducible in $\cC$.  For simple $\cC$-subobjects $\mX, \mY, \mZ$ of $\aA$ we define the {\bf multiplication rule} $M_{\mX, \mY}^{\mZ}$ to be
	\[
	M_{\mX, \mY}^{\mZ} := \begin{cases} 1 & \ \text{if} \ \mZ \subseteq \mathrm{Image}\left( \mu\vert_{\mX \otimes \mY} \rightarrow \aA\right),  \\ 0 & \ \text{otherwise} \end{cases}.
	\]
	and the unital {\bf multiplication algebra} of $\aA$ to be the free $\ZZ$-module with the set $B$ of inequivalent simple $\cC$-subobjects of $\aA$ as basis and product
	\[
	\mX \cdot \mY := \sum_{\mZ\in B} M_{\mX, \mY}^\mZ \mZ.
	\]
\end{defn}

\subsection{Direct sum completion}

We would like to work with algebras $\aA$ that are infinite direct sums of objects in $\cC$, and thus may not be objects of $\cC$ itself.
The most natural setting for this is the direct sum completion of the (ribbon) category $\cC$ as in \cite{AR}. 
The idea is to construct an extended category $\cC_{\oplus}$ whose objects are direct sums  $\bigoplus_{s\in S}\mX_s$ of objects in $\cC$, where $S$ is an arbitrary index set, and whose morphisms 
$f: \bigoplus_{s\in S}\mX_s \rightarrow \bigoplus_{t\in T}\mY_t$
are such that for any $s\in S$, $f|_{\mX_s}$ maps to $\bigoplus_{t\in T'}\mY_t$ for some \textit{finite} subset $T'\subseteq T$.

It was shown in \cite{AR} that if $\cC$ is a tensor category, possibly with additional structure such as braiding, then $\cC_{\oplus}$ can be naturally endowed with the same structures, essentially defining all structure isomorphisms ``componentwise.'' Moreover there is a braided monoidal functor $\cC\rightarrow \cC_{\oplus}$ which is fully faithful, that is, bijective on morphisms.
In effect, if we start with a braided tensor category $\cC$ with a balancing isomorphism, we can enlarge it to contain arbitrary direct sums.

There are three caveats: 
\begin{enumerate}
\item First, if $\cC$ is abelian, then $\cC_{\oplus}$ is not in general abelian. However, this complication does not arise if $\cC$ is semisimple, since in this case $\cC_{\oplus}$ is the category of arbitrary direct sums of simple objects in $\cC$, which is closed under subobjects and quotients (see for instance \cite[Section 3.5]{J}). More generally, one could work with the smallest category that contains $\cC$ and is closed under direct sums, kernels, and cokernels; however, we will not need this here.
\item  Second, we will be working with representation categories of vertex operator algebras, where morphism spaces $\HHom(\mX_1\otimes\mX_2,\mY)$ are naturally isomorphic to spaces of intertwining operators of type $\binom{\mY}{\mX_1\,\mX_2}$. We will need this same correspondence to hold in the direct sum completion.
\item Third, even if $\cC$ is rigid, one can not define a coevaluation on $\cC_\oplus$. However, for our purposes, we will only need a contragredient functor on a subcategory of $\cC_\oplus$, as we now explain.
\end{enumerate}

Assume that $\cC$ is a semisimple tensor category with a contragredient functor. Let $\cC_\oplus^{fin}$ denote the full subcategory of $\cC_\oplus$ whose objects contain any simple object in $\cC$ with at most finite multiplicity. Note that $\cC_\oplus^{fin}$ is not a tensor subcategory of $\cC_\oplus$ unless $\cC$ has finitely many equivalence classes of simple objects (in which case $\cC_\oplus^{fin}=\cC$). However, $\cC_\oplus^{fin}$ admits a contragredient functor in a suitable sense. If $\mX=\bigoplus_{s\in S} \mX_s$ is an object of $\cC_\oplus^{fin}$, then so is $\mX'=\bigoplus_{s\in S} \mX_s'$ because the contragredient functor permutes the simple objects of $\cC$. Moreover, if $F: \mX=\bigoplus_{s\in S}\mX_s\rightarrow\mY=\bigoplus_{t\in T} \mY_t$ is a morphism in $\cC_\oplus^{fin}$, we can define $F': \mY'\rightarrow\mX'$ as follows: For any $s\in S$, $t\in T$, let $F_{s,t}: \mX_s\rightarrow\mY_t$ denote the projection onto $\mY_t$ of the restriction of $F$ to $\mX_s$. Then we define $F'$ by
\begin{equation*}
 F'\vert_{\mY_t'}=\sum_{s\in S} F_{s,t}'.
\end{equation*}
To see why this sum is well defined, note that $\mY_t$, as an object of $\cC$, is the direct sum of finitely many simple objects of $\cC$ with finite multiplicity. Then since $\mX$ is an object of $\cC_\oplus^{fin}$, these finitely many simple objects can occur in only finitely many $\mX_s$, and so $F_{s,t}=0$ for all but finitely many $s\in S$.
\begin{prop}\label{Cont_in_Cplus}
 If $\cC$ is a semisimple tensor category with a contragredient functor, then there are natural isomorphisms
 \begin{equation*}
  \Gamma_{\mX,\mY}: \HHom_{\cC_\oplus}(\mX\otimes\mY,\vac)\rightarrow\HHom_{\cC_\oplus}(\mX,\mY')
 \end{equation*}
for $\mX$ any object of $\cC_\oplus$ and $\mY$ an object of $\cC_\oplus^{fin}$.
\end{prop}
\begin{proof}
 Suppose $\mX=\bigoplus_{s\in S} \mX_s$ and $\mY=\bigoplus_{t\in T} \mY_t$. Then because
 \begin{equation*}
  \mX\otimes\mY =\bigoplus_{(s,t)\in S\times T} \mX_s\otimes\mY_t,
 \end{equation*}
we have natural isomorphisms
\begin{equation*}
 \HHom_{\cC_\oplus}(\mX\otimes\mY,\vac)\cong\prod_{(s,t)\in S\times T} \HHom_{\cC}(\mX_s\otimes\mY_t,\vac)\cong\prod_{(s,t)\in S\times T} \HHom_\cC(\mX_s,\mY_t').
\end{equation*}
Now, for any tuple $\lbrace F_{s,t}\rbrace_{s\in S,t\in T}\in \prod_{(s,t)\in S\times T} \HHom_\cC(\mX_s,\mY_t')$ and $s\in S$, we must have $F_{s,t}=0$ for all but finitely many $t\in T$ since the finitely many simple objects occurring in $\mX_s$ can occur in only finitely many $\mY_t'$, given that $\mY$ is an object of $\cC_\oplus^{fin}$. This shows that in fact
\begin{equation*}
 \prod_{(s,t)\in S\times T} \HHom_\cC(\mX_s,\mY_t') \cong\HHom_{\cC_\oplus} (\mX,\mY'),
\end{equation*}
as required.
\end{proof}

The details of the definitions and structures in $\cC_\oplus$ are
gathered in Appendix \ref{app:dirsum}. Since most arguments in the following sections do not change when $\cC$ is a semisimple ribbon category and $\aA$ is an algebra in $\cC_\oplus$ rather than in $\cC$, we shall frequently omit references to $\cC_\oplus$.

\subsection{Representation categories of an algebra object}
\label{subsec:repA}

Now we define representations of an algebra in a tensor category and recall some important theorems from \cite{KO}, \cite{HKL} and \cite{CKM}. From now on, we will assume that the tensoring functors $\mX\otimes\bullet$ and $\bullet\otimes\mX$ for an object $\mX$ in a tensor category $\cC$ are right exact (so that in particular, these functors preserve surjections). This is needed to guarantee that the category of representations of an algebra object, defined below, has a tensor product bifunctor.

\begin{defn}
Suppose that $(\aA,\mu_\aA,\iota_\aA)$ is an associative algebra in $\cC$.
Define $\repA$ to be the category of pairs $(\mX,\mu_\mX)$ where $\mX\in\obj(\cC)$
and $\mu_{\mX}\in\HHom_{\cC}(\aA\otimes \mX,\mX)$ satisfy the following:
\begin{enumerate}
	\item Unit property: $l_{\mX}=\mu_\mX\circ(\iota_\aA\otimes\Id_{\mX}): \vac_{\cC}\otimes\mX\rightarrow \mX$,
	\item Associativity: $\mu_\mX\circ (\Id_{\aA}\otimes\mu_\mX) = \mu_\mX\circ(\mu_\aA\otimes\Id_{\mX}) \circ\cA_{\aA,\aA,\mX}: \aA\otimes(\aA\otimes\mX)\rightarrow \mX$.
\end{enumerate}
A morphism $f\in\HHom_{\repA}((\mX_1,\mu_{\mX_1}) , (\mX_2,\mu_{\mX_2}))$
is a morphism $f\in\HHom_{\cC}(\mX_1,\mX_2)$ such that $\mu_{\mX_2}\circ(\Id_{\aA}\otimes f)=f\circ\mu_{\mX_1}$.

When $\aA$ is commutative, we define $\rep^0\aA$ to be the full subcategory of $\repA$ containing ``dyslectic'' objects: those $(\mX,\mu_{\mX})$ such that $\mu_{\mX}\circ\cR_{\mX,\aA}\circ \cR_{\aA,\mX}=\mu_\mX$.
\end{defn}

	Note that $\repA$ is the category of left $\aA$-modules. 
	One may define the category of right $\aA$-modules analogously.
	It is easy to show that if $(\mX,\mu_\mX)$ is in $\repA$, then $(\mX,\mu_{\mX}\circ\cR_{\mX,\aA})$ and $(\mX,\mu_{\mX}\circ\cR^{-1}_{\aA,\mX})$ are right modules for the opposite algebras $(\aA,\mu_\aA\circ\cR_{\aA,\aA},\iota_\aA)$ and $(\aA,\mu_{\aA}\circ\cR_{\aA,\aA}^{-1}, \iota_\aA)$, respectively. Note that when $\aA$ is commutative, both opposite algebras are equal to $\aA$ and the dyslectic modules $\mX$ (objects of $\rep^0\aA$) are precisely those for which the two right $\aA$-module structures on $\mX$ coincide.

Clearly, $(\aA,\mu_\aA)$ is both a left and a right $\aA$-module, and an object of $\rep^0\aA$, when $\aA$ is commutative.

We have an induction functor
\begin{equation*}
\cF: \cC\rightarrow\repA
\end{equation*}
given on objects by $\cF(\Mod{W})=\aA\otimes\Mod{W}$ for objects $\Mod{W}$ in $\cC$ and on 
morphisms by $\cF(f)=\Id_\aA\otimes  f$ for $f: \Mod{W}_1\rightarrow \Mod{W}_2$ in $\cC$. Note that if $\aA$ is an algebra in  $\cC_{\oplus}$, we will still take $\cC$ as the domain of our induction functor to $\repA$. Critically, the induction functor $\cF$ satisfies \textbf{Frobenius reciprocity} \cite{KO,CKM}, that is, it is left adjoint to the forgetful functor $\cG$ from $\repA$ to $\cC$: there is a natural isomorphism
\begin{align}
\Hom{\repA}{\cF(\Mod{W}), \mX} \cong \Hom{\cC}{\Mod{W}, \cG(\mX)}
\label{eqn:frobenius}
\end{align}
for objects $\Mod{W}$ in $\cC$ and $\mX$ in $\repA$. Under this isomorphism, $f\in \Hom{\repA}{\cF(\Mod{W}), \mX}$ maps to the composition
	\begin{align*}
	\Mod{W} \rightarrow \one \otimes \Mod{W} 
	\xrightarrow{\iota_{\aA}\otimes \Id_{\Mod{W} }} \aA\otimes \Mod{W} 
	\xrightarrow{f} \mX, 
	\end{align*}
and $g\in \Hom{\cC}{\Mod{W}, \cG(\mX)}$ maps to the composition
\begin{align*}
\aA\otimes\Mod{W} \xrightarrow{\Id_{\aA} \otimes  g} \aA\otimes \mX\xrightarrow{\mu_{\mX}}\mX.
\end{align*}

When $\aA$ is commutative, the category $\repA$ is a tensor category with tensor product $\otimes_{\aA}$ and unit object $\aA$, and the subcategory
$\rep^0 \aA$ is a braided tensor category (see for example \cite{KO, CKM}). Since $\otimes_\aA$ is defined as the cokernel of a certain morphism in $\cC$, if $\aA$ is an algebra in $\cC_\oplus$, we assume that $\cC$ is semisimple to guarantee $\cC_{\oplus}$ is abelian. Crucially, the induction functor is \textbf{monoidal} \cite{KO,CKM}, that is, there are natural $\repA$-isomorphisms
\begin{align}
f_{\Mod{W}_1,\Mod{W}_2}:\cF(\Mod{W}_1\otimes \Mod{W}_2)\xrightarrow{\sim} \cF(\Mod{W}_1)\otimes_{\aA}\cF(\Mod{W}_2) \label{eqn:Fmon}
\end{align}
which together with the $\repA$-isomorphism $r_\aA: \cF(\vac)\xrightarrow{\sim} \aA$ are suitably compatible with the unit and associativity isomorphisms in $\cC$ and $\repA$.

The following lemma in the case that $\cC$ is rigid amounts to part of \cite[Theorem 1.15]{KO}, but here we assume only that $\cC$ has a contragredient functor because we will need to apply the result in $\cC_\oplus$.
\begin{lemma}
If $\cC$ has a contragredient functor and $(\mX,\mu_\mX)$ is an object of $\repA$, then $(\mX',\mu_{\mX'})$ is a right $\aA$-module, where $\mu_{\mX'}: \mX'\otimes\aA\rightarrow\mX'$ is given by
\begin{equation*}
 \mu_{\mX'}=\Gamma_{\mX'\otimes\aA,\mX}\left(e_{\mX}\circ(\Id_{\mX'}\otimes\mu_\mX)\circ\cA_{\mX',\aA,\mX}^{-1}\right).
\end{equation*}
\end{lemma}
\begin{proof}
 To prove the unit property of $\mu_{\mX'}$, we use the commutative diagram
 \begin{equation*}
  \xymatrixcolsep{4pc}
 \xymatrix{
 \HHom_{\cC}((\mX'\otimes\aA)\otimes\mX,\vac)  \ar[d]_{F\mapsto F\circ((\Id_{\mX}'\otimes\iota_\aA)r_{\mX'}^{-1}\otimes\Id_\mX)} & \HHom_{\cC}(\mX'\otimes\aA,\mX') \ar[l]_(.45){\Gamma^{-1}_{\mX'\otimes\aA,\mX}} \ar[d]^{G\mapsto G\circ(\Id_{\mX'}\otimes\iota_\aA)\circ r_{\mX'}^{-1}} \\
 \HHom_\cC(\mX'\otimes\mX,\vac)  & \HHom_{\cC}(\mX',\mX') \ar[l]_(.47){\Gamma^{-1}_{\mX',\mX}} \\
 }
 \end{equation*}
given by the naturality of $\Gamma$. Applying both compositions to $\mu_{\mX'}$, and then using the definition of $\mu_{\mX'}$, the naturality of the associativity isomorphisms, the triangle axiom, and the unit property for $\mX$, we get
\begin{align*}
 \Gamma^{-1}_{\mX',\mX}(\mu_{\mX'}\circ( & \Id_{\mX'}\otimes\iota_\aA)\circ r_{\mX'}^{-1})  = \Gamma^{-1}_{\mX'\otimes\aA,\mX}(\mu_{\mX'})\circ((\Id_{\mX'}\otimes\iota_\aA)\otimes\Id_\mX)\circ(r_{\mX'}^{-1}\otimes\Id_\mX)\nonumber\\
 & = e_\mX\circ(\Id_{\mX'}\otimes\mu_\mX)\circ\cA^{-1}_{\mX',\aA,\mX}\circ((\Id_{\mX'}\otimes\iota_\aA)\otimes\Id_\mX)\circ(r_{\mX'}^{-1}\otimes\Id_\mX)\nonumber\\
 & =e_\mX\circ(\Id_{\mX'}\otimes\mu_\mX)\circ(\Id_{\mX'}\otimes(\iota_\aA\otimes\Id_\mX))\circ\cA_{\mX',\vac,\mX}^{-1}\circ(r_{\mX'}^{-1}\otimes\Id_\mX)\nonumber\\
 & =e_\mX\circ(\Id_{\mX'}\otimes\mu_\mX)\circ(\Id_{\mX'}\otimes(\iota_\aA\otimes\Id_\mX))\circ(\Id_{\mX'}\otimes l_\mX^{-1})\nonumber\\
 & = e_\mX.
\end{align*}
We conclude that
\begin{equation*}
 \mu_{\mX'}\circ(\Id_{\mX'}\otimes\iota_\aA)\circ r_{\mX'}^{-1} =\Gamma_{\mX',\mX}(e_{\mX})=\Id_{\mX'}
\end{equation*}
as required.

To prove the associativity of $\mu_{\mX'}$, we first use the commutative diagram
\begin{equation*}
 \xymatrixcolsep{4pc}
 \xymatrix{
 \HHom_\cC(\mX'\otimes\aA,\mX') \ar[r]^(.45){\Gamma^{-1}_{\mX'\otimes\aA,\mX}} \ar[d]_{G\mapsto G\circ(\mu_{\mX'}\otimes\Id_\aA)} & \HHom_\cC((\mX'\otimes\aA)\otimes\mX,\vac) \ar[d]^{F\mapsto F\circ((\mu_{\mX'}\otimes\Id_\aA)\otimes\Id_\mX)} \\
 \HHom_{\cC}((\mX'\otimes\aA)\otimes\aA, \mX') \ar[r]^(.47){\Gamma^{-1}_{(\mX'\otimes\aA)\otimes\aA,\mX}} & \HHom_\cC(((\mX'\otimes\aA)\otimes\aA)\otimes\mX,\vac) \\
 }
\end{equation*}
given by the naturality of $\Gamma$. Applying both compositions to $\mu_{\mX'}$ and then using the definition of $\mu_{\mX'}$ and the naturality of the associativity isomorphisms, we get
\begin{align}\label{Wprime_assoc_calc}
 \Gamma^{-1}_{(\mX'\otimes\aA)\otimes\aA,\mX}(  \mu_{\mX'} & \circ(\mu_{\mX'}\otimes\Id_\aA)) = \Gamma_{\mX'\otimes\aA,\mX}^{-1}(\mu_{\mX'})\circ((\mu_{\mX'}\otimes\Id_\aA)\otimes\Id_\mX)\nonumber\\
 & = e_\mX\circ(\Id_{\mX'}\otimes\mu_\mX)\circ\cA^{-1}_{\mX',\aA,\mX}\circ((\mu_{\mX'}\otimes\Id_\aA)\otimes\Id_\mX)\nonumber\\
 & = e_\mX\circ(\mu_{\mX'}\otimes\Id_\mX)\circ(\Id_{\mX'\otimes\aA}\otimes\mu_\mX)\circ\cA_{\mX'\otimes\aA,\aA,\mX}^{-1}.
\end{align}
Now, the naturality of $\Gamma$ implies
\begin{align*}
 e_\mX\circ(\mu_{\mX'}\otimes\Id_\mX) & =\Gamma_{\mX',\mX}^{-1}(\Id_{\mX'})\circ(\mu_{\mX'}\otimes\Id_\mX)\nonumber\\
 &= \Gamma^{-1}_{\mX'\otimes\aA,\mX}(\Id_{\mX'}\circ\mu_{\mX'}) = e_\mX\circ(\Id_{\mX'}\otimes\mu_\mX)\circ\cA_{\mX',\aA,\mX}^{-1}.
\end{align*}
Putting this back into \eqref{Wprime_assoc_calc} and using the naturality of the associativity isomorphisms, the associativity of $\mu_{\mX}$ the pentagon axiom, the definition of $\mu_{\mX'}$, and the naturality of $\Gamma$, we get
\begin{align*}
 & \Gamma^{-1}_{(\mX'\otimes\aA)\otimes\aA,\mX} (  \mu_{\mX'}  \circ(\mu_{\mX'}\otimes\Id_\aA))\nonumber\\
  & \hspace{2em}= e_\mX\circ(\Id_{\mX'}\otimes\mu_\mX)\circ\cA_{\mX',\aA,\mX}^{-1}\circ(\Id_{\mX'\otimes\aA}\otimes\mu_\mX)\circ\cA_{\mX'\otimes\aA,\aA,\mX}^{-1}\nonumber\\
  &\hspace{2em} =e_\mX\circ(\Id_{\mX'}\otimes\mu_\mX)\circ(\Id_{\mX'}\otimes(\Id_\aA\otimes\mu_\mX))\circ\cA^{-1}_{\mX',\aA,\aA\otimes\mX}\circ\cA^{-1}_{\mX'\otimes\aA,\aA,\mX}\nonumber\\
  &\hspace{2em} =e_\mX\circ(\Id_{\mX'}\otimes\mu_\mX)\circ(\Id_{\mX'}\otimes(\mu_\aA\otimes\Id_\mX))\circ(\Id_{\mX'}\otimes\cA_{\aA,\aA,\mX})\circ\cA^{-1}_{\mX',\aA,\aA\otimes\mX}\circ\cA^{-1}_{\mX'\otimes\aA,\aA,\mX}\nonumber\\
  &\hspace{2em} =e_\mX\circ(\Id_{\mX'}\otimes\mu_\mX)\circ(\Id_{\mX'}\otimes(\mu_\aA\otimes\Id_\mX))\circ\cA_{\mX',\aA\otimes\aA,\mX}^{-1}\circ(\cA_{\mX',\aA,\aA}^{-1}\otimes\Id_\mX)\nonumber\\
  &\hspace{2em} =e_\mX\circ(\Id_{\mX'}\otimes\mu_\mX)\circ\cA_{\mX',\aA,\mX}^{-1}\circ((\Id_{\mX'}\otimes\mu_\aA)\otimes\Id_\mX)\circ(\cA_{\mX',\aA,\aA}^{-1}\otimes\Id_\mX)\nonumber\\
  &\hspace{2em} =\Gamma^{-1}_{\mX'\otimes\aA,\mX}(\mu_{\mX'})\circ((\Id_{\mX'}\otimes\mu_\aA)\otimes\Id_\mX)\circ(\cA_{\mX',\aA,\aA}^{-1}\otimes\Id_\mX)\nonumber\\
  &\hspace{2em} = \Gamma_{(\mX'\otimes\aA)\otimes\aA,\mX}^{-1}(\mu_{\mX'}\circ(\Id_{\mX'}\otimes\mu_\aA)\circ\cA_{\mX',\aA,\aA}^{-1}).
\end{align*}
Applying $\Gamma_{(\mX'\otimes\aA)\otimes\aA,\mX}$ to both sides, we conclude
\begin{equation*}
 \mu_{\mX'}  \circ(\mu_{\mX'}\otimes\Id_\aA) =\mu_{\mX'}\circ(\Id_{\mX'}\otimes\mu_\aA)\circ\cA_{\mW',\aA,\aA}^{-1},
\end{equation*}
which is the associativity of $\mu_{\mX'}$.
\end{proof}

Although in this paper we are mostly concerned with commutative algebras, we will prove some results for superalgebras, which are associative algebras that in particular satisfy $\mu_{\aA}\circ\cR_{\aA,\aA}^2 =\mu_\aA$. For such an associative algebra, we have a single opposite algebra $\aA^{\text{op}}=(\aA,\mu_\aA\circ\cR_{\aA,\aA},\iota_\aA)$. Thus in light of the preceding lemma and Proposition \ref{Cont_in_Cplus}, we immediately have:
\begin{cor}\label{cor:AdualinRepA}
Assume that $\cC$ is a braided tensor category with contragredient functor and that either:
\begin{enumerate}
 \item $\aA$ is an associative algebra in $\cC$, or
 \item $\cC$ is semisimple and $\aA$ is an associative algebra in $\cC_\oplus^{fin}$.
\end{enumerate}
If $\mu_\aA\circ\cR_{\aA,\aA}^2=\mu_\aA$, then $(\aA',\mu_{\aA'}^+=\mu_{\aA'}\circ \cR_{\aA,\aA'})$ and 
$(\aA',\mu_{\aA'}^-=\mu_{\aA'}\circ \cR^{-1}_{\aA',\aA})$
are objects of $\rep\,\aA^{\mathrm{op}}$. 
\end{cor}

Now the following lemma in the case that $\aA$ is an algebra in $\cC$ is essentially a special case of \cite[Theorem 1.17.3]{KO}, but again we will need the result for algebras in $\cC_\oplus^{fin}$:
\begin{lemma}
 Assume that $\cC$ is a rigid braided tensor category and that either:
 \begin{enumerate}
  \item $\aA$ is an associative algebra in $\cC$, or
  \item $\cC$ is semisimple and $\aA$ is an associative algebra in $\cC_\oplus^{fin}$.
 \end{enumerate}
If $\mu_\aA\circ\cR_{\aA,\aA}^2=\mu_{\aA}$, then the morphism $\psi_{\aA}$ of Remark \ref{rema:deltaX} is an isomorphism in $\repA$ from $(\aA,\mu_\aA)$ to $(\aA^{**},\mu_{\aA^{**}}\circ\cR_{\aA,\aA^{**}})$, where the right $\aA^{\mathrm{op}}$-module structure $\mu_{\aA^{**}}$ is obtained using the left $\aA^{\mathrm{op}}$-module structure $\mu_{\aA^*}^+$ on $\aA^*$.
\end{lemma}
\begin{proof}
First observe that $(\aA^{**},\mu_{\aA^{**}}\circ\cR_{\aA,\aA^{**}})$ is indeed an object of $\repA$ because our assumption $\mu_\aA\circ\cR_{\aA,\aA}^2=\mu_{\aA}$ implies that $(\aA^{\mathrm{op}})^{\mathrm{op}}=\aA$.

 If $\aA$ is an algebra in $\cC$, then $\psi_\aA$ is an isomorphism in $\cC$ because $\cC$ is rigid. If on the other hand $\aA=\bigoplus_{s\in S} \aA_s$ is an algebra in $\cC_\oplus^{fin}$, then using the definition
 \begin{equation*}
  \psi_{\aA}=\Gamma_{\aA,\aA^*}(e_{\aA}\circ\cR_{\aA,\aA^*}),
 \end{equation*}
where $\Gamma_{\aA,\aA^*}$ is the natural isomorphism of Proposition \ref{Cont_in_Cplus}, together with the natural isomorphism
\begin{equation*}
 \HHom_{\cC_\oplus}(\aA,\aA^{**})\cong\prod_{(s,t)\in S\times S} \HHom_\cC(\aA_s,\aA_t^{**}),
\end{equation*}
we can identify $\psi_\aA$ with the tuple
\begin{equation*}
 \lbrace \delta_{s,t}\psi_{\aA_s}\rbrace_{(s,t)\in S\times S}\in\prod_{(s,t)\in S\times S} \HHom_\cC(\aA_s,\aA_t^{**}).
\end{equation*}
Now since $\cC$ is rigid, each $\psi_{\aA_s}$ is an isomorphism in $\cC$, and so $\psi_\aA$ is an isomorphism in $\cC_{\oplus}$.

It remains to show that $\psi_{\aA}$ is actually a morphism in $\repA$. We need to show that
\begin{equation*}
 \mu_{\aA^{**}}\circ\cR_{\aA,\aA^{**}}\circ(\Id_\aA\otimes\psi_\aA)=\psi_\aA\circ\mu_\aA.
\end{equation*}
By the naturality of the braiding isomorphisms and the assumption $\mu_\aA\circ\cR_{\aA,\aA}^2=\mu_{\aA}$, this is equivalent to showing
\begin{equation*}
 \mu_{\aA^{**}}\circ(\psi_\aA\otimes \Id_\aA)=\psi_\aA\circ\mu_\aA\circ\cR_{\aA,\aA}
\end{equation*}
(that is, $\psi_\aA$ is a homomorphism of right $\aA^{\mathrm{op}}$-modules). For this, we first note a general fact about contragredient functors: if $f: \Mod{W}\otimes\mX\rightarrow\vac$ is a morphism in $\cC$, then $\Gamma_{\Mod{W},\mX}(f)$ is the unique morphism in $\HHom_\cC(\Mod{W},\mX')$ such that $e_\mX\circ(\Gamma_{\Mod{W},\mX}(f)\otimes\Id_\mX)=f$. This follows from applying both compositions in the commutative diagram
\begin{equation*}
 \xymatrixcolsep{4pc}
 \xymatrix{
 \HHom_\cC(\mX'\otimes\mX,\vac) \ar[r]^{\Gamma_{\mX',\mX}} \ar[d]_{F\mapsto F\circ(\Gamma(f)\otimes\Id_{\mX})} & \HHom_\cC(\mX',\mX') \ar[d]^{G\mapsto G\circ\Gamma(f)} \\
 \HHom_\cC(\Mod{W}\otimes\mX, \vac) \ar[r]^{\Gamma_{\Mod{W},\mX}} & \HHom_\cC(\Mod{W},\mX') \\ 
 }
\end{equation*}
to $e_{\mX}=\Gamma_{\mX',\mX}^{-1}(\Id_{\mX'})$. In the cases $\Gamma_{\Mod{W},\mX}(f)=\mu_{\aA^{**}}$ and $\Gamma_{\Mod{W},\mX}(f)=\psi_\aA$, we see that $\mu_{\aA^{**}}: \aA^{**}\otimes\aA\rightarrow\aA^{**}$ is the unique morphism such that
\begin{equation*}
 e_{\aA^*}\circ(\mu_{\aA^{**}}\otimes\Id_{\aA^*})=e_{\aA^*}\circ(\Id_{\aA^{**}}\otimes\mu_{\aA^*}^+)\circ\cA^{-1}_{\aA^{**},\aA,\aA^*}
\end{equation*}
and $\psi_\aA:\aA\rightarrow\aA^{**}$ is the unique morphism such that
\begin{equation*}
 e_{\aA^*}\circ(\psi_\aA\otimes\Id_{\aA^*})=e_\aA\circ\cR_{\aA,\aA^*}.
\end{equation*}
Using these relations together with the naturality of the associativity isomorphisms, we obtain the commutative diagram
\begin{equation*}
 \xymatrixcolsep{4pc}
 \xymatrix{
 (\aA\otimes\aA)\otimes\aA^* \ar[r]^{\cA_{\aA,\aA,\aA^*}^{-1}} \ar[d]^{(\psi_\aA\otimes\Id_\aA)\otimes\Id_{\aA^*}} & \aA\otimes(\aA\otimes\aA^*) \ar[r]^{\Id_\aA\otimes\cR_{\aA,\aA^*}} \ar[d]^{\psi_\aA\otimes\Id_{\aA\otimes\aA^*}} & \aA\otimes(\aA^*\otimes\aA) \ar[d]^{\psi_\aA\otimes\Id_{\aA^*\otimes\aA}} \ar[rd]^{\Id_\aA\otimes\mu_{\aA^*}} & \\
 (\aA^{**}\otimes\aA)\otimes\aA^* \ar[r]^{\cA^{-1}_{\aA^{**}\otimes\aA\otimes\aA^*}} \ar[rd]^{\mu_{\aA^{**}}\otimes\Id_{\aA^*}} & \aA^{**}\otimes(\aA\otimes\aA^*) \ar[r]^{\Id_{\aA^{**}}\otimes\cR_{\aA,\aA^*}} \ar[rd]^{\Id_{\aA^{**}}\otimes\mu_{\aA^*}^+} & \aA^{**}\otimes(\aA^*\otimes\aA) \ar[d]^{\Id_{\aA^{**}}\otimes\mu_{\aA^*}} & \aA\otimes \aA^* \ar[ld]^{\psi_\aA\otimes\Id_{\aA^*}} \ar[d]^{\cR_{\aA,\aA^*}} \\
& \aA^{**}\otimes\aA^* \ar[rd]^{e_{\aA^*}}  & \aA^{**}\otimes\aA^* \ar[d]^{e_{\aA^*}} & \aA^*\otimes\aA \ar[ld]^{e_\aA} \\
&  & \vac &  \\
 }
\end{equation*}

Now we can use the naturality of the braiding isomorphisms and the hexagon axioms to rewrite the outer composition on the top and right sides of the diagram as
\begin{align*}
 (\aA\otimes\aA)\otimes\aA^* & \xrightarrow{\cA_{\aA,\aA,\aA^*}^{-1}} \aA\otimes(\aA\otimes\aA^*)\xrightarrow{\Id_\aA\otimes\cR_{\aA,\aA^*}} \aA\otimes(\aA^*\otimes\aA)\xrightarrow{\cA_{\aA,\aA^*,\aA}} (\aA\otimes\aA^*)\otimes\aA\nonumber\\
 & \xrightarrow{\cR_{\aA,\aA^*}\otimes\Id_\aA} (\aA^*\otimes\aA)\otimes\aA\xrightarrow{\cA^{-1}_{\aA^*,\aA,\aA}} \aA^*\otimes(\aA\otimes\aA) \xrightarrow{\Id_{\aA^*}\otimes\cR_{\aA,\aA}} \aA^*\otimes(\aA\otimes\aA)\nonumber\\
 & \xrightarrow{\cA_{\aA^*,\aA,\aA}} (\aA^*\otimes\aA)\otimes\aA \xrightarrow{\mu_{\aA^*}\otimes\Id_\aA} \aA^*\otimes\aA\xrightarrow{e_\aA} \vac.
\end{align*}
Next we use the hexagon axioms to rewrite the first five arrows and use the relation
\begin{equation*}
 e_\aA\circ(\mu_{\aA^*}\otimes\Id_\aA) = e_\aA\circ(\Id_{\aA^*}\otimes\mu_\aA)\circ\cA_{\aA^*,\aA,\aA}^{-1}
\end{equation*}
to obtain
\begin{align*}
 (\aA\otimes\aA)\otimes\aA^*\xrightarrow{\cR_{\aA\otimes\aA,\aA^*}} \aA^*\otimes(\aA\otimes\aA)\xrightarrow{\Id_{\aA^*}\otimes\cR_{\aA,\aA}} \aA^*\otimes(\aA\otimes\aA)\xrightarrow{\Id_{\aA^*}\otimes\mu_\aA} \aA^*\otimes\aA\xrightarrow{e_\aA} \vac.
\end{align*}
The naturality of the braiding now implies this is
\begin{equation*}
 (\aA\otimes\aA)\otimes\aA^*\xrightarrow{\cR_{\aA,\aA}\otimes\Id_{A^*}}(\aA\otimes\aA)\otimes\aA^*\xrightarrow{\mu_\aA\otimes\Id_{\aA^*}} \aA\otimes\aA^*\xrightarrow{\cR_{\aA,\aA^*}} \aA^*\otimes\aA\xrightarrow{e_\aA}\vac,
\end{equation*}
which is $\Gamma_{\aA,\aA^*}^{-1}(\psi_\aA)\circ((\mu_\aA\circ\cR_{\aA,\aA})\otimes\Id_{\aA^*})$.

Our calculations have now shown that
\begin{align*}
 e_{\aA^*}\circ(\mu_{\aA^{**}}\circ(\psi_\aA\otimes\Id_\aA)\otimes\Id_{\aA^*}) & = \Gamma_{\aA,\aA^*}^{-1}(\psi_\aA)\circ((\mu_\aA\circ\cR_{\aA,\aA})\otimes\Id_{\aA^*})\nonumber\\
 &= \Gamma^{-1}_{\aA\otimes\aA,\aA^*}(\psi_\aA\circ\mu_\aA\circ\cR_{\aA,\aA}), 
\end{align*}
where we have used the naturality of $\Gamma$ for the second equality. Applying $\Gamma_{\aA\otimes\aA,\aA^*}$ to both sides then yields the desired equality $\mu_{\aA^{**}}\circ(\psi_{\aA}\otimes\Id_\aA)=\psi_\aA\circ\mu_\aA\circ\cR_{\aA,\aA}$.
\end{proof}

The main reason we need the preceding lemma is the following corollary:
\begin{cor}\label{Adualsimple}
  Assume that $\cC$ is a rigid braided tensor category and that either:
 \begin{enumerate}
  \item $\aA$ is an associative algebra in $\cC$, or
  \item $\cC$ is semisimple and $\aA$ is an associative algebra in $\cC_\oplus^{fin}$.
 \end{enumerate}
 If $\mu_{\aA}\circ\cR_{\aA,\aA}^2=\mu_\aA$ and $\aA$ is simple as an object of $\repA$, then $(\aA^*,\mu_{\aA^*})$ is a simple right $\aA$-module.
\end{cor}
\begin{proof}
 We need to show that any right $\aA$-module inclusion (equivalently, $\rep\,\aA^{\mathrm{op}}$-inclusion) $i: \mX\rightarrow\aA^*$ is either $0$ or an isomorphism. In fact, the cokernel $c: \aA^*\rightarrow\mathrm{coker}\,i$ is also a morphism in $\rep\,\aA^{\mathrm{op}}$ (see for instance \cite[Lemma 1.4]{KO} or \cite[Theorem 2.9]{CKM}). It is straightforward to show that the dual
 \begin{equation*}
  c^*: (\mathrm{coker}\,i)^*\rightarrow\aA^{**}
 \end{equation*}
is then a right $\aA^{\mathrm{op}}$-module homomorphism (equivalently, a morphism in $\repA$), and it is injective. But $\aA^{**}$ is simple in $\repA$ since it is isomorphic to $\aA$. Therefore $c^*$ is either $0$ or an isomorphism, and the same then holds for $c$ and $i$.
\end{proof}

\subsection{The center of a tensor category}\label{sec:center}

The center $\cZ(\cC)$ of a tensor category $\cC$ is an important construction we will use for studying the commutativity of algebras in $\cC$. 
\begin{defn}
Let $\cC$ be a tensor category. The \textbf{center} $\cZ(\cC)$ is the category whose objects are pairs $(\mX, \gamma^\mX)$ where $\mX\in\obj(\cC)$ and $\gamma^\mX= \{ \gamma^\mX_\mM: \mM \otimes \mX \rightarrow \mX \otimes \mM\ | \ \mM \ \in \ \cC\}$ is a family of isomorphisms in $\cC$, called a \textbf{half-braiding},
that are natural in the sense that
\begin{align*}
\xymatrix{
\mM \otimes \mX  \ar[rr]^{\gamma^\mX_\mM}\ar[d]_{f\otimes \Id_\mX} && \mX \otimes \mM \ar[d]^{\Id_\mX \otimes f} \\
\mN\otimes \mX \ar[rr]^{ \gamma^\mX_\mN}  && \mX \otimes \mN.
}
\end{align*}
commutes for all $f$ in $\text{Hom}_\cC(\mM, \mN)$ and
such that 
\begin{equation}\nonumber
\xymatrix{
&   \left(\mY \otimes \mZ \right) \otimes \mX \ar[rr]^{\phantom{aa} \gamma^\mX_{\mY \otimes \mZ}  \phantom{aa}} && \mX \otimes \left(\mY \otimes \mZ\right)  \ar[rd]^{\phantom{aa} \cA_{\mX, \mY, \mZ}} &  \\  
\mY \otimes \left( \mZ  \otimes \mX\right) \ar[ru]^{\cA_{\mY, \mZ, \mX}} \ar[rd]_{\Id_\mY \otimes \gamma^\mX_{\mZ} }
&&&&  \left(\mX \otimes \mY\right)\otimes \mZ \qquad \text{(Hexagon 1)}
\\
&   \mY\otimes \left(\mX \otimes \mZ\right)   \ar[rr]_{\phantom{aa} \cA_{\mY, \mX, \mZ}\phantom{aa} } &&  \left(\mY \otimes \mX\right) \otimes \mZ\ar[ru]_{\ \ \gamma^\mX_{\mY}\otimes \Id_\mZ} &\\
}
\end{equation}
commutes for all objects $\mY$, $\mZ$.
A morphism in $\cZ(\cC)$ from $(\mX, \gamma^\mX)$ to $(\mY, \gamma^\mY)$ is a $\cC$-morphism $f$ from $\mX$ to $\mY$ satisfying commutativity of the following diagram for all $\mM\in\obj(\cC)$.
\begin{align}\label{eq:centermor}
\xymatrix{
\mM \otimes \mX  \ar[rr]^{\Id_\mM \otimes f}\ar[d]_{\gamma^\mX_\mM} && \mM \otimes \mY\ar[d]^{\gamma^\mY_\mM} \\
\mX\otimes \mM \ar[rr]^{f \otimes \Id_\mM}  && \mY \otimes \mM.
}
\end{align}

\end{defn}

We have a forgetful functor $\cI:\cZ(\cC)\rightarrow\cC$ with $\cI(\mX,\gamma^\mX)=\mX$ for an object $(\mX,\gamma^\mX)$ and $\cI(f)=f$ for a morphism $f$.

A basic property of half-braidings that we will use is the following:
\begin{lemma}\label{lem:gamma_1}
 If $\gamma^{\mX}$ is a half-braiding, then $\gamma^\mX_\vac =r_\mX^{-1}\circ l_\mX$.
\end{lemma}
\begin{proof}
 Using the naturality of $\gamma^{\mX}$, Hexagon 1, and properties of the unit in $\cC$, the following diagram commutes:
 \begin{equation*}
  \xymatrixcolsep{5pc}
  \xymatrix{
   & \vac\otimes\mX \ar[r]^{\gamma^{\mX}_{\vac}} & X\otimes \vac & \\
   \vac\otimes(\vac\otimes\mX) \ar[ru]^{l_{\vac\otimes\mX}} \ar[r]^{\cA_{\vac,\vac,\mX}} \ar[rd]_{\Id_\vac\otimes\gamma^\mX_\vac} & (\vac\otimes\vac)\otimes\mX \ar[u]^{(l_\vac=r_\vac)\otimes\Id_{\mX}} \ar[r]^{\gamma^{\mX}_{\vac\otimes\vac}} & \mX\otimes(\vac\otimes\vac) \ar[u]^{\Id_{\mX}\otimes(l_\vac=r_\vac)} \ar[r]^{\cA_{\mX,\vac,\vac}} & (\mX\otimes\vac)\otimes\vac \ar[lu]_{r_{\mX\otimes\vac}} \\
    & \vac\otimes(\mX\otimes\vac) \ar[r]^{\cA_{\vac,\mX,\vac}} & (\vac\otimes\mX)\otimes\vac \ar[ru]_{\gamma^\mX_{\vac}\otimes\Id_{\vac}} & \\
  }
 \end{equation*}
That is,
\begin{align*}
 \gamma^\mX_{\vac} &  = r_{\mX\otimes\vac}\circ(\gamma^\mX_\vac\otimes\Id_\vac)\circ\cA_{\vac,\mX,\vac}\circ(\Id_\vac\otimes\gamma^\mX_\vac)\circ l_{\vac\otimes\mX}^{-1}\nonumber\\
 & =\gamma^\mX_{\vac}\circ r_{\vac\otimes\mX}\circ\cA_{\vac,\mX,\vac}\circ l_{\mX\otimes\vac}^{-1}\circ\gamma^\mX_\vac,
\end{align*}
using also the naturality of the unit isomorphisms. So
\begin{align*}
 \gamma^\mX_\vac & = l_{\mX\otimes\vac}\circ\cA_{\vac,\mX,\vac}^{-1}\circ r_{\vac\otimes\mX}^{-1} = l_{\mX\otimes\vac}\circ(\Id_{\vac}\otimes r_\mX^{-1}) = r_{\mX}^{-1}\circ l_\mX,
\end{align*}
by properties of the unit isomorphisms.
\end{proof}

The center is a tensor category with tensor product 
\begin{align}\label{ZC_Tens_Prod_Def}
(\mX, \gamma^\mX) \otimes (\mY, \gamma^\mY) = \left(\mX \otimes \mY, \gamma^{\mX \otimes \mY}\right), \qquad \gamma^{\mX \otimes \mY}:= \cA_{\mX, \mY,  \bullet}  \left(\Id_\mX \otimes \gamma^\mY\right)  \cA_{\mX, \bullet , \mY}^{-1}  \left(\gamma^\mX \otimes \Id_\mY\right) \cA_{\bullet, \mX, \mY}.
\end{align}
The definition of $\gamma^{\mX \otimes \mY}$ is precisely saying that the diagram
\begin{equation*}
\xymatrixcolsep{1.85pc}
\xymatrix{
&   \mM \otimes \left(\mX\otimes \mY\right) \ar[rr]^{\phantom{aa} \gamma^{\mX\otimes \mY}_{\mM}  \phantom{aa}} && \left(\mX \otimes \mY\right) \otimes \mM \ar[rd]^{\phantom{aa} \cA^{-1}_{\mX, \mY, \mM}} &  \\  
\left(\mM \otimes \mX \right) \otimes \mY \ar[ru]^{\cA^{-1}_{\mM, \mX, \mY}} \ar[rd]_{\gamma^\mX_{\mM} \otimes \Id_\mY}
&&&& \mX \otimes \left(\mY \otimes \mM\right)\qquad \text{(Hexagon 2)}
\\
& \left(\mX \otimes \mM\right) \otimes \mY \ar[rr]_{\phantom{aa} \cA^{-1}_{\mX, \mM, \mY}\phantom{aa} } && \mX \otimes \left(\mM \otimes \mY\right) \ar[ru]_{\ \ \Id_\mX \otimes \gamma_{\mM}^\mY} &\\
}
\end{equation*}
commutes for all objects $\mM$ in $\cC$. 
The unit of the center is $\one_{\cZ(\cC)} = (\one_\cC, r^{-1} \circ l)$ with $r$ and $l$ the right and left unit constraints in $\cC$.

The two commutative diagrams (Hexagon 1 and 2) together with the naturality of half-braidings ensure that for objects $(\mX,\gamma^\mX)$ and $(\mY,\gamma^\mY)$ in $\cZ(\cC)$, the $\cC$-isomorphism
\begin{equation*}
\cR_{(\mX, \gamma^\mX),(\mY, \gamma^\mY)}:= \gamma^\mY_\mX: (\mX\otimes\mY,\gamma^{\mX\otimes\mY})\rightarrow(\mY\otimes\mX,\gamma^{\mY\otimes\mX})
 \end{equation*}
is actually a morphism in $\cZ(\cC)$ and defines a braiding on the center. If $\cC$ is already a braided tensor category, then we have two tensor functors from $\cC$ to $\cZ(\cC)$:
\[
\cF: \mX \mapsto (\mX, \cR_{\bullet,\mX}),\,\,f\mapsto f 
\qquad \text{and} \qquad
\cF^{\text{rev}}: \mX \mapsto (\mX, \cR^{-1}_{\mX,  \bullet}),\,\,f\mapsto f. 
\]
Note that composing either $\cF$ or $\cF^{\rev}$ with the forgetful functor $\cI: \cZ(\cC)\rightarrow\cC$ yields the identity functor on $\cC$. In fact, $\cF$ is an example of a central functor structure on $\Id_{\cC}$. In general, a tensor functor $\cF: \cC\rightarrow\cM$ where $\cC$ is a braided tensor category is a \textbf{central functor} if there is a braided tensor functor $\cG: \cC\rightarrow\cZ(\cM)$ such that the diagram
\begin{equation}\nonumber
\xymatrix{
\cC  \ar[rr]^{\cF}\ar[rd]_{\cG} && \cM  \\
& \cZ(\cM) \ar[ur]_\cI &\\
}
\end{equation}
commutes. In the above example, $\cF^{\rev}$ is a central functor structure on $\Id_{\cC^{\rev}}$, where $\cC^{\rev}$ equals $\cC$ as a tensor category but has reversed braidings: $\cR^{\rev}_{\mX,\mY}=\cR_{\mY,\mX}^{-1}$.

We can slightly generalize the above central functor structures on the identity of a braided tensor category in the following way. Suppose we have a fully faithful tensor functor $\cF: \cC\rightarrow\cM$, with $\cC$ braided, so that $\cF$ is an equivalence of tensor categories between $\cC$ and $\cF(\cC)$ (the full subcategory of $\cM$ consisting of objects isomorphic to some $\cF(\mX)$ for $\mX$ an object of $\cC$). Then we can choose a functor $\cF': \cF(\cC)\rightarrow\cC$ and natural isomorphisms $\eta: \Id_{\cF(\cC)}\rightarrow\cF\circ\cF'$, $h: \cF'\circ\cF\rightarrow\Id_{\cC}$ which satisfy 
\begin{equation*}
 \cF(h_\mX)=\eta_{\cF(\mX)}^{-1}
\end{equation*}
for any object $\mX$ in $\cC$ (see for instance the proof of \cite[Proposition XI.1.5]{Ka}). Then $\cF: \cC\rightarrow\cF(\cC)$ is a central functor with extension $\cG: \cC\rightarrow\cZ(\cF(\cC))$ defined as follows:
\begin{equation*}
 \cG: \mX\mapsto(\cF(\mX), \gamma^{\cF(\mX)}),\,\,f\mapsto\cF(f) 
\end{equation*}
where for an object $\mY$ of $\cF(\cC)$, $\gamma^{\cF(\mX)}_\mY$ is the composition
\begin{align*}
 \mY\otimes\cF(\mX)\xrightarrow{\eta_{\mY}\otimes\Id_{\cF(\mX)}} & \cF(\cF'(\mY))\otimes\cF(\mX)\xrightarrow{\cong} \cF(\cF'(\mY)\otimes\mX)\nonumber\\
 &\xrightarrow{\cF(\cR_{\cF'(\mY),\mX})} \cF(\mX\otimes\cF'(\mY))\xrightarrow{\cong} \cF(\mX)\otimes\cF(\cF'(\mY))\xrightarrow{\Id_{\mX}\otimes\eta_{\mY}^{-1}} \cF(\mX)\otimes\mY.
\end{align*}
The tensor functor $\cF:\cC^{\rev}\rightarrow\cF(\cC)$ is also central with a braided extension given by
\begin{equation*}
 \cG^{\rev}: \mX\mapsto(\cF(\mX), (\gamma^{\rev})^{\cF(\mX)}),\,\,f\mapsto\cF(f),
\end{equation*}
where $(\gamma^{\rev})^{\cF(\mX)}_\mY$ is defined similarly to $\gamma^{\cF(\mX)}_\mY$, except that we use $\cR_{\mX,\cF'(\mY)}^{-1}$ instead of $\cR_{\cF'(\mY),\mX}$.

\begin{defn}
	Given a braided tensor category $\cC$, and a full braided tensor subcategory $\cB$, the \textbf{centralizer} $\cB'$ is the full subcategory of $\cC$ such that $\mX\in\obj(\cC)$ is an object of $\cB'$ if and only if $\cR_{\mY,\mX}\circ\cR_{\mX,\mY}=\Id_{\mX\otimes \mY}$ for all $\mY\in\obj(\cB)$. 
\end{defn}
It is easy to see from the hexagon axioms that $\cB'$ is a braided tensor subcategory of $\cC$. The following proposition is essentially \cite[Proposition 7.3]{Mue}; although it is only stated there for fusion categories, it is clear from the proof that no finiteness or semisimplicity conditions are necessary.
\begin{prop}\label{prop:mueger}
Let $\cC$ be a braided tensor category and suppose $\cF: \cC\rightarrow\cM$ is a fully faithful tensor functor. Then with the braided tensor functors $\cG:\cC\rightarrow\cZ(\cF(\cC))$ and $\cG^{\rev}:\cC^{\rev}\rightarrow\cZ(\cF(\cC))$ defined as above, we have
\begin{equation*}
 \cG(\cC)'=\cG^{\rev}(\cC^{\rev})\qquad and \qquad \cG^{\rev}(\cC^{\rev})'=\cG(\cC).
\end{equation*}
\end{prop}
\begin{proof}
An object $(\cF(\mX), g^{\cF(\mX)})$ of $\cZ(\cF(\cC))$ is an object of $\cG(\cC)'$ precisely when
\begin{equation*}
g^{\cF(\mX)}_{\cF(\mY)}\circ\gamma^{\cF(\mY)}_{\cF(\mX)}=\Id_{\cF(\mX)\otimes\cF(\mY)} 
\end{equation*}
for all objects $\mY$ in $\cC$. Since $g^{\cF(\mX)}$ is natural, this occurs precisely when, for all objects $\mZ$ in $\cF(\cC)$,
\begin{equation*}
 g^{\cF(\mX)}_{\mZ} = (\Id_{\cF(\mX)}\otimes\eta^{-1}_{\mZ})\circ g^{\cF(\mX)}_{\cF(\cF'(\mZ))}\circ(\eta_{\mZ}\otimes\Id_{\cF(\mX)}) = (\Id_{\cF(\mX)}\otimes\eta^{-1}_{\mZ})\circ(\gamma^{\cF(\cF'(\mZ))}_{\cF(\mX)})^{-1}\circ(\eta_{\mZ}\otimes\Id_{\cF(\mX)}).
\end{equation*}
By definition, this means $g^{\cF(\mX)}_\mZ$ is the composition
\begin{align*}
 \mZ &\otimes  \cF(\mX)\xrightarrow{\eta_\mZ\otimes\eta_{\cF(\mX)}} \cF(\cF'(\mZ))\otimes\cF(\cF'(\cF(\mX))) \xrightarrow{\cong} \cF(\cF'(\mZ)\otimes\cF'(\cF(\mX)))\nonumber\\
& \xrightarrow{\cF(\cR^{-1}_{\cF'(\cF(\mX)),\cF'(\mZ)})} \cF(\cF'(\cF(\mX))\otimes\cF'(\mZ))\xrightarrow{\cong} \cF(\cF'(\cF(\mX)))\otimes\cF(\cF'(\mZ)) \xrightarrow{\eta_{\cF(\mX)}^{-1}\otimes\eta_{\mZ}^{-1}} \cF(\mX)\otimes\mZ.
\end{align*}
However, because $\eta_{\cF(\mX)}^{-1}=\cF(h_\mX)$ and because the central isomorphisms in the composition are natural, we actually have $g^{\cF(\mX)}_\mZ$ equal to the composition
\begin{align*}
 \mZ\otimes\cF(\mX)\xrightarrow{\eta_{\mZ}\otimes\Id_{\cF(\mX)}} & \cF(\cF'(\mZ))\otimes\cF(\mX)\xrightarrow{\cong} \cF(\cF'(\mZ)\otimes\mX)\nonumber\\
 &\xrightarrow{\cF(\cR^{-1}_{\mX,\cF'(\mZ)})} \cF(\mX\otimes\cF'(\mZ))\xrightarrow{\cong} \cF(\mX)\otimes\cF(\cF'(\mZ))\xrightarrow{\Id_{\mX}\otimes\eta_{\mZ}^{-1}} \cF(\mX)\otimes\mZ,
\end{align*}
which is precisely $(\gamma^{\rev})^{\cF(\mX)}_\mZ$. This shows that $\cG(\cC)'=\cG^{\rev}(\cC^{\rev})$, and the proof that $\cG^{\rev}(\cC^{\rev})'=\cG(\cC)$ is the same.
\end{proof}

\section{From braid-reversed equivalences to algebras}

In this section we construct the canonical algebra associated to a tensor category $\cC$, paying close attention to the assumptions on $\cC$ needed to ensure the existence and commutativity of this algebra. In particular, we will see that when $\cC$ is rigid, braided, and semisimple but has infinitely many equivalence classes of simple objects, then the canonical algebra may be constructed in $\cC_\oplus$.

\subsection{Associative algebras}

For this subsection, we take $\cC$ to be an $\mathbb{F}$-linear (abelian) tensor category, not necessarily braided or finite, where $\mathbb{F}$ is a field. (For our purposes in this paper, we may take $\mathbb{F}=\CC$.) We will construct associative algebras in $\cC$ associated to a $\cC$-module category $\cM$. Although this construction seems to be well known (see for example \cite[Exercise 7.9.9]{EGNO}), we could not find a complete proof in the literature.

Recall that in a $\cC$-module category $\cM$ we have natural associativity isomorphisms
\begin{align*}
\cA_{\mX_1,\mX_2,\mM}:\mX_1\otimes(\mX_2\otimes \mM)\rightarrow (\mX_1 \otimes \mX_2)\otimes \mM.
\end{align*}
for $\mX_1,\mX_2\in\obj(\cC), \mM\in\obj(\cM)$. For every pair $\mM_1,\mM_2\in\obj(\cM)$, we have a contravariant functor
\begin{align*}
\cG_{\mM_1, \mM_2} : \cC &\rightarrow \cVec_\mathbb{F}\\
\mX &\mapsto \HHom_\cM\left( \mX\otimes \mM_1, \mM_2\right)\quad\mathrm{for}\quad\mX\in\obj(\cC)\\
f&\mapsto (g\mapsto g\circ(f\otimes\Id_{\mM_1})) \quad\mathrm{for}\quad f\in\HHom_{\cC}(\mX,\mY), g\in\HHom_{\cM}(\mY\otimes \mM_1,\mM_2).
\end{align*}
Assume that $\cG_{\mM_1,\mM_2}$ is representable, which means that there exists $\intHom(\mM_1, \mM_2)$, perhaps in a suitable completion of $\cC$, called the internal Hom of $\mM_1$ and $\mM_2$ such that there are natural isomorphisms
\begin{equation}\label{eq:lambda}
\lambda_\mX(\mM_1, \mM_2): \HHom_\cM\left(\mX\otimes \mM_1, \mM_2\right)  \xrightarrow{\ \ \cong\ \ } \HHom_\cC\left( \mX, \intHom(\mM_1, \mM_2)\right).
\end{equation}
So, for $f: \mX\rightarrow \mY$ with $\mX,\mY\in\obj(\cC)$, we have the following commuting diagram:
\begin{align*}
\xymatrix{
\HHom_\cM\left(\mX\otimes \mM_1, \mM_2\right) \ar[r]^(.45){\lambda_\mX} &  \HHom_\cC\left( \mX, \intHom(\mM_1, \mM_2)\right) \\
\HHom_\cM\left(\mY\otimes \mM_1, \mM_2\right) \ar[r]^(.45){\lambda_\mY} \ar[u]^{g\,\mapsto\, g\circ(f\otimes\Id_{\mM_1})} &  \HHom_\cC\left( \mY, \intHom(\mM_1, \mM_2)\right)\ar[u]_{h\,\mapsto\, h\circ f}.
}
\end{align*}

\begin{rema}\label{Grepresent}
 By \cite[Corollary 1.8.11, Section 7.9]{EGNO}, $\cG_{\mM_1,\mM_2}$ is representable if $\cC$ is a finite (multi)tensor category. If $\cC$ is not finite, then \cite[Section 7.9]{EGNO} states that internal Homs still exist as ind-objects of the completion $ind-\cC$. In the setting of vertex operator algebras, certain internal Homs were constructed in \cite{Li} as weak modules (and thus not necessarily objects of $\cC$ itself).
\end{rema}

Fix an object $\mM\in\obj(\cM)$ and abbreviate $\lambda_\mX:= \lambda_\mX({\mM, \mM})$ and $\aA:= \intHom(\mM, \mM)$. For $i=1,2,3$, we will be using the notation $\varphi_i$ for a morphism in $\HHom_\cM\left(\mX_i\otimes \mM, \mM\right)$, where $\mX_i \in \obj(\cC)$. For $\mX_1,\mX_2\in\obj(\cC)$, we have a linear map
\begin{equation*}
 \nu_{\mX_1,\mX_2}: \HHom_\cM\left(\mX_1\otimes \mM, \mM\right) \otimes_\mathbb{F} \HHom_\cM\left(\mX_2\otimes \mM, \mM\right) \rightarrow \HHom_\cM\left((\mX_1 \otimes \mX_2)\otimes \mM, \mM\right)
\end{equation*}
under which $\varphi_1\otimes_\mathbb{F}\varphi_2$ is sent to the composition
\begin{equation*}
(\mX_1 \otimes \mX_2) \otimes \mM  \xrightarrow{\cA^{-1}_{\mX_1, \mX_2, \mM}} \mX_1 \otimes (\mX_2 \otimes \mM ) \xrightarrow{\Id_{\mX_1} \otimes \varphi_2} \mX_1 \otimes \mM \xrightarrow{\  \varphi_1 \\ } \mM.
\end{equation*}
Then the natural family of isomorphisms $\{ \lambda_\mX \}$ induces a natural family of linear maps
\begin{align*}
\mu_{\mX_1, \mX_2}: \HHom_\cC\left( \mX_1, \aA \right) \otimes_\mathbb{F} \HHom_\cC\left( \mX_2, \aA \right) & \rightarrow \HHom_\cC\left( (\mX_1 \otimes \mX_2), \aA\right), \\
\lambda_{\mX_1}(\varphi_1) \otimes_\mathbb{F}  \lambda_{\mX_2}(\varphi_2) & \mapsto  \lambda_{\mX_1\otimes \mX_2}(\nu_{\mX_1,\mX_2}(\varphi_1\otimes_\mathbb{F}\varphi_2)).
\end{align*}
{Note that $\nu$ is a natural transformation of contravariant bifunctors $\cM\times \cM\rightarrow \cV ec_\mathbb{F}$ and
	 $\mu$ is a natural transformation of contravariant bifunctors $\cC\times \cC\rightarrow \cVec_\mathbb{F}$.}
{By definition, we  have:
\begin{align}
\lambda_{\mX_1\otimes \mX_2}\circ \nu_{\mX_1,\mX_2}=
\mu_{\mX_1,\mX_2}\circ (\lambda_{\mX_1}\otimes_{\mathbb{F}}\lambda_{\mX_2}).
\end{align}}
Consider the diagram 
\begin{equation}\nonumber
\begin{split}
\xymatrix{
\left( \left(\mX_1 \otimes \mX_2 \right)\otimes \mX_3\right) \otimes \mM  \ar[rr]^{  \cA^{-1}_{\mX_1, \mX_2, \mX_3} \otimes \Id_\mM }\ar[d]_{\cA^{-1}_{\mX_1\otimes \mX_2, \mX_3, \mM}} 
&&  \left(\mX_1 \otimes \left(\mX_2 \otimes \mX_3\right)\right) \otimes \mM \ar[d]^{\cA^{-1}_{\mX_1, \mX_2\otimes \mX_3, \mM}}
\\
\left(\mX_1 \otimes \mX_2 \right)\otimes \left(\mX_3 \otimes \mM\right)   \ar[d]_{\Id_{\mX_1 \otimes \mX_2} \otimes\varphi_3} \ar[rrd]^{\cA^{-1}_{\mX_1, \mX_2, \mX_3\otimes \mM}} 
&&  \mX_1 \otimes \left(\left(\mX_2 \otimes \mX_3\right)  \otimes \mM\right) \ar[d]^{\Id_{\mX_1} \otimes\cA^{-1}_{ \mX_2, \mX_3, \mM}}
\\
\left(\mX_1 \otimes \mX_2 \right)\otimes  \mM   \ar[rd]_{\cA^{-1}_{\mX_1, \mX_2, \mM}} 
&&  \mX_1 \otimes \left(\mX_2 \otimes \left(\mX_3 \otimes \mM\right) \right)  \ar[ld]^{\qquad \ \ \ \Id_{\mX_1} \otimes (\Id_{\mX_2} \otimes\varphi_3)}
\\
& \mX_{1} \otimes (\mX_2 \otimes \mM)\ar[d]_{\varphi_1 \circ(\Id_{\mX_1} \otimes\varphi_2)} &
\\
& \mM &
\\
} \\ 
\end{split}
\end{equation}
which commutes due to naturality of associativity and the pentagon axiom for $\cM$.
By naturality of both $\lambda$ and associativity we then have the corresponding commutative diagram
\begin{equation}\label{eq:comd1}
\begin{split}
\xymatrix{
\left(\mX_1 \otimes \mX_2 \right)\otimes \mX_3   \ar[rr]^{  \cA^{-1}_{\mX_1, \mX_2, \mX_3} }\ar[rd]_{\LHS(\varphi_1, \varphi_2, \varphi_3)\ \ \ } 
&&  \mX_1 \otimes \left(\mX_2 \otimes \mX_3\right)  \ar[ld]^{\ \ \ \RHS(\varphi_1, \varphi_2, \varphi_3)}
\\
& \aA &
\\
} 
\end{split}
\end{equation}
where (leaving out identity morphisms for readability)
\begin{equation*}
\begin{split}
\LHS(\varphi_1, \varphi_2, \varphi_3)&=\lambda_{(\mX_1\otimes \mX_2) \otimes \mX_3}\left( \varphi_1 \circ \varphi_2 \circ\varphi_3 \circ \cA^{-1}_{\mX_1, \mX_2, \mX_3\otimes \mM} \circ \cA^{-1}_{\mX_1\otimes \mX_2, \mX_3, \mM}\right)\\
&= \lambda_{(\mX_1\otimes \mX_2) \otimes \mX_3}\left( \varphi_1 \circ \varphi_2 \circ \cA^{-1}_{\mX_1, \mX_2, \mM} \circ\varphi_3  \circ \cA^{-1}_{\mX_1\otimes \mX_2, \mX_3, \mM}\right), \\
\RHS(\varphi_1, \varphi_2, \varphi_3) & = \lambda_{\mX_1\otimes \left(\mX_2\otimes \mX_3\right)}\left( \varphi_1 \circ \varphi_2 \circ\varphi_3 \circ \cA^{-1}_{\mX_2, \mX_3\otimes M} \circ \cA^{-1}_{\mX_1, \mX_2\otimes \mX_3, \mM}\right) .
\end{split}
\end{equation*} 

We now define 
\begin{align}
m :=  \mu_{\aA, \aA}\left(\Id_\aA \otimes_\CC \Id_\aA\right) = \lambda_{\aA\otimes \aA}\left( \lambda^{-1}_\aA(\Id_\aA) \circ \lambda^{-1}_\aA(\Id_\aA) \circ \cA^{-1}_{\aA, \aA, \mM}\right) .
\end{align}
Given any $f_1:\mX_1\rightarrow \aA$ and $f_2: \mX_2\rightarrow \aA$, we get the following commuting diagram:
\begin{align}
\xymatrix{
\HHom(\mX_1\otimes \mX_2, \aA) & \HHom(\aA\otimes \aA, \aA) \ar[l]\\
\HHom(\mX_1,\aA)\otimes_{\mathbb{F}}\HHom(\mX_2, \aA)  \ar[u]^{\mu_{\mX_1,\mX_2}} 
& \HHom(\aA, \aA)\otimes_\mathbb{F} \HHom(\aA,\aA) \ar[u]^{\mu_{\aA,\aA}} \ar[l].}
\end{align}
Tracing the image of $\Id_\aA\otimes_{\mathbb{F}}\Id_\aA\in \HHom(\aA, \aA)\otimes_\mathbb{F} \HHom(A,A)$ gets us the following commutative diagram:
\begin{equation}\label{eq:muproperty}
\xymatrix{
\mX_1\otimes \mX_2 \ar[rr]^{\mu_{\mX_1, \mX_2}(f_1\otimes_\CC f_2)}\ar[rd]_{f_1\otimes f_2} && \aA \\
& \aA\otimes \aA \ar[ur]_m
}.
\end{equation}
Taking $\mX_1=\aA\otimes \aA$, $\mX_2=\aA$, $f_1=m$, $f_2=\Id_\aA$,
this diagram implies that 
\begin{equation*}
\begin{split}
m &\circ\left( m\otimes \Id_\aA\right) = \mu_{\aA\otimes \aA, \aA}\left(m \otimes_\mathbb{F} \Id_\aA\right) \\
&= \mu_{\aA\otimes \aA, \aA}\left(\lambda_{\aA\otimes \aA}\left( \lambda^{-1}_\aA(\Id_\aA) \circ \lambda^{-1}_\aA(\Id_\aA) \circ \cA^{-1}_{\aA, \aA, \mM}\right)  \otimes_\mathbb{F} \Id_\aA\right) \\ 
&=\lambda_{(\aA \otimes \aA) \otimes \aA}\left( \lambda_{\aA \otimes \aA}^{-1}\left(\lambda_{\aA\otimes \aA}\left( \lambda^{-1}_\aA(\Id_\aA) \circ \lambda^{-1}_\aA(\Id_\aA) \circ \cA^{-1}_{\aA, \aA, \mM}\right)\right) \circ  \lambda_\cA^{-1}(\Id_\aA)\circ\cA^{-1}_{\aA\otimes \aA, \aA, \mM}\right)\\
&=\lambda_{(\aA \otimes \aA) \otimes \aA}\left( \lambda^{-1}_\aA(\Id_\aA) \circ \lambda^{-1}_\aA(\Id_\aA) \circ \cA^{-1}_{\aA, \aA, \mM} \circ  \lambda^{-1}_\aA(\Id_\aA)\circ\cA^{-1}_{\aA\otimes \aA, \aA, \mM}\right)\\
&= \LHS(\lambda^{-1}_\aA(\Id_\aA), \lambda^{-1}_\aA(\Id_\aA), \lambda^{-1}_\aA(\Id_\aA)).
\end{split}
\end{equation*}
Analogously
\begin{equation*}
\begin{split}
m \circ\left( \Id_\aA \otimes m\right) &= \mu_{\aA, \aA\otimes \aA}\left(\Id_\aA \otimes_\mathbb{F} m\right) \\
&= \mu_{\aA, \aA\otimes \aA}\left( \Id_\aA \otimes_\mathbb{F}  \lambda_{\aA\otimes \aA}\left( \lambda^{-1}_\aA(\Id_\aA) \circ \lambda^{-1}_\aA(\Id_\aA) \circ \cA^{-1}_{\aA, \aA, \mM}\right) \right) \\ 
&=\lambda_{\aA \otimes (\aA \otimes \aA)}\left( \lambda_\aA^{-1}(\Id_\aA)\circ  \lambda^{-1}_\aA(\Id_\aA) \circ \lambda^{-1}_\aA(\Id_\aA) \circ \cA^{-1}_{\aA, \aA, \mM} \circ 
 \cA^{-1}_{\aA, \aA\otimes \aA, \mM}\right)\\
&= \RHS(\lambda^{-1}_\aA(\Id_\aA), \lambda^{-1}_\aA(\Id_\aA), \lambda^{-1}_\aA(\Id_A)), 
\end{split}
\end{equation*}
so that this computation together with \eqref{eq:comd1} implies that all triangles of the diagram commute:
\begin{equation}\label{eq:comd2}
\begin{split}
\xymatrix{
(\aA \otimes \aA)\otimes \aA    \ar[rr]^{  \cA^{-1}_{\aA, \aA, \aA} } \ar[d]_{m\otimes \Id_\aA }\ar[rdd]^{ \LHS\ \ \ } 
&&  \aA \otimes \left(\aA \otimes \aA\right)  \ar[d]^{\Id_\aA\otimes m}\ar[ldd]_{\ \ \ \RHS}
\\
\aA\otimes \aA \ar[rd]_{m} && \aA\otimes \aA \ar[ld]^m \\
& \aA &
\\
} 
\end{split}
\end{equation}
with $\LHS=\LHS(\lambda^{-1}_\aA(\Id_\aA), \lambda^{-1}_\aA(\Id_\aA), \lambda^{-1}_\aA(\Id_\aA))$,   $\RHS=\RHS(\lambda^{-1}_\aA(\Id_\aA), \lambda^{-1}_\aA(\Id_\aA), \lambda^{-1}_\aA(\Id_\aA))$. Thus the multiplication $m: \aA\otimes\aA\rightarrow\aA$ is associative.

The natural candidate for a unit morphism is $\iota_\aA: \vac\rightarrow{\aA}$ is $\lambda_\one(l_\mM)$ , where $l_\mM:\one\otimes \mM\rightarrow \mM$ is the left unit isomorphism for $\mM$.
We have:
\begin{align}
m&\circ(\iota_\aA \otimes \Id_\aA)=\mu_{\one, \aA}(\iota_\aA \otimes_\mathbb{F} \Id_\aA)
=\lambda_{\one\otimes \aA}\left( \lambda_\one^{-1}\left(\lambda_\one(l_\mM)\right) \circ \lambda_\aA^{-1}(\Id_\aA) \circ \cA^{-1}_{\one, \aA, \mM}\right)\notag\\
&=\lambda_{\one\otimes \aA}\left( l_\mM\circ \lambda_\aA^{-1}(\Id_\aA) \circ \cA^{-1}_{\one, \aA, \mM}\right)
=\lambda_{\one\otimes \aA}\left(\lambda_\aA^{-1}(\Id_\aA)\circ (l_\aA\otimes \Id_\mM) \right)
=\Id_\aA\circ l_\aA = l_\aA,
\end{align}
where the first equality follows by \eqref{eq:muproperty}, fourth by properties of unit isomorphisms and fifth by naturality of $\lambda$.
We conclude
\begin{thm}\label{thm:assoc}
Let $\cC$ be a multitensor category, $\cM$ a left $\cC$-module category, and $\mM$ an object of $\cM$ such that the functor $\cG_{\mM,\mM}$ is representable. Then with the natural isomorphisms $\lambda$ defined by \eqref{eq:lambda}, $\aA:=\intHom(\mM, \mM)$ together with left unit $\iota_\aA=\lambda_\one(l_\mM)$ and multiplication $m=\lambda_{\aA\otimes \aA}\left( \lambda^{-1}_\aA(\Id_\aA) \circ \lambda^{-1}_\aA(\Id_\aA) \circ \cA^{-1}_{\aA, \aA, \mM}\right)$ is a left-unital associative algebra in $\cC$. 
\end{thm}

\subsection{Commutative algebras}

Now taking $\cC$ to be a braided tensor category, we will find conditions under which the algebra $\aA$ of the previous subsection is commutative. For this, we will take $\cM$ of the previous subsection to be itself a tensor  category with tensor unit $\one_\cM$, and we will consider the algebra $\aA=\intHom(\one_\cM, \one_\cM)$. 

We assume there are natural associativity isomorphisms
\begin{equation}\label{CMassocisos}
 \cA_{\mX,\mM_1,\mM_2}: \mX\otimes(\mM_1\otimes_\cM \mM_2)\rightarrow (\mX\otimes\mM_1)\otimes_{\cM} \mM_2
\end{equation}
for objects $\mX\in\obj(\cC)$ and $\mM_1, \mM_2\in\obj(\cM)$, and that all associativity and unit isomorphisms are compatible in the sense that all triangle and pentagon diagrams commute. Let 
\[
\cF: \cC \rightarrow \cM, \qquad \mX \mapsto \mX \otimes \one_\cM
\]
be the induction functor, which is in fact a tensor functor with functorial isomorphisms
\begin{equation*}
 \cF(\mX_1\otimes X_2)\xrightarrow{\cong} \cF(X_1)\otimes\cF(X_2)
\end{equation*}
given by the composition
\begin{align*}
 (\mX_1\otimes\mX_2)\otimes\vac_\cM\xrightarrow{\cA_{\mX_1,\mX_2,\vac_\cM}^{-1}} \mX_1\otimes(\mX_2\otimes\vac_\cM) & \xrightarrow{\Id_{\mX_1}\otimes l_{\mX_2\otimes\vac_\cM}^{-1}} \mX_1\otimes(\vac_{\cM}\otimes(\mX_2\otimes\vac_\cM))\nonumber\\
  &\xrightarrow{\cA_{\mX_1,\vac_\cM,\mX_2\otimes\vac_{\cM}}} (\mX_1\otimes\vac_\cM)\otimes(\mX_2\otimes\vac_{\cM}).
\end{align*}
Assume that $\cF$ is a central functor, so that there is a braided tensor functor $\cG : \cC \rightarrow \cZ(\cM)$ such that $\cF=\cI\circ\cG$, where $\cI: \cZ(\cM)\rightarrow\cM$ is the forgetful functor.

The following theorem is \cite[Lemma 3.5]{DMNO}, but we add details and observe that neither finiteness nor semisimplicity is needed in the argument:
\begin{thm}\label{thm:comm}
In the setting of Theorem \ref{thm:assoc}, assume that $\cM$ is a tensor category, the natural associativity isomorphisms \eqref{CMassocisos} exist, and that the functor $\cG_{\vac_\cM,\vac_\cM}$ of the previous subsection is representable. Assume in addition that $\cC$ is a braided tensor category and induction $\cF: \cC\rightarrow\cM$ is a central functor. Then the multiplication on $\aA=\intHom(\one_\cM, \one_\cM)$ is commutative.
\end{thm}
\begin{proof}
We need to show that the multiplication map $m \in  \HHom_{\cC}(\aA \otimes \aA, \aA)$ is commutative. 
Recalling that
\[
m  = \lambda_{\aA\otimes \aA}\left( \lambda^{-1}_\aA(\Id_\aA) \circ \lambda^{-1}_\aA(\Id_\aA) \circ \cA_{\aA, \aA, \one_\cM}^{-1}\right), 
\]
we consider the image of $m$ under $\lambda^{-1}_{\aA\otimes \aA}$ in $\HHom_{\cM}(\cF(\aA \otimes \aA), \one_\cM)$.
By naturality of $\lambda_{\aA\otimes \aA}$ we have that
\[
\lambda^{-1}_{\aA\otimes \aA}(m\circ \cR_{\aA\otimes \aA}) = \lambda^{-1}_{\aA\otimes \aA}(m) \circ (\cR_{\aA\otimes \aA} \otimes \Id_{\one_\cM} ) = \lambda^{-1}_{\aA\otimes \aA}(m) \circ \cF(\cR_{\aA\otimes \aA}),
\]
so we must show that $\lambda_{\aA\otimes\aA}^{-1}(m)=\lambda_{\aA\otimes\aA}^{-1}(m)\circ\cF(\cR_{\aA\otimes\aA})$. To show this, we will use the diagram
\begin{equation*}
 \xymatrix{
 \cF(\aA\otimes\aA) \ar[r]^{\cong\ \ \ } \ar[d]^{\cF(\cR_{\aA,\aA})} & \cF(\aA)\otimes\cF(\aA) \ar[rr]^(.52){\lambda_\aA^{-1}(\Id_\aA)\otimes\Id_{\cF(\aA)}} \ar[d]^{\gamma^{\cF(\aA)}_{\cF(\aA)}} && \vac_\cM\otimes\cF(\aA) \ar[rr]^(.52){\Id_{\vac_\cM}\otimes\lambda_\aA^{-1}(\Id_\aA)} \ar[d]^{\gamma^{\cF(\aA)}_{\vac_\cM}} && \vac_\cM\otimes\vac_\cM \ar[r]^(.6){l_{\vac_\cM}} & \vac_{\cM} \\
 \cF(\aA\otimes\aA) \ar[r]^{\cong\ \ \ } & \cF(\aA)\otimes\cF(\aA) \ar[rr]^(.52){\Id_{\cF(\aA)}\otimes\lambda_\aA^{-1}(\Id_\aA)} && \cF(\aA)\otimes\vac_{\cM} \ar[rr]^(.52){\lambda_\aA^{-1}(\Id_\aA)\otimes\Id_{\vac_\cM}} && \vac_{\cM}\otimes\vac_\cM \ar[ru]_{r_{\vac_\cM}} & \\
 }
\end{equation*}
The left square commutes because $\cF$ lifts to the braided tensor functor $\cG$ and because the braiding isomorphism $\cR_{(\cF(\aA),\gamma^{\cF(\aA)}), (\cF(\aA),\gamma^{\cF(\aA)})}$ in $\cZ(\cM)$ is given by $\gamma^{\cF(\aA)}_{\cF(\aA)}$. The square in the middle commutes by the naturality of $\gamma^{\cF(\aA)}$, and the pentagon commutes by Lemma \ref{lem:gamma_1} and naturality of the unit isomorphisms.

Because $l_{\vac_\cM}=r_{\vac_\cM}$ we now see that it suffices to show that $\lambda_{\aA\otimes\aA}^{-1}(m)$ is given by the top (equivalently the bottom) row of the diagram. Since by definition
\begin{equation*}
 \lambda_{\aA\otimes\aA}^{-1}(m) =\lambda_\aA^{-1}(\Id_\aA)\circ(\Id_\aA\otimes\lambda_\aA^{-1}(\Id_\aA))\circ\cA_{\aA,\aA,\vac_{\cM}}^{-1},
\end{equation*}
this will follow from commutativity of the diagram
\begin{equation*}
\xymatrix{
\cF(\aA\otimes \aA)   \ar[r]^{ \cA_{\aA, \aA, \one_\cM}^{-1} }\ar[rd]_{\cong } & \aA \otimes \cF(\aA) \ar[rrr]^{\Id_\aA \otimes \lambda^{-1}_\aA(\Id_\aA)}  &&& \cF(\aA) \ar[rr]^{ \lambda^{-1}_\aA(\Id_\aA) }  && \one_{\cM}   \\
 & \cF(\aA) \otimes \cF(\aA) \ar[rrr]^{\Id_{\cF(\aA)} \otimes \lambda^{-1}_\aA(\Id_\aA)} \ar[u]_{\cong } &&& \cF(\aA) \otimes \one_{\cM} \ar[rr]^{ \lambda^{-1}_\aA(\Id_\aA) \otimes \Id_{\one_\cM}} \ar[u]_{r_{\cF(\aA)}}  && \one_{\cM} \otimes \one_\cM \ar[u]^{r_{\vac_{\cM}} }
}
\end{equation*}
In fact, the right square commutes by naturality of the unit isomorphisms, and recalling the definition of the functorial isomorphism $\cF(\aA\otimes\aA)\rightarrow\cF(\aA)\otimes\cF(\aA)$, we see that the triangle commutes if the vertical isomorphism is $(\Id_{\aA}\otimes l_{\aA\otimes\vac_\cM})\circ\cA_{\aA,\vac_\cM,\aA\otimes\vac_\cM}^{-1}$. Then commutativity of the square in the middle follows from the commutative diagram
\begin{equation*}
\xymatrixcolsep{4pc}
 \xymatrix{
 (\aA\otimes\vac_\cM)\otimes(\aA\otimes\vac_{\cM}) \ar[r]^{\cA_{\aA,\vac_\cM,\aA\otimes\vac_\cM}^{-1}} \ar[d]^{\Id_{\aA\otimes\vac_\cM}\otimes\lambda_\aA^{-1}(\Id_\aA)} & \aA\otimes(\vac_\cM\otimes(\aA\otimes\vac_\cM)) \ar[r]^(.58){\Id_\aA\otimes l_{\aA\otimes\vac_\cM}} \ar[d]^{\Id_\aA\otimes(\Id_{\vac_\cM}\otimes\lambda_\aA^{-1}(\Id_\aA))} & \aA\otimes(\aA\otimes\vac_\cM) \ar[d]^{\Id_\aA\otimes\lambda_\aA^{-1}(\Id_\aA)} \\
 (\aA\otimes \vac_\cM)\otimes\vac_{\cM} \ar[r]^{\cA_{\aA,\vac_\cM,\vac_\cM}^{-1}} & \aA\otimes(\vac_\cM\otimes\vac_\cM) \ar[r]^(.58){\Id_{\aA}\otimes l_{\vac_\cM}} & \aA\otimes\vac_{\cM} \\
 }
\end{equation*}
together with the unit triangle constraint
\begin{equation*}
 (\Id_\aA\otimes l_{\vac_\cM})\circ\cA_{\aA,\vac_\cM,\vac_\cM}^{-1} = (\Id_\aA\otimes r_{\vac_\cM})\circ\cA_{\aA,\vac_\cM,\vac_\cM}^{-1} = r_{\aA\otimes\vac_\cM}.
\end{equation*}
This completes the proof that $m\circ\cR_{\aA,\aA}=m$.
\end{proof}

\subsection{Canonical algebras}\label{sec:canonical}

We will now construct commutative algebras more concretely: we will see that what is called the canonical algebra associated to a suitable braided tensor category is always commutative.

\begin{defn}
Let $\cC$ be a monoidal category. The {\bf opposite category} $\cCop$ is the same as $\cC$ as a category but has monoidal structure
\[
\mX \otimes_{\text{op}} \mY := \mY\otimes \mX,
\]
associativity isomorphisms $\cA_{\mX, \mY, \mZ}^{\text{op}} = \cA_{\mZ, \mY, \mX}^{-1}$, and unit isomorphisms $l^{\text{op}}_\mX=r_\mX$, $r^{\text{op}}_\mX=l_{\mX}$. 
\end{defn}

By \cite[Example 7.4.2]{EGNO}, a multitensor category $\cC$ is a module category for the Deligne product $\cC\boxtimes \cCop$ with module map
\[
(\mX\boxtimes \mY) \otimes \mZ := (\mX \otimes \mZ) \otimes \mY.
\]
If we assume in addition that $\cC$ is braided, then $\cC^\text{op}$ also has a braiding given by
\begin{equation*}
 \cR_{\mX,\mY}^{\text{op}} =\cR_{\mX,\mY}^{-1}: \mX\otimes_\text{op}\mY\rightarrow\mY\otimes_\text{op}\mX.
\end{equation*}
Recalling the braid-reversed category $\cC^{\rev}$ from Section \ref{sec:center}, the identity functor on $\cC$ gives a braided tensor equivalence between $\cC^\text{rev}$ and $\cC^\text{op}$, with functorial isomorphisms
\begin{equation*}
 \cR^{-1}_{\mY,\mX} : \Id_\cC(\mX\otimes\mY) \rightarrow \Id_\cC(\mX)\otimes_\text{op}\Id_\cC(\mY).
\end{equation*}
Thus we may view $\cC$ as a module category for either $\cC\fus\cC^\text{op}$ or $\cC\fus\cC^{\text{rev}}$. In this setting, the natural associativity isomorphism
\begin{equation*}
 \cA_{\mX\fus\mY, \mZ_1,\mZ_2}: (\mX\fus\mY)\otimes(\mZ_1\otimes\mZ_2)\rightarrow((\mX\fus\mY)\otimes\mZ_1)\otimes\mZ_2
\end{equation*}
amounts to an isomorphism
\begin{equation*}
 (\mX\otimes(\mZ_1\otimes\mZ_2))\otimes\mY\rightarrow ((\mX\otimes\mZ_1)\otimes\mY)\otimes\mZ_2,
\end{equation*}
which is given by an appropriate combination of associativity isomorphisms together with $\cR_{\mY,\mZ_2}^{-1}$.

To see that the induction functor $\cF: \cC\fus\cC^\text{rev}\rightarrow\cC$ is a central functor, we note that $\cF$ is naturally isomorphic via the unit isomorphisms to the functor
\begin{equation*}
 \mX\fus\mY\mapsto\mX\otimes\mY,\quad f\fus g\mapsto f\otimes g.
\end{equation*}
That is, $\cF$ amounts to the extension to $\cC\fus\cC^\text{rev}$ of the identity functors from $\cC$ and $\cC^\text{rev}$ into $\cC$, both of which are central functors lifting to $\cZ(\cC)$ via $\mX\rightarrow(\mX,\cR_{\bullet,\mX})$ and $\mY\rightarrow(\mY,\cR_{\mY,\bullet}^{-1})$, respectively. Since the images of these two functors in $\cZ(\cC)$ centralize each other (recall Proposition \ref{prop:mueger}, or \cite[Proposition 7.3]{Mue}), the extension to $\cC\fus\cC^\text{rev}$ is also central. To be more specific, $\cF$ lifts to the functor $\cG: \cC\fus\cC^{\text{rev}}\rightarrow\cZ(\cC)$ given on objects by
$\mX\fus\mY\rightarrow(\mX\otimes\mY, \gamma^{\mX\otimes\mY})$ where $\gamma^{\mX\otimes\mY}_{\mZ}$ is given by the composition
\begin{align*}
 \mZ\otimes(\mX\otimes\mY)\xrightarrow{\cA_{\mZ,\mX,\mY}} (\mZ\otimes\mX) & \otimes\mY  \xrightarrow{\cR_{\mZ,\mX}\otimes\Id_\mY} (\mX\otimes\mZ)\otimes\mY\nonumber\\
 &\xrightarrow{\cA_{\mX,\mZ,\mY}^{-1}} \mX\otimes(\mZ\otimes\mY)\xrightarrow{\Id_\mX\otimes\cR_{\mY,\mZ}^{-1}} \mX\otimes(\mY\otimes\mZ)\xrightarrow{\cA_{\mX,\mY,\mZ}} (\mX\otimes\mY)\otimes\mZ,
\end{align*}
as in \eqref{ZC_Tens_Prod_Def}.

It now follows from Theorem \ref{thm:comm} that if the functor $\cG_{\vac,\vac}: \cC\fus\cC^{\rev}\rightarrow\cVec$ given on objects by
\begin{equation*}
 \mX\fus\mY\mapsto\HHom_\cC\left((\mX\otimes\vac)\otimes\mY,\vac\right)\cong\HHom_\cC\left(\mX\otimes\mY,\vac\right)
\end{equation*}
is representable, then $\aA=\intHom(\one, \one)$ is a commutative associative algebra in $\cC\fus\cCrev$. We have two situations in which $\cG_{\vac,\vac}$ is representable. First, recalling Remark \ref{Grepresent}, this holds when $\cC$ is a finite (in particular, rigid) braided tensor category. Secondly, when $\cC$ is not necessarily finite or rigid but is semisimple and has a contragredient functor, then $\cG_{\vac,\vac}$ is representable if we replace $\cC$ and $\cC\fus\cCrev$ with their direct sum completions. In this case we can take
\begin{equation*}
 \aA = \bigoplus_{\mX \in \irr(\cC)} \mX' \boxtimes \mX
\end{equation*}
where $\irr(\cC)$ is a set of equivalence class representatives for the simple objects in $\cC$. Note that when $\cC$ has infinitely many equivalence classes of simple objects, then $\aA$ is not an object of $\cC\fus\cCrev$ but is an object of $(\cC\fus\cCrev)_\oplus^{fin}$. For the braided tensor category structure on such completions, we refer again to \cite{CGR} and especially \cite{AR}; see also Appendix \ref{app:dirsum}.

To describe the algebra structure on $\aA$ more concretely, let us assume $\cC$ is rigid and take simple objects $\mX$, $\mY$ of $\cC$. We would like to determine the $\mZ\in\irr(\cC)$ for which $m((\mX^*\fus\mX)\otimes(\mY^*\fus\mY))\cap(\mZ^*\fus\mZ)\neq 0$.
We first observe that by the definition of $m$ and the naturality of $\lambda$,
\begin{equation*}
 m\vert_{(\mX^*\fus\mX)\otimes(\mY^*\fus\mY)} =\lambda_{(\mX^*\fus\mX)\otimes(\mY^*\fus\mY)}\left(\lambda_{\aA\otimes\aA}^{-1}(m)\vert_{((\mX^*\fus\mX)\otimes(\mY^*\fus\mY))\otimes\vac}\right).
\end{equation*}
The morphism inside parentheses here is the right-side composition in the diagram
\begin{equation*}
 \xymatrixcolsep{4pc}
 \xymatrix{
 ((\mX^*\fus\mX)\otimes(\mY^*\fus\mY))\otimes\vac \ar[d]^{\cA^{-1}_{\mX^*\fus\mX,\mY^*\fus\mY,\vac}} \ar[r]^(.625){(i_\mX\otimes i_{\mY})\otimes\Id_\vac} & (\aA\otimes\aA)\otimes\vac \ar[d]^{\cA_{\aA,\aA,\vac}^{-1}} \\
 (\mX^*\fus\mX)\otimes((\mY^*\fus\mY)\otimes\vac) \ar[d]^{\Id_{\mX^*\fus\mX}\otimes\lambda_{\mY^*\fus\mY}^{-1}(i_\mY)} \ar[r]^(.625){i_\mX\otimes(i_\mY\otimes\Id_\vac)} & \aA\otimes(\aA\otimes\vac) \ar[d]^{\Id_\aA\otimes\lambda_\aA^{-1}(\Id_\aA)}\\
 (\mX^*\fus\mX)\otimes\vac \ar[d]^{\lambda^{-1}_{\mX^*\fus\mX}(i_{\mX})} \ar[r]^(.6){i_\mX\otimes\Id_\vac} & \aA\otimes\vac \ar[ld]^{\lambda_\aA^{-1}(\Id_\aA)} \\
 \vac \\
 }
\end{equation*}
where $i_\mX$ and $i_\mY$ represent the obvious inclusions. The diagram commutes by naturality of associativity and $\lambda$. Moreover, because $\lambda$ is an isomorphism, $\lambda_{\mX^*\fus\mX}^{-1}(i_\mX)$ is a non-zero morphism in
\begin{equation*}
 \HHom_{\cC}((\mX^*\otimes\vac)\otimes\mX, \vac),
\end{equation*}
a one-dimensional space spanned by
\begin{equation*}
 d_\mX: (\mX^*\otimes\vac)\otimes\mX\xrightarrow{r_{\mX^*}\otimes\Id_\mX} \mX^*\otimes\mX\xrightarrow{e_{\mX}} \vac,
\end{equation*}
where $e_\mX$ is the evaluation morphism. Hence $\lambda_{\mX^*\fus\mX}^{-1}(i_\mX)=a_\mX d_\mX$ for some $a_\mX\neq 0$, and similarly $\lambda_{\mY^*\fus\mY}^{-1}(i_\mY) =a_\mY d_\mY$ for non-zero $a_\mY$.

From this discussion, it follows that 
\begin{equation*}
 m\vert_{(\mX^*\fus\mX)\otimes(\mY^*\fus\mY)} =a_\mX a_\mY\lambda_{(\mX^*\fus\mX)\otimes(\mY^*\fus\mY)}\left(d_\mX\circ(\Id_{\mX^*\fus\mX}\otimes d_\mY)\circ\cA^{-1}_{\mX^*\fus\mX,\mY^*\fus\mY,\vac}\right),
\end{equation*}
and now the morphism inside parentheses is a morphism from
\begin{equation*}
 ((\mX^*\fus\mX)\otimes(\mY^*\fus\mY))\otimes\vac = ((\mX^*\otimes\mY^*)\fus(\mX\otimes_\text{op}\mY))\otimes\vac =((\mX^*\otimes\mY^*)\otimes\vac)\otimes(\mY\otimes\mX)
\end{equation*}
to $\vac$. In fact, we can identify $\mX^*\otimes\mY^*=(\mY\otimes\mX)^*$, and then this morphism is simply $d_{\mY\otimes\mX}=e_{\mY\otimes\mX}\circ(r_{\mX^*\otimes\mY^*}\otimes\Id_{\mY\otimes\mX})$ (see for instance \cite{BK, T, EGNO}). Now since $\cC$ is semisimple, we have an isomorphism $\mY\otimes\mX\cong\bigoplus_{i\in I} \mZ_i$ where the $\mZ_i$ are simple objects of $\cC$ and $I$ is a finite index set. Under this isomorphism, $d_{\mY\otimes\mX}$ is identified with the direct sum of the $d_{\mZ_i}$ which in turn is identified with non-zero multiples of the inclusions $\mZ_i^*\otimes\mZ_i\hookrightarrow\aA$ under $\lambda_{\mZ_i}$. Hence under the natural isomorphism
\begin{align*}
 \HHom_{\cC\fus\cCrev}\left((\mX^*\fus\mX)\otimes(\mY^*\fus\mY),\aA\right)\cong\prod_{i,j\in I}\HHom_{\cC\fus\cCrev}\left(\mZ_i^*\fus\mZ_j ,\aA\right),
\end{align*}
$m\vert_{(\mX^*\fus\mX)\otimes(\mY^*\fus\mY)}$ is sent to the product over $i\in I$ of non-zero multiples of the inclusions of $\mZ_i^*\fus\mZ_i$ into $\aA$. We conclude that $\mZ^*\fus\mZ$ is included in $m((\mX^*\fus\mX)\otimes(\mY^*\fus\mY))$ precisely when $\mZ$ occurs as a direct summand of $\mY\otimes\mX$ (or equivalently $\mX\otimes\mY$ since $\cC$ is braided). A slightly different explanation for this observation is given in \cite[Example 7.9.14]{EGNO}.

We have shown that if $\cC$ is rigid, then the multiplication rules of $\aA$ satisfy $M_{\mX^* \boxtimes \mX, \mY^* \boxtimes \mY}^{\mZ^* \boxtimes \mZ}=1$ if and only if $\mZ$ is a summand of $\mX\otimes\mY$. We can use this to show that $\aA$ is a simple algebra: suppose $\mI\hookrightarrow\aA$ is a non-zero ideal of $\aA$. Because $\aA$ is semisimple in $(\cC\boxtimes\cCrev)_\oplus$ and all simple subobjects of $\aA$ occur with multiplicity $1$, any subobject such as $\mI$ is also semisimple and a direct sum of certain $\mX^*\fus\mX$. For such $\mX$, we have $\mX\cong\mX^{**}$ because $\cC$ is braided (recall Remark \ref{rema:deltaX}) and therefore $M^{\vac\fus\vac}_{\mX^{**}\fus\mX^*, \mX^*\fus\mX}=1$. This means $\vac\fus\vac\subseteq\mI$, and then $M^{\mY^*\fus\mY}_{\mY^*\fus\mY, \vac\fus\vac}=1$ for any $\mY\in\irr(\cC)$ implies $\mI=\aA$.

We summarize the results of this section:
\begin{thm}\label{thm:canonicalalgebra}
Let $\cC$ be a (not necessarily finite) semisimple braided tensor category with a contragredient functor. Then 
\[
\aA = \bigoplus_{\mX \in \irr(\cC)} \mX' \boxtimes \mX
\]
is a commutative associative algebra in $(\cC\boxtimes \cCrev)_{\oplus}^{fin}$. If $\cC$ is rigid, then $\aA$ is simple and for simple objects $\mX, \mY, \mZ$ of $\cC$, the multiplication rules are given by $M_{\mX^* \boxtimes \mX, \mY^* \boxtimes \mY^*}^{\mZ^* \boxtimes \mZ}=1$
if and only if $\mZ$ is a summand of $\mX\otimes\mY$.
\end{thm}

\begin{defn}
The algebra constructed in this subsection is called the {\bf canonical algebra} in $\cC\boxtimes \cCrev$ (equivalently, in $\cC\fus\cCop$). 
\end{defn}

\begin{rema}\label{rema:canonicalalgebra}
 Since commutative algebras are preserved by braided tensor equivalences, we can restate Theorem \ref{thm:canonicalalgebra} as follows. Let $\cC$ be a semisimple braided tensor category with a contragredient functor, and suppose $\tau:\cC\rightarrow\cD$ is a braid-reversed tensor equivalence (so that $\tau: \cCrev\rightarrow\cD$ is a braided equivalence). Then
 \begin{equation*}
  \aA=\bigoplus_{\mX\in\irr(\cC)} \mX'\fus\tau(\mX)
 \end{equation*}
is a commutative associative algebra in $(\cC\fus\cD)_\oplus^{fin}$, and if $\cC$ is rigid, then $\aA$ is simple.
\end{rema}

\section{From algebras to braid-reversed equivalences}

In the previous section, we showed how to construct a commutative associative algebra from a braid-reversed tensor equivalence. In this section, we consider the converse problem: given a simple algebra $\aA$ in the Deligne product of two braided tensor categories, obtain a braid-reversed equivalence between the two factors of the Deligne product. Such a braid-reversed equivalence was obtained in \cite{lin} under the strong assumptions that the two braided tensor categories are modular (in particular, finite) and that $\repA$ is semisimple. Here we obtain the equivalence without any finiteness assumptions and without any semisimplicity assumption on $\repA$.

\subsection{Mirror equivalence}\label{sec:mirrorduality}

In this section, we work in the following setting:
\begin{enumerate}
\item $\cU$ is a (not necessarily finite) 
semisimple ribbon category, and $\lbrace\mU_i\rbrace_{i\in I}$ is a subset of distinct simple objects in $\cU$ that includes $\vac_\cU$. We use the notation $\mU_0=\vac_\cU$.

\item $\cW$ is a ribbon category. In particular, both $\cU$ and $\cW$ are rigid.

\item We have a (commutative, associative, unital) algebra 
\begin{align*}
\aA = \bigoplus_{i\in I} \mM_{ii} 
= \bigoplus_{i\in I} \mU_{i}\boxtimes \mW_i.
\end{align*}
in $\cC=\cU\fus\cW$, or $\cC_\oplus$ if $I$ is infinite, where the $\mW_i$ are objects of $\cW$, not all assumed to be simple, with $\mW_0=\vac_\cW$. Thus $\mM_{00}=\vac_\cU\fus\vac_\cW=\vac_\cC$, which we will denote by $\vac$.

\item  The tensor units $\one_\cU=\mU_0, \one_\cW=\mW_0$ form a mutually commuting (or dual) pair in $\aA$, in the sense that
\begin{equation*}
\dim\Hom{\cU}{\Mod{U}_0, \Mod{U}_i} =\delta_{i,0} = \dim\Hom{\cW}{\Mod{V}_0, \Mod{V}_i}.
\end{equation*}
Note that the first equality is automatic because the $\mU_i$ are simple and distinct.

\item There is a partition $I=I^0\sqcup I^1$ of the index set with $0\in I^0$ such that for each $i\in I^j$, $j=0,1$, the twist satisfies $\theta_\aA\vert_{\mM_{ii}} =(-1)^j\Id_{\mM_{ii}}$. In particular, $\theta_{\aA}^2=\Id_\aA$.

\item Finally, $\aA$ is simple as an object of $\repA$. 

\end{enumerate}

Note that although we are not assuming $\cC=\cU\fus\cW$ is semisimple or finite, the conclusion of Proposition \ref{Cont_in_Cplus} still holds for the cases $\mY=\aA,\aA^*$ because we are assuming $\cU$ is semisimple and the $\mU_i$ are distinct. This means that we can use Corollary \ref{Adualsimple} together with the final assumption on $\aA$ to conclude that $\aA^*$ is also simple as a right $\aA$-module (and in fact simple in $\repA$ since $\aA$ is commutative).

As a consequence of the assumption that $\vac_\cU$ and $\vac_\cW$ form a dual pair in $\aA$, we have
\begin{align*}
 \HHom_{\cC_\oplus}(\vac,\aA) & = \bigoplus_{i\in I} \HHom_{\cC}(\vac_\cU\fus\vac_\cW,\mU_i\fus\mW_i)\nonumber\\
 & =\bigoplus_{i\in I} \HHom_\cU(\vac_\cU,\mU_i)\otimes_\mathbb{F} \HHom_\cW(\vac_\cW, \mW_i)\nonumber\\
 & =\HHom_{\cC}(\vac, \mU_0\fus\mW_0),
\end{align*}
which is the one-dimensional space $\mathbb{F}=\mathrm{End}_{\cC}(\vac)$. Thus we may take $\iota_\aA$ to be the canonical injection of $\vac=\mU_0\fus\mW_0$ into the direct sum. We then define $\varepsilon_\aA: \aA\rightarrow\vac$ to be the canonical projection with respect to the direct sum decomposition of $\aA$, so that $\varepsilon_\aA\circ\iota_\aA=\Id_\vac$.

Let $\cU_{\alg{A}} \subseteq \cU$ and $\cW_{\alg{A}} \subseteq \cW$ denote the full subcategories consisting of objects isomorphic to direct sums of objects appearing in the decomposition of $\aA$, that is, of the $\mU_i$ and $\mW_i$, respectively. Our main theorem will be a braid-reversed tensor equivalence between $\cU_\aA$ and $\cW_\aA$, although we have not yet shown  that they are tensor categories. The key idea is to use the induction functor $\cF: \cU\fus\cW\rightarrow\repA$ to identify $\cU_\aA$ and $\cW_\aA$ with a common subcategory of $\repA$. More specifically, we will use the two tensor functors
\begin{align*}
\begin{array}{rl}
 \cF_\cU:\cU&\rightarrow \repA\\
\mX&\mapsto \cF(\Mod{\mX} \boxtimes \one_\cW)\\
f &\mapsto \cF(f \boxtimes \Id_{\one_\cW})\\
\end{array}\qquad
\begin{array}{rl}
 \cF_\cW:\cW&\rightarrow \repA\\
\mY&\mapsto \cF(\one_\cU \boxtimes \Mod{\mY})\\
g &\mapsto \cF(\Id_{\one_\cU} \boxtimes  g)\\
\end{array}
\end{align*} 
First we show that $\cF_\cU$ and $\cF_\cW$ are fully faithful, so that $\cU$ and $\cW$ are tensor equivalent to subcategories of $\repA$:
\begin{lemma}\label{lemma:FU_FW_fully_faithful}
The functors $\cF_\cU$ and $\cF_\cW$ are fully faithful.
\end{lemma}
\begin{proof}
We prove that $\cF_{\mathcal{U}}$ is fully faithful; the proof for $\cF_{\cW}$ is the same (in particular, the proof does not use semisimplicity). We need to show that for any objects $\mX_1$, $\mX_2$ in $\mathcal{U}$, the linear map
\begin{equation*}
 \cF_{\mathcal{U}}: \Hom{\mathcal{U}}{\mX_1,\mX_2}\rightarrow\Hom{\repA}{\aA\otimes(\mX_1\fus\vac_{\mathcal{V}}), \aA\otimes(\mX_2\fus\vac_{\mathcal{V}})}
\end{equation*}
is an isomorphism. We first observe that since in the Deligne product $\cC=\mathcal{U}\fus\mathcal{V}$ we have the isomorphism
\begin{align*}
 \Hom{\mathcal{U}}{\mX_1,\mX_2}\otimes_\mathbb{F}\Hom{\mathcal{V}}{\vac_\cW,\vac_\cW} & \xrightarrow{\cong}\Hom{\cC}{\mX_1\fus\vac_\cW, \mX_2\fus\vac_\cW}
\end{align*}
given by $f\otimes_\mathbb{F} g \mapsto f\fus g$, and since $\Hom{\mathcal{V}}{\vac_\cW,\vac_\cW}=\mathbb{F}\Id_{\vac_\cW}$, it is sufficient to show that
\begin{equation*}
 \cF: \Hom{\cC}{\mX_1\fus\vac_\cW,\mX_2\fus\vac_\cW}\rightarrow\Hom{\repA}{\aA\otimes(\mX_1\fus\vac_{\mathcal{V}}), \aA\otimes(\mX_2\otimes\vac_{\mathcal{V}})}
\end{equation*}
is an isomorphism.

We show that $\cF$ is an isomorphism by constructing an inverse: given $F: \aA\otimes(\mX_1\fus\vac_{\mathcal{V}})\rightarrow \aA\otimes(\mX_2\otimes\vac_{\mathcal{V}})$ in $\repA$, we define $\cG(F): \mX_1\fus\vac_\cW \rightarrow \mX_2\fus\vac_\cW$ in $\cC$ to be the composition
\begin{align*}
 \mX_1\fus\vac_\cW\xrightarrow{l_{\mX_1\fus\vac_\cW}^{-1}} \vac\otimes(\mX_1\fus\vac_\cW) & \xrightarrow{\iota_\aA\otimes \Id_{\mX_1\fus\vac_\cW}} \aA\otimes(\mX_1\fus\vac_\cW)\xrightarrow{F} \aA\otimes(\mX_2\fus\vac_\cW)\nonumber\\
 &\xrightarrow{\varepsilon_\aA\otimes\Id_{\mX_2\fus\vac_\cW}} \vac\otimes(\mX_2\fus\vac_\cW)\xrightarrow{l_{\mX_2\fus\vac_\cW}} \mX_2\fus\vac_\cW.
\end{align*}
For $f\in\Hom{\cC}{\mX_1\fus\vac_\cW, \mX_2\fus\vac_\cW}$, it is easy to see that $\cG(\cF(f))=f$: in fact, $\cG(\cF(f))$ is the composition
\begin{equation*}
 \mX_1\fus\vac_\cW\xrightarrow{l_{\mX_1\fus\vac_\cW}^{-1}} \vac\otimes(\mX_1\fus\vac_\cW)\xrightarrow{(\varepsilon_\aA\circ\iota_\aA)\otimes f} \vac\otimes(\mX_2\fus\vac_\cW)\xrightarrow{l_{\mX_2\fus\vac_\cW}} \mX_2\fus\vac_{\cW}.
\end{equation*}
Using $\varepsilon_\aA\circ\iota_\aA=\Id_{\vac}$ and the naturality of the left unit isomorphisms, this reduces to $f$.

On the other hand, given $F:\cF(\mX_1\otimes\vac_\cW)\rightarrow\cF(\mX_2\otimes\vac_\cW)$ in $\repA$, $\cF(\cG(F))$ is the composition
\begin{align}\label{FGofF}
 \aA\otimes(\mX_1\fus\vac_\cW) & \xrightarrow{\Id_\aA\otimes l_{\mX_1\fus\vac_\cW}^{-1}} \aA\otimes(\vac\otimes(\mX_1\fus\vac_\cW))\xrightarrow{\Id_\aA\otimes(\iota_\aA\otimes\Id_{\mX_1\fus\vac_\cW})} \aA\otimes(\aA\otimes(\mX_1\fus\vac_\cW))\nonumber\\
 &\xrightarrow{\Id_\aA\otimes F} \aA\otimes(\aA\otimes(\mX_2\fus\vac_\cW))\xrightarrow{\Id_\aA\otimes(\varepsilon_\aA\otimes\Id_{\mX_2\fus\vac_\cW})} \aA\otimes(\vac\otimes(\mX_2\fus\vac_\cW))\nonumber\\
 &\xrightarrow{\Id_\aA\otimes l_{\mX_2\fus\vac_\cW}} \aA\otimes(\mX_2\fus\vac_\cW).
\end{align}

We first use the triangle axiom, the right unit property of $\aA$, and the naturality of the associativity isomorphisms to rewrite
\begin{align*}
 \Id_\aA\otimes l_{\mX_2\fus\vac_\cW} & = (r_\aA\otimes\Id_{\mX_2\fus\vac_\cW})\circ\cA_{\aA,\vac, \mX_2\fus\vac_\cW}\nonumber\\
 & =(\mu_\aA\fus\Id_{\mX_2\fus\vac_\cW})\circ((\Id_\aA\otimes\iota_\aA)\otimes\Id_{\mX_2\fus\vac_\cW})\circ\cA_{\aA,\vac,\mX_2\fus\vac_\cW}\nonumber\\
 & =(\mu_\aA\fus\Id_{\mX_2\fus\vac_\cW})\circ\cA_{\aA,\aA,\mX_2\fus\vac_\cW}\circ(\Id_\aA\otimes(\iota_\aA\otimes\Id_{\mX_2\fus\vac_\cW})).
\end{align*}
Thus \eqref{FGofF} becomes
\begin{align}\label{FGofF2}
 \aA\otimes(\mX_1\fus\vac_\cW) & \xrightarrow{\Id_\aA\otimes\widetilde{F}} \aA\otimes(\aA\otimes(\mX_2\fus\vac_\cW))\xrightarrow{\Id_\aA\otimes((\iota_\aA\circ\varepsilon_\aA)\otimes\Id_{\mX_2\fus\vac_\cW})} \aA\otimes(\aA\otimes(\mX_2\fus\vac_\cW))\nonumber\\
 & \xrightarrow{\cA_{\aA,\aA,\mX_2\fus\vac_\cW}} (\aA\otimes \aA)\otimes(\mX_2\fus\vac_\cW)\xrightarrow{\mu_\aA\otimes\Id_{\mX_2\fus\vac_\cW}} \aA\otimes(\mX_2\fus\vac_\cW)
\end{align}
where $\widetilde{F}=F\circ(\iota_\aA\otimes\Id_{\mX_1\fus\vac_\cW}))\circ l_{\mX_1\fus\vac_\cW}^{-1}$. Now we use the assumption that $\vac_\cU$ and $\vac_\cW$ form a dual pair inside $\aA$ to observe that
\begin{align*}
 \widetilde{F}\in\, & \Hom{\cC}{\mX_1\fus\vac_\cW, \aA\otimes(\mX_2\fus\vac_\cW)}\nonumber\\
 & \cong\Hom{\cC}{\mX_1\fus\vac_\cW, \bigoplus_{i\in I} (\mU_i\fus \mW_i)\otimes(\mX_2\fus\vac_\cW)}\nonumber\\
 &\cong\bigoplus_{i\in I} \Hom{\cC}{\mX_1\fus\vac_\cW, (\mU_i\otimes \mX_2)\fus(\mW_i\otimes\vac_\cW)}\nonumber\\
 &\cong \bigoplus_{i\in I} \Hom{\cU}{\mX_1,\mU_i\otimes \mX_2}\otimes_\mathbb{F} \Hom{\cW}{\vac_\cW, \mW_i}\nonumber\\
 & \cong \Hom{\cU}{\mX_1, \vac_\cU\otimes \mX_2}\otimes_\mathbb{F} \Hom{\cW}{\vac_\cW,\vac_\cW}\nonumber\\
 & \cong\Hom{\cC}{\mX_1\fus\vac_\cW, \vac\otimes(\mX_2\fus\vac_\cW)}.
\end{align*}
In other words, the image of $\widetilde{F}$ inside $\aA\otimes(\mX_2\fus\vac_\cW)$ is contained in $\iota_\aA(\vac)\otimes(\mX_2\fus\vac_\cW)$. Since $\iota_\aA\circ\varepsilon_\aA$ is the projection from $\aA$ to $\iota_\aA(\vac)\subseteq \aA$, it follows that $((\iota_\aA\circ\varepsilon_\aA)\fus\Id_{\mX_2\fus\vac_W})\circ\widetilde{F}=\widetilde{F}$. Consequently, \eqref{FGofF2} becomes
\begin{align*}
 \aA & \otimes(\mX_1\fus\vac_{\cW})  \xrightarrow{\Id_\aA\otimes l_{\mX_1\fus\vac_\cW}^{-1}} \aA\otimes(\vac\otimes(\mX_1\fus\vac_\cW))\xrightarrow{\Id_\aA\otimes(\iota_\aA\otimes\Id_{\mX_1\fus\vac_\cW})} \aA\otimes(\aA\otimes(\mX_1\fus\vac_\cW))\nonumber\\
& \xrightarrow{\Id_\aA\otimes F} \aA\otimes(\aA\otimes(\mX_2\fus\vac_\cW))\xrightarrow{\cA_{\aA,\aA,\mX_2\fus\vac_\cW}} (\aA\otimes \aA)\otimes(\mX_2\fus\vac_\cW)\xrightarrow{\mu_\aA\otimes\Id_{\mX_2\fus\vac_\cW}} \aA\otimes(\mX_2\fus\vac_\cW).
\end{align*}
Now because $F$ is a morphism in $\repA$ and because $\mu_{\cF(\mX_i\fus\vac_\cW)}=(\mu_\aA\otimes\Id_{\mX_i\fus\vac_\cW})\circ\cA_{\aA,\aA,\mX_i\fus\vac_\cW}$ for $i=1,2$, we get
\begin{align*}
 \aA & \otimes(\mX_1\fus\vac_{\cW})  \xrightarrow{\Id_\aA\otimes l_{\mX_1\fus\vac_\cW}^{-1}} \aA\otimes(\vac\otimes(\mX_1\fus\vac_\cW))\xrightarrow{\Id_\aA\otimes(\iota_\aA\otimes\Id_{\mX_1\fus\vac_\cW})} \aA\otimes(\aA\otimes(\mX_1\fus\vac_\cW))\nonumber\\
& \xrightarrow{\cA_{\aA,\aA,\mX_1\fus\vac_\cW}} (\aA\otimes \aA)\otimes(\mX_1\fus\vac_\cW)\xrightarrow{\mu_\aA\otimes\Id_{\mX_1\fus\vac_\cW}} \aA\otimes(\mX_1\fus\vac_\cW)\xrightarrow{F} \aA\otimes(\mX_2\fus\vac_\cW).
\end{align*}
Finally, the naturality of the associativity isomorphisms, the triangle axiom, and the right unit property of $\mu_\aA$ imply that this composition equals $F$, as required.
\end{proof}

Now the following result is key for showing that $\cU_\aA$ and $\cW_\aA$ are tensor subcategories and $\cF_\cU(\cU_\aA)=\cF_\cW(\cW_\aA)$. It was proved in \cite{lin} under the assumption that $\cU$, $\cW$, and $\repA$ are all semisimple; but even if $\cU$ and $\cW$ are modular tensor categories, $\repA$ is not guaranteed to be semisimple. By \cite[Theorem 3.3]{KO}, $\repA$ is semisimple when $\aA$ is a rigid algebra in $\cC$, which by \cite[Lemma 1.20]{KO} means $\aA$ is simple \textit{and} $\dim_\cC \aA\neq 0$. Here we remove the assumption $\dim_\cC \aA\neq 0$. Because the proof is lengthy and requires some preparatory lemmas, we will defer it to Section \ref{sec:keylemma}.
\begin{keylemma}\label{lemma:keyiso}
If $\mU_i$ is simple and $\dim_{\cU} {\mU_i}\neq 0$, then $\cF({\mM_{i0}})\cong\cF(\mM_{0i}^\ast)$
in $\repA$.
\end{keylemma}

\begin{remark}\label{rem:keylemmaassumptions}
In the setting of this section, with the $\mU_i$ simple objects in a semisimple tensor category, $\dim_\cU \mU_i\neq 0$ is automatic by \cite[Proposition 4.8.4]{EGNO}. But as the proof of Key Lemma \ref{lemma:keyiso} does not use the semisimplicity of $\cU$, we have chosen to specify more precisely what conditions are needed for the result to hold. 
\end{remark}

As a consequence of the Key Lemma, we will also prove the following in Section \ref{sec:keylemma}:
\begin{prop}\label{prop:UA_WA_ribbon}
 The categories $\cU_\aA\subseteq\cU$ and $\cW_\aA\subseteq\cW$ are ribbon subcategories. Moreover, $\cW_\aA$ is semisimple with distinct simple objects $\lbrace\mW_i\rbrace_{i\in I}$.
\end{prop}

Key Lemma \ref{lemma:keyiso} and Proposition \ref{prop:UA_WA_ribbon} already show that $\cF_\cU(\cU_\aA)=\cF_\cW(\cW_\aA)$, so that by Lemma \ref{lemma:FU_FW_fully_faithful}, $\cU_\aA$ and $\cW_\aA$ are tensor equivalent. However, to show that $\cU_\aA$ and $\cW_\aA$ are braid-reversed equivalent, we will need to lift to the center. Let $\cF_{\cU_{\aA}}$ and $\cF_{\cW_{\aA}}$ denote the restrictions of $\cF_\cU$ and $\cF_\cW$ to $\cU_\aA$ and $\cW_\aA$, respectively, and let $\cM=\cF_{\cU_\aA}(\cU_\aA)\subseteq\repA$. By Key Lemma \ref{lemma:keyiso} and Proposition \ref{prop:UA_WA_ribbon}, we also have $\cM=\cF_{\cW_\aA}(\cW_\aA)$. Then Lemma \ref{lemma:FU_FW_fully_faithful} and the discussion in Section \ref{sec:center} show that $\cF_{\cU_\aA}$ and $\cF_{\cW_\aA}$ are central functors that lift to braided tensor functors
\begin{equation*}
 \cG_{\cU_\aA}: \cU_\aA\rightarrow \cZ(\cM),\qquad\cG_{\cW_\aA}: \cW_\aA\rightarrow\cZ(\cM),
\end{equation*}
using the braidings on $\cU_\aA$ and $\cW_\aA$ (see Section \ref{sec:center} for the precise definitions), as well as
\begin{equation*}
 \cG^{\rev}_{\cU_\aA}: \cU^{\rev}_\aA\rightarrow \cZ(\cM),\qquad\cG^{\rev}_{\cW_\aA}: \cW^{\rev}_\aA\rightarrow\cZ(\cM)
\end{equation*}
using the inverse braidings.

With this setup, we can now conclude our main theorem:

\begin{thm}\label{thm:main_bequiv}
In the setting of this section, there is a braid-reversed tensor equivalence $\tau:\cU_\aA\rightarrow\cW_\aA$ such that $\tau(\mU_i)\cong\mW_i^*$ for $i\in I$.
\end{thm}
\begin{proof}
By the universal property of Deligne products, the four fully faithful braided tensor functors $\cG^{(\rev)}_{\cU_\aA}$, $\cG^{(\rev)}_{\cW_\aA}$ combine into braided tensor functors
\begin{equation*}
 \cG_{\cU_\aA\fus\cW_\aA}: \cU_\aA\fus\cW_\aA\rightarrow\cZ(\cM),\qquad \cG_{\cU_\aA\fus\cW_\aA}^{\rev}: \cU_\aA^{\rev}\fus\cW_\aA^{\rev}\rightarrow\cZ(\cM).
\end{equation*}
Now because $\mU\fus\vac_\cW$ for $\mU\in\obj(\cU_\aA)$ and $\vac_\cU\fus\mW$ for $\mW\in\obj(\cW_\aA)$ centralize each other in $\cU_\aA\fus\cW_\aA$ (this follows from $\cR_{\mU\fus\vac_{\cW}, \vac_{\cU}\fus\mW} =\cR_{\mU,\vac_{\cU}}\fus\cR_{\vac_\cW,\mW}$), so do their images in $\cZ(\cM)$ under $\cG^{(\rev)}_{\cU_\aA\fus\cW_\aA}$. Thus
\begin{equation*}
\begin{split}
\cG_{\cU_{\aA}}(\cU_{\aA}) = \cG_{\cU_{\aA}\boxtimes \cW_{\aA}}(\cU_{\aA})  &\subseteq  \cG_{\cU_{\aA}\boxtimes \cW_{\aA}}(\cW_{\aA})' = \cG_{\cW_{\aA}}(\cW_{\aA})',\\
\cG_{\cW_{\aA}}(\cW_{\aA}) = \cG_{\cU_{\aA}\boxtimes \cW_{\aA}}(\cW_{\aA})  &\subseteq  \cG_{\cU_{\aA}\boxtimes \cW_{\aA}}(\cU_{\aA})' = \cG_{\cU_{\aA}}(\cU_{\aA})',\\
\cG^{\rev}_{\cU_{\aA}}(\cU_{\aA}^{\rev}) = \cG^{\rev}_{\cU_{\aA}\boxtimes \cW_{\aA}}(\cU_{\aA}^{\rev})  &\subseteq  \cG^{\rev}_{\cU_{\aA}\boxtimes \cW_{\aA}}(\cW_{\aA}^{\rev})' = \cG^{\rev}_{\cW_{\aA}}(\cW_{\aA}^{\rev})',\\
\cG^{\rev}_{\cW_{\aA}}(\cW_{\aA}^{\rev}) = \cG^{\rev}_{\cU_{\aA}\boxtimes \cW_{\aA}}(\cW_{\aA}^{\rev})  &\subseteq \cG^{\rev}_{\cU_{\aA}\boxtimes \cW_{\aA}}(\cU_{\aA}^{\rev})' = \cG^{\rev}_{\cU_{\aA}}(\cU_{\aA}^{\rev})'.
\end{split}
\end{equation*}
It follows using Proposition \ref{prop:mueger}, that is, \cite[Proposition 7.3]{Mue}, that
\[
 \cG_{\cU_{\aA}}(\cU_{\aA}) \subseteq \cG_{\cW_{\aA}}(\cW_{\aA})' = \cG^{\text{rev}}_{\cW_{\aA}}(\cW_{\aA}^{\rev}) \qquad \text{and} \qquad \cG^{\text{rev}}_{\cW_{\aA}}(\cW_{\aA}^{\rev}) \subseteq \cG^{\text{rev}}_{\cU_{\aA}}(\cU_{\aA}^{\rev})'=\cG_{\cU_{\aA}}(\cU_{\aA}),
\]
that is, $\cG_{\cU_\aA}(\cU_\aA)=\cG_{\cW_\aA}^{\rev}(\cW_\aA^{\rev})$. Hence we can get a braided tensor equivalence $\tau: \cU_\aA\rightarrow\cW_\aA^{\rev}$ by composing $\cG_{\cU_\aA}$ with an inverse to $\cG_{\cW_\aA}^{\rev}$. This is the same thing as a braid-reversed tensor equivalence $\tau: \cU_\aA\rightarrow\cW_\aA$, and $\tau(\mU_i)\cong\mW_i^*$ follows directly from Key Lemma \ref{lemma:keyiso}.
\end{proof}

\subsection{Proof of Key Lemma \ref{lemma:keyiso} and Proposition \ref{prop:UA_WA_ribbon}}\label{sec:keylemma}

We use the same notation as in the previous subsection. For future use, we shall prove Key Lemma \ref{lemma:keyiso} and Proposition \ref{prop:UA_WA_ribbon} in a slightly more general setting than that of the previous subsection: we allow $\aA$ to be a superalgebra in $\cC_\oplus$. More specifically, suppose that $I=I^0\sqcup I^1$ is a partition of the index set with $0\in I^0$, and set $\aA^0=\bigoplus_{i\in I^0} \mM_{ii}$ and $\aA^1=\bigoplus_{i\in I^1} \mM_{ii}$; assume that $\mu_\aA$ is an even morphism, that is,
\begin{equation*}
 \mu_\aA(\aA^i\otimes\aA^j)\subseteq \aA^{i+j}
\end{equation*}
for $i,j\in\lbrace 0,1\rbrace$, interpreting $i+j$ mod $2\ZZ$. Following \cite{CKL}, we say that $\aA$ is:
\begin{enumerate}
 \item An \textbf{algebra of correct statistics} if $\aA$ is a commutative algebra and the twist satisfies $\theta_\aA=\Id_\aA$ (in this case we may assume $I^1=\emptyset$),
 \item An \textbf{algebra of wrong statistics} if $\aA$ is a commutative algebra and $\theta_\aA=\Id_{\aA^0}\oplus(-\Id_{\aA^1})$,
 \item A \textbf{superalgebra of correct statistics} if $\mu_{\aA}\vert_{\aA^j\otimes\aA^i}\circ\cR_{\aA^i,\aA^j} = (-1)^{ij}\mu_\aA\vert_{\aA^i\otimes\aA^j}$ for $i,j\in\lbrace 0,1\rbrace$ and $\theta_\aA=\Id_{\aA^0}\oplus(-\Id_{\aA^1})$, and
 \item A \textbf{superalgebra of wrong statistics} if $\mu_{\aA}\vert_{\aA^j\otimes\aA^i}\circ\cR_{\aA^i,\aA^j} = (-1)^{ij}\mu_\aA\vert_{\aA^i\otimes\aA^j}$ for $i,j\in\lbrace 0,1\rbrace$ and $\theta_\aA=\Id_\aA$.
\end{enumerate}
Note that both commutative algebras and superalgebras are monodromy-free: $\mu_{\aA}\circ\cR_{\aA,\aA}^2=\mu_\aA$. While in the previous subsection we only needed to assume $\aA$ simple as an object of $\repA$, here we will also assume $\aA$ is simple as a right $\aA$-module.

\begin{lemma}\label{lem:AiBiduals}
Assume $\aA$ is simple both as a left and right $\aA$-module. Then there is an involution $i\mapsto i'$ of the index set $I$ such that $\mU_{i'}\cong \mU_i^\ast$  and $\mW_{i'}\cong \mW_i^\ast$. Moreover, there is a unique isomorphism 
$\varphi_i: \mM_{i'i'}\rightarrow \mM_{ii}^\ast$ for each $i\in I$
 such that the diagram
 \begin{equation*}
\xymatrixcolsep{4pc}
\xymatrix{
	\mM_{i'i'}\otimes \mM_{ii} \ar[d]^{\varphi_i\otimes \Id} \ar[r]^(.62){\mu_\aA} & \aA \ar[d]^{\varepsilon_\aA} \\
	\mM_{ii}^\ast\otimes \mM_{ii} \ar[r]^(.6){e_{\mM_{ii}}} & \one 
}
\end{equation*}
commutes, where $e_{\mM_{ii}}=e_{\mU_i}\boxtimes e_{\mW_i}$ is the coevaluation in $\cC$.
\end{lemma}
\begin{proof}
The morphism $\varepsilon_\aA\circ\mu_\aA: \aA\otimes\aA\rightarrow\vac$ induces $\varphi=\Gamma_{\aA,\aA}(\varepsilon_\aA\circ\mu_\aA)$, the unique morphism in $\HHom_{\cC}(\aA,\aA^*)$ making the diagram
\begin{equation}\label{rigidalgdiag}
\begin{matrix}
  \xymatrixcolsep{4pc}
  \xymatrix{
  \aA\otimes \aA \ar[d]^{\varphi\otimes \Id} \ar[r]^(.6){\mu_\aA} & \aA \ar[d]^{\varepsilon_\aA} \\
  \aA^\ast\otimes \aA \ar[r]^(.61){e_\aA} & \one\\
  }
  \end{matrix}
 \end{equation}
commute. Since $\varepsilon_{\aA}\circ\mu_{\aA}$ is non-zero by the unit property of $\aA$, $\varphi$ is also non-zero. We will show that $\varphi$ is a homomorphism of right $\aA$-modules. Then since $\aA$ is simple in $\repA$, $\aA^*$ is a simple right $\aA$-module by Corollary \ref{Adualsimple} and it will follow that $\varphi$ is an isomorphism of right $\aA$-modules.

To show that $\varphi$ is a right $\aA$-module homomorphism, we use the diagram
\begin{equation*}
 \xymatrixcolsep{4pc}
 \xymatrix{
  & (\aA\otimes\aA)\otimes\aA \ar[ld]_{(\varphi\otimes\Id_\aA)\otimes\Id_\aA\quad} \ar[d]^{\cA^{-1}_{\aA,\aA,\aA}} \ar[r]^(.57){\mu_\aA\otimes\Id_\aA} & \aA\otimes\aA \ar[rdd]^{\mu_\aA} & \\
 (\aA^*\otimes\aA)\otimes\aA \ar[d]^{\mu_{\aA^*}\otimes\Id_\aA} \ar[rd]^{\cA^{-1}_{\aA^*,\aA,\aA}} & \aA\otimes(\aA\otimes\aA) \ar[r]^(.57){\Id_\aA\otimes\mu_\aA} \ar[d]^{\varphi\otimes\Id_{\aA\otimes\aA}} & \aA\otimes\aA \ar[rd]^{\mu_\aA} \ar[d]^{\varphi\otimes\Id_\aA} & \\
 \aA^*\otimes\aA \ar[rrd]_{e_\aA} & \aA^*\otimes(\aA\otimes\aA) \ar[r]^(.57){\Id_{\aA^*}\otimes\mu_\aA} & \aA^*\otimes\aA \ar[d]^{e_\aA} & \aA \ar[ld]^{\varepsilon_\aA} \\
  & & \vac & \\
 }
\end{equation*}
which commutes as a consequence of the definitions of $\mu_{\aA^*}$ and $\varphi$, the naturality of the associativity isomorphisms, and the associativity of $\mu_\aA$. Using the outer compositions of the diagram and the naturality of $\Gamma$, we get
\begin{align*}
 e_{\aA}\circ((\mu_{\aA^*}\circ(\varphi\otimes\Id_\aA))\otimes\Id_\aA) = \Gamma_{\aA,\aA}^{-1}(\varphi)\circ(\mu_\aA\otimes\Id_{\aA}) =\Gamma^{-1}_{\aA\otimes\aA,\aA}(\varphi\circ\mu_\aA).
\end{align*}
Applying $\Gamma_{\aA\otimes\aA,\aA}$ to both sides then yields the desired equality $\mu_{\aA^*}\circ(\varphi\otimes\Id_\aA)=\varphi\circ\mu_\aA$.

Now we have 
$\aA^\ast=\bigoplus\limits_{i\in I} \mM_{ii}^\ast$, 
and the factors $\mM_{ii}^\ast=\mU_i^\ast\boxtimes \mW_i^\ast$ are inequivalent since the  
$\mU_i$ are inequivalent. Thus if we consider $\varphi\vert_{\mM_{ii}}$ for any $i$, we see that the image must be contained in a unique $\mM_{jj}^\ast=\mU_j^\ast\boxtimes \mW_j^\ast$, and $\mU_j^\ast\cong \mU_i$. Thus if we set $j=i'$, we get the involution $i\mapsto i'$ of the index set $I$ such that $\mU_{i'}\cong \mU_i^\ast$. Moreover, $$\varphi_i:=\varphi\vert_{\mM_{i'i'}}=\varphi\vert_{\mU_{i'}\boxtimes \mW_{i'}}: 
\mM_{i'i'}=
\mU_{i'}\boxtimes \mW_{i'}\rightarrow \mU_{i}^\ast\boxtimes \mW_{i}^\ast=\mM_{ii}^\ast$$
must be an isomorphism. Since $\mU_{i'}$ is a simple object in $\cU$, we can identify $\varphi_i=f_i\boxtimes g_i$ where $f_i: \mU_{i'}\rightarrow \mU_i^\ast$ is a fixed (and unique up to scale) isomorphism and $g_i: \mW_{i'}\rightarrow \mW_i^\ast$ is a morphism. But in fact $g_i$ must be an isomorphism as well because $\varphi_i$ is an isomorphism. Thus $\mW_{i'}\cong \mW_i^\ast$ as well.

Now we restrict the commutative diagram \eqref{rigidalgdiag} to
\begin{equation*}
\mM_{i'i'}\otimes \mM_{ii} =  (\mU_{i'}\boxtimes \mW_{i'})\otimes (\mU_i\boxtimes \mW_i)\subseteq \aA\otimes \aA.
\end{equation*}
If we identify
\begin{equation*}
 \aA^\ast\otimes \aA = 
 \bigoplus_{i,j\in I} (\mU_i^\ast\boxtimes \mW_i^\ast)\otimes (\mU_j\boxtimes \mW_j)=
 \bigoplus_{i,j\in I} \mM_{ii}^\ast\otimes\mM_{jj},
\end{equation*}
then under this identification,
\begin{equation*}
 e_\aA=\sum_{i\in I} e_{\mM_{ii}}\circ p_{i,i},
\end{equation*}
where $p_{i,j}: \aA\otimes \aA\rightarrow \mM_{ii}^\ast\otimes \mM_{jj}$ denotes the projection. Consequently,
\begin{align*}
 e_\aA\circ(\varphi\otimes \Id_\aA)\vert_{\mM_{i'i'}\otimes \mM_{ii}} & = \sum_{j\in I} e_{\mM_{jj}}\circ p_{j,j}\circ(\varphi_i\otimes \Id_{\mM_{ii}})= e_{\mM_{ii}}\circ(\varphi_i\otimes 1_{\mM_{ii}}),
\end{align*}
as desired.
\end{proof}

\begin{rema}\label{i_i'_same_parity}
 Note that if $\aA$ is a (super)algebra with $I^1\neq\emptyset$ (and $0\in I^0$), Lemma \ref{lem:AiBiduals} together with the evenness of $\mu_\aA$ imply that $i\in I^1$ if and only if $i'\in I^1$. 
\end{rema}

Let us use the notation $e_{ij}: \mM_{ij}^\ast\otimes \mM_{ij}\rightarrow\vac$ to denote the evaluation $e_{\mU_i}\fus e_{\mW_j}$ in $\cC$, and similarly for coevaluations. As usual in a rigid tensor category, we identify $\vac^*=\vac$ with evaluation and coevaluation given by unit isomorphisms and their inverses (and similarly for $\vac_\cU$ and $\vac_\cW$). We will need the following simple lemma:
\begin{lemma}\label{lem:tensprodevals}
 For $i, j\in I$, the composition
 \begin{align*}
  \mM_{ij}^\ast & \otimes  \mM_{ij}  \xrightarrow{(l_{\mU_i^*}^{-1}\fus r_{\mW_j^*}^{-1})\otimes(r_{\mU_i}^{-1}\fus l_{\mW_j}^{-1})} (\mM_{0j}^\ast\otimes \mM_{i0}^\ast)\otimes(\mM_{i0}\otimes \mM_{0j})\nonumber\\
  & \xrightarrow{\cA_{\mM_{0j}^\ast\otimes \mM_{i0}^\ast, \mM_{i0}, \mM_{0j}}} ((\mM_{0j}^\ast\otimes \mM_{i0}^\ast)\otimes \mM_{i0})\otimes \mM_{0j} \xrightarrow{\cA_{\mM_{0j}^\ast,\mM_{i0}^\ast,\mM_{i0}}^{-1}\otimes \Id_{\mM_{0j}}} (\mM_{0j}^\ast\otimes(\mM_{i0}^\ast\otimes \mM_{i0}))\otimes \mM_{0j}\nonumber\\
  &\xrightarrow{ (\Id_{\mM_{0j}^*}\otimes e_{i0})\otimes \Id_{\mM_{0j}}} (\mM_{0j}^\ast\otimes\vac)\otimes \mM_{0j}\xrightarrow{r_{\mM_{0j}^*}\otimes\Id_{\mM_{0j}}} \mM_{0j}^*\otimes\mM_{0j} \xrightarrow{e_{0j}} \one
 \end{align*}
is equal to $e_{ij}=e_{\mU_i}\boxtimes e_{\mW_j}$.
\end{lemma}
\begin{proof}
This composition in $\cU\fus\cW$ is the Deligne product of a morphism in $\cU$ with a morphism in $\cW$. On the $\cU$ side we get
\begin{align*}
 \mU_{i}^*  \otimes\mU_i \xrightarrow{l_{\mU_i^*}^{-1}\otimes r_{\mU_i}^{-1}} & (\vac_\cU\otimes\mU_i^*)\otimes(\mU_i\otimes\vac_\cU) \xrightarrow{\cA_{\vac_\cU\otimes\mU_i^*,\mU_i,\vac_\cU}} ((\vac_\cU\otimes\mU_i^*)\otimes\mU_i)\otimes\vac_\cU\nonumber\\
 & \xrightarrow{\cA^{-1}_{\vac_\cU,\mU_i^*,\mU_i}\otimes\Id_{\vac_\cU}} (\vac_\cU\otimes(\mU_i^*\otimes\mU_i))\otimes\vac_\cU \xrightarrow{(\Id_{\vac_\cU}\otimes e_{\mU_i})\otimes\Id_{\vac_\cU}}(\vac_\cU\otimes\vac_\cU)\otimes\vac_\cU\nonumber\\
 & \xrightarrow{r_{\vac_\cU}\otimes\Id_{\vac_\cU}} \vac_\cU\otimes\vac_\cU\xrightarrow{l_{\vac_\cU}=r_{\vac_\cU}} \vac_\cU.
\end{align*}
By properties of unit isomorphisms, the first two arrows here equal $r^{-1}_{(\vac_\cU\otimes\mU_i^*)\otimes\mU_i}\circ(l_{\mU_i^*}^{-1}\otimes\Id_{\mU_i})$, and then we can use naturality to move this inverse right unit isomorphism to the end of the composition, where it cancels with $r_{\vac_\cU}$. Thus we get
\begin{align*}
 \mU_i^*\otimes\mU_i\xrightarrow{l_{\mU_i^*}^{-1}} (\vac_\cU\otimes\mU_i^*)\otimes\mU_i \xrightarrow{\cA_{\vac_\cU,\mU_i^*,\mU_i}^{-1}} \vac_\cU\otimes(\mU_i^*\otimes\mU_i)\xrightarrow{\Id_{\vac_\cU}\otimes e_{\mU_i}} \vac_\cU\otimes\vac_\cU\xrightarrow{r_{\vac_\cU}=l_{\vac_\cU}} \vac_\cU.
\end{align*}
Now the first two arrows are $l_{\mU_i^*\otimes\mU_i}^{-1}$, and naturality of the left unit isomorphisms implies the composition reduces to $e_{\mU_i}$, as required.

The $\cW$ side of the composition is similar:
\begin{align}\label{Wsidecomp}
 \mW_j^* \otimes & \mW_j\xrightarrow{r_{\mW_j^*}^{-1}\otimes l_{\mW_j}^{-1}}  (\mW_j^*\otimes\vac_\cW)\otimes(\vac_\cW\otimes\mW_j) \xrightarrow{\cA_{\mW_j^*\otimes\vac_\cW,\vac_\cW,\mW_j}} ((\mW_j^*\otimes\vac_\cW)\otimes\vac_\cW)\otimes\mW_j\nonumber\\
 & \xrightarrow{\cA^{-1}_{\mW_j^*,\vac_\cW,\vac_\cW}\otimes\Id_{\mW_j}}  (\mW_j^*\otimes(\vac_\cW\otimes\vac_\cW))\otimes\mW_j \xrightarrow{\Id_{\mW_j^*}\otimes(l_{\vac_\cW}=r_{\vac_\cW}))\otimes\Id_{\mW_j}} (\mW_j^*\otimes\vac_\cW)\otimes\mW_j\nonumber\\
 & \xrightarrow{r_{\mW_j^*}\otimes\Id_{\mW_j}} \mW_j^*\otimes\mW_j\xrightarrow{e_{\mW_j}} \vac_\cW.
\end{align}
We can use the triangle axiom to rewrite the first two arrows as $(r_{\mW_j^*\otimes\vac_\cW}^{-1}\otimes\Id_{\mW_j})\circ(r_{\mW_j^*}^{-1}\otimes\Id_{\mW_j})$, and then the automorphism of $\mW_j^*$ resulting from the first five arrows is
\begin{align*}
 \mW_j^*\xrightarrow{r_{\mW_j^*}^{-1}} \mW_j^* & \otimes\vac_\cW\xrightarrow{r_{\mW_j^*\otimes\vac_\cW}^{-1}} (\mW_j^*\otimes\vac_\cW)\otimes\vac_\cW\nonumber\\
 &\xrightarrow{\cA^{-1}_{\mW_j^*,\vac_\cW,\vac_\cW}} \mW_j^*\otimes(\vac_\cW\otimes\vac_\cW)\xrightarrow{\Id_{\mW_j^*}\otimes r_{\vac_\cW}} \mW_j^*\otimes\vac_\cW\xrightarrow{r_{\mW_j^*}} \mW_j^*.
\end{align*}
But the middle three arrows are $\Id_{\mW_j^*\otimes\vac_\cW}$ by properties of the right unit isomorphisms, so the whole composition is $\Id_{\mW_j^*}$. Thus \eqref{Wsidecomp} is simply $e_{\mW_j}$.
\end{proof}

Now we can begin the proof of Key Lemma \ref{lemma:keyiso}:
\begin{proof}
 Frobenius reciprocity and properties of duals show that we have natural isomorphisms
 \begin{align*}
  \Hom{\repA}{\cF(\mM_{i0}),\cF(\mM_{0i}^\ast)}
  &\cong \Hom{\cC}{\mM_{i0}, \aA\otimes \mM_{0i}^\ast}\nonumber\\
  &\cong \Hom{\cC}{\mM_{i0}\otimes \mM_{0i}, \aA}\nonumber\\
  &\cong \Hom{\cC}{\mU_i\fus \mW_i, \aA},
 \end{align*}
and similarly
\begin{equation*}
 \Hom{\repA}{\cF(\mM_{0i}^\ast),\cF(\mM_{i0})} \cong \Hom{\cC}{\mU_{i'}\fus \mW_{i'}, \aA}.
\end{equation*}
Specifically, the inclusion $\mU_i\fus \mW_i\hookrightarrow \aA$ induces a $\repA$-homomorphism 
\begin{align*}\Phi: \cF(\mM_{i0})\rightarrow\cF(\mM_{0i}^\ast)\end{align*}
given by the composition
\begin{align*}
 \aA \otimes \mM_{i0} & \xrightarrow{r_{\aA\otimes\mM_{i0}}^{-1}} (\aA\otimes\mM_{i0})\otimes\vac\xrightarrow{\Id_\aA \otimes i_{\mM_{0i}}} (\aA\otimes \mM_{i0})\otimes (\mM_{0i}\otimes \mM_{0i}^\ast)\nonumber\\
 &\xrightarrow{assoc.} (\aA\otimes (\mM_{i0}\otimes \mM_{0i}))\otimes \mM_{0i}^\ast\xrightarrow{(\Id_\aA\otimes(r_{\mU_i}\fus l_{\mW_i}))\otimes\Id_{\mM_{0i}^*}} (\aA\otimes \mM_{ii})\otimes \mM_{0i}^\ast\nonumber\\
 &\xrightarrow{(\mu_\aA\vert_{\aA\otimes \mM_{ii}}) \otimes \Id_{\mM_{0i}^*}} \aA\otimes \mM_{0i}^\ast,
 \end{align*}
where \textit{assoc.} indicates the obvious composition of associativity isomorphisms in $\cC$.
Similarly, the inclusion $\mU_{i'}\otimes \mW_{i'}\hookrightarrow \aA$ induces a $\repA$-morphism 
\begin{align*}
\Psi: \cF(\mM_{0i}^\ast)\rightarrow\cF(\mM_{i0})
\end{align*}
given by the composition
\begin{align*}
 \aA\otimes  \mM_{0i}^\ast& \xrightarrow{r_{\aA\otimes\mM_{0i}^*}^{-1}} (\aA\otimes\mM_{0i}^*)\otimes\vac\xrightarrow{\Id_\aA\otimes\widetilde{i}_{\mM_{i0}^\ast}}  (\aA\otimes \mM_{0i}^\ast)\otimes (\mM_{i0}^\ast\otimes \mM_{i0})\nonumber\\
 &\xrightarrow{assoc.} (\aA\otimes (\mM_{0i}^\ast\otimes \mM_{i0}^\ast))\otimes \mM_{i0}\xrightarrow{(\Id_\aA\otimes\varphi_i^{-1}\circ(l_{\mU_i^*}\fus r_{\mW_i^*}))\otimes\Id_{\mM_{i0}}}  (\aA\otimes \mM_{i'i'})\otimes \mM_{i0}\nonumber\\
 &\xrightarrow{\mu_\aA\vert_{\aA\otimes\mM_{i' i'}}\otimes \Id_{\mM_{i0}}} \aA\otimes \mM_{i0},
\end{align*}
where we have identified $\mM_{i0}$ as the dual of $\mM_{i0}^\ast$ using the coevaluation
\begin{equation*}
 \widetilde{i}_{\mM_{i0}^\ast} =\cR_{\mM_{i0},\mM_{i0}^\ast}\circ(\theta_{\mM_{i0}}\otimes \Id_{\mM_{i0}^\ast})\circ i_{\mM_{i0}};
\end{equation*}
the corresponding evaluation is
\begin{equation*}
 \widetilde{e}_{\mM_{i0}^\ast} = e_{\mM_{i0}}\circ\cR_{\mM_{i0}^\ast,\mM_{i0}}^{-1}\circ(\theta_{\mM_{i0}}^{-1}\otimes \Id_{\mM_{i0}^\ast}).
\end{equation*}
We represent $\Phi$ and $\Psi$ diagrammatically as follows:
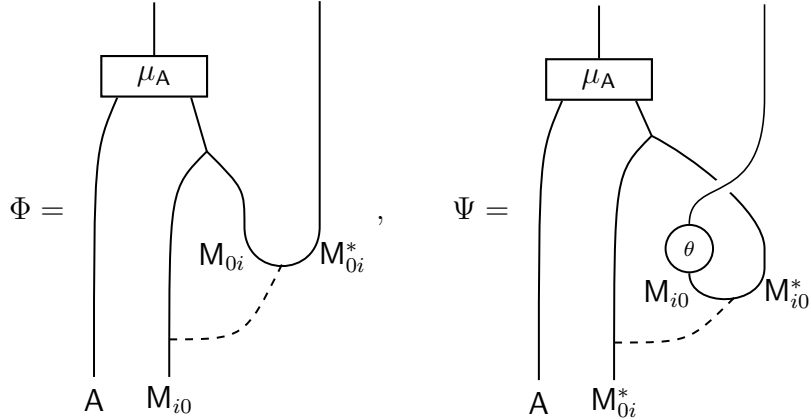
\begin{align*}
\Phi = 
\begin{matrix}
\begin{tikzpicture}[scale = 1, line width = 0.75pt, baseline = {(current bounding box.center)}]
\node (mu) at (0.8,4) [draw,minimum width=40pt,minimum height=10pt,thick, fill=white] {$\mu_\aA$};
\draw (0,0) .. controls (0,3) .. (mu.210);
\draw[dashed] (1,0.5) .. controls (2,0.5) .. (2.5,1.5);
\draw (1,0) .. controls (1,2.5) .. (1.5,3) .. controls (2,2.5) .. (2,2) .. controls (2,1.3) and (3,1.3) .. (3,2) -- (3,5);
\draw (1.5,3) -- (mu.330);
\draw (mu.90) -- ($(mu.90)+(0,0.7)$);
\node at (0,-0.3) {$\aA$};
\node at (1,-0.3) {$\mM_{i0}$};
\node at (1.7,1.6) {$\mM_{0i}$};
\node at (3.3,1.6) {$\mM_{0i}^\ast$};
\end{tikzpicture}
\end{matrix},
\quad\quad
\Psi = 
\begin{matrix}
\begin{tikzpicture}[scale = 1, baseline = {(current bounding box.center)}, line width=0.75pt]
\node (mu) at (0.8,4) [draw,minimum width=40pt,minimum height=10pt,thick, fill=white] {$\mu_\aA$};
\node (th) at (2,1.75) [circle, draw=black, fill=white] {$\scriptstyle{\theta}$};
\draw (0,0) .. controls (0,3) .. (mu.210);
\draw(1,0) .. controls (1,2.75) .. (1.5,3.25) 
 .. controls (2,3)  and (3,2.25) .. (3,1.75) -- (3,1.5) .. controls (3,0.95) and (2,0.95) ..  (th.270);
\draw[white, double=black, line width = 3pt ]  (th.90) .. controls (2,2.75) and (3,2.25) .. (3,3.75) -- (3,5);
\draw[dashed] (1,0.5) .. controls (2,0.5) .. (2.6,1.1);
\draw (1.5,3.25) -- (mu.330);
\draw (mu.90) -- ($(mu.90)+(0,0.7)$);
\node at (0,-0.3) {$\aA$};
\node at (1,-0.3) {$\mM_{0i}^\ast$};
\node at (1.7,1.15) {$\mM_{i0}$};
\node at (3.3,1.15) {$\mM_{i0}^\ast$};
\end{tikzpicture}
\end{matrix}
\end{align*}

We will show that 
\begin{equation*}
\Phi\circ\Psi =(\dim_{\cU} \mU_i)\cdot \Id_{\cF(\mM_{0i}^\ast)} 
\end{equation*}
and
\begin{equation*}
 \Psi\circ\Phi =\pm(\dim_{\cW} \mW_i) \cdot \Id_{\cF(\mM_{i0})},
\end{equation*}
where the minus sign in the second equation occurs precisely when $\aA$ is a (super)algebra of wrong statistics and $i\in I^1$. This will mean that 
\begin{equation*}
 (\dim_\cU \mU_i)\cdot\Phi=\Phi\circ\Psi\circ\Phi =\pm(\dim_\cW \mW_i)\cdot\Phi,
\end{equation*}
and since $\Phi\neq 0$ (it is induced by the non-zero inclusion $\mU_i\fus\mW_i\hookrightarrow\aA$), we will get
\begin{equation*}
 \dim_\cU \mU_i=\pm(\dim_\cW \mW_i).
\end{equation*}
Since $\dim_\cU \mU_i\neq 0$ by assumption, it will then follow that $\Phi$ is an isomorphism in $\repA$ with inverse $(\dim_\cU \mU_i)^{-1}\cdot\Psi$.

In order to calculate $\Phi\circ\Psi$, we first observe that by Frobenius reciprocity and properties of duals, we have natural isomorphisms
\begin{align*}
 \Hom{\repA}{\cF(\mM_{0i}^\ast),\cF(\mM_{0i}^\ast)} & \cong \Hom{\cC}{\mM_{0i}^\ast, \aA\otimes \mM_{0i}^\ast}\cong \Hom{\cC}{\mM_{0i}^\ast\otimes \mM_{0i}, \aA}.
\end{align*}
Under these identifications, a morphism $F\in\Hom{\repA}{\cF(\mM_{0i}^\ast),\cF(\mM_{0i}^\ast)}$ corresponds to the composition
\begin{align}\label{eqn:PhiPsi}
 \mM_{0i}^\ast\otimes \mM_{0i} & \xrightarrow{l_{\mM_{0i}^*}^{-1}\otimes\Id_{\mM_{0i}}} (\vac\otimes\mM_{0i}^*)\otimes\mM_{0i} \xrightarrow{(\iota_\aA\otimes \Id_{\mM_{0i}^*})\otimes \Id_{\mM_{0i}}} (\aA\otimes \mM_{0i}^\ast)\otimes \mM_{0i}\nonumber\\
 &\xrightarrow{F\otimes \Id_{\mM_{0i}}} (\aA\otimes \mM_{0i}^\ast)\otimes \mM_{0i} \xrightarrow{\cA_{\aA,\mM_{0i}^\ast,\mM_{0i}}^{-1}} \aA\otimes (\mM_{0i}^\ast\otimes_{\cC} \mM_{0i}) \xrightarrow{\Id_\aA\otimes e_{\mM_{0i}}} \aA\otimes\vac\xrightarrow{r_\aA} \aA.
\end{align}
Thus to show $\Phi\circ\Psi=(\dim_{\cU} \mU_i) \cdot \Id_{\cF(\mM_{0i}^\ast)}$, it follows from properties of the unit isomorphisms that it suffices to show \eqref{eqn:PhiPsi} with $F=\Phi\circ\Psi$ reduces to $(\dim_{\cU} \mU_i)\cdot(\iota_\aA\circ e_{\mM_{0i}})$.

Similarly, we have a natural isomorphism
\begin{equation*}
 \Hom{\repA}{\cF(\mM_{i0}),\cF(\mM_{i0})} \cong \Hom{\cC}{\mM_{i0}\otimes \mM_{i0}^\ast, \aA}
\end{equation*}
under which $G\in\Hom{\repA}{\cF(\mM_{i0}),\cF(\mM_{i0})}$ corresponds to the composition
\begin{align}\label{eqn:PsiPhi}
\mM_{i0}\otimes\mM_{i0}^* & \xrightarrow{l^{-1}_{\mM_{i0}}\otimes\Id_{\mM_{i0}^*}} (\vac\otimes\mM_{i0})\otimes \mM_{i0}^\ast  \xrightarrow{(\iota_\aA\otimes \Id_{\mM_{i0}})\otimes \Id_{\mM_{i0}^*}} (\aA\otimes \mM_{i0})\otimes \mM_{i0}^\ast\nonumber\\
&\xrightarrow{G\otimes \Id_{\mM_{i0}^*}} (\aA\otimes \mM_{i0})\otimes \mM_{i0}^\ast\xrightarrow{\cA_{\aA,\mM_{i0},\mM_{i0}^\ast}^{-1}} \aA\otimes (\mM_{i0}\otimes \mM_{i0}^\ast) \xrightarrow{\Id_\aA\otimes \widetilde{e}_{\mM_{i0}^\ast}} \aA\otimes\vac\xrightarrow{r_\aA} \aA.
\end{align}
Again by properties of the unit isomorphisms, we need to show that the above composition for $G=\Psi\circ\Phi$ reduces to $(\pm\dim_{\cW} \mW_i)\cdot(\iota_\aA\circ\widetilde{e}_{\mM_{i0}^\ast})$.

We now calculate \eqref{eqn:PhiPsi} with $F$ replaced by $\Phi\circ\Psi$, manipulating according to the following template:
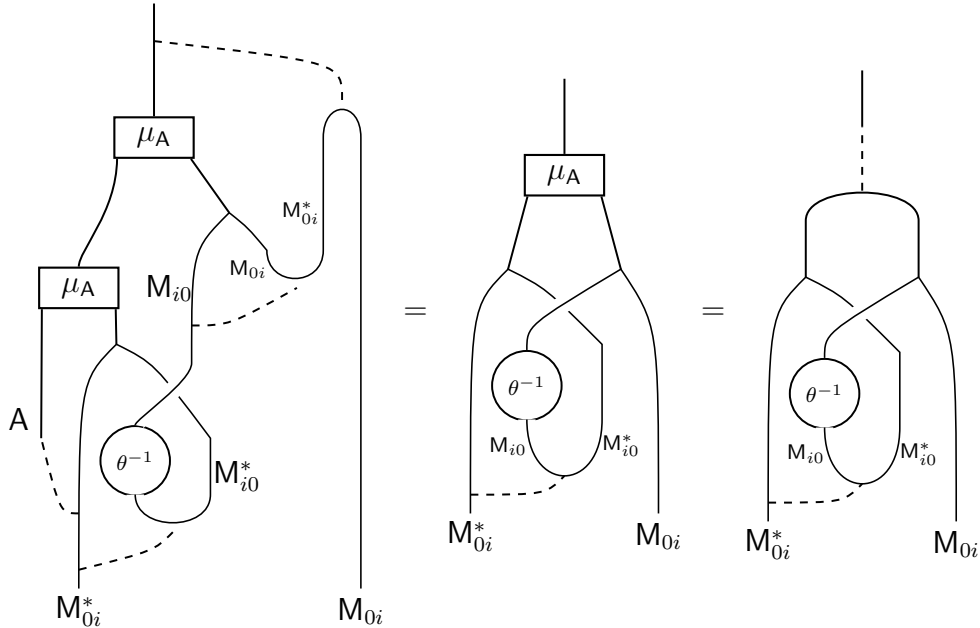
\begin{align*}
&\begin{tikzpicture}
[scale = 1, baseline = {(current bounding box.center)}, line width=0.75pt]
\node (mu) at (0.5,4) [draw,minimum width=30pt,minimum height=10pt,thick, fill=white] {$\mu_\aA$};
\node (mu2) at (1.5,6) [draw,minimum width=30pt,minimum height=10pt,thick, fill=white] {$\mu_\aA$};
\node (th1) at (1.25,1.7) [circle, draw=black, fill=white] {$\scriptstyle{\theta^{-1}}$};
\draw[white, double=black, line width = 3pt] (0.5,0) .. controls (0.5,2.75) .. (1,3.25) .. controls (1.5,3.0)
.. (2.25,2) -- (2.25,1.25) .. controls (2.25,0.75) and (1.25,0.75) .. (1.25,1.25);
\draw[white, double=black, line width = 3pt] (th1.90) .. controls (1.25,2.25) and (2,2.75).. (2,3) .. controls (2,4.5)
.. (2.5,5) .. controls (2.75,4.8) .. (3,4.5) .. controls (3,4) and (3.75,4) .. (3.75,4.5)
-- (3.75,6).. controls (3.75,6.5) and (4.25,6.5) .. (4.25,6)-- (4.25,0);
\draw (mu2.90) -- ($(mu2.90)+(0,1.5)$);
\draw (mu.90) .. controls (0.5,4.5) and (1,5).. (mu2.210);
\draw (2.5,5) --  (mu2.330);
\draw (1,3.25) -- (mu.330);
\draw[dashed] (0.5,1) .. controls (0.25,1) .. (0,2);
\draw (0,2) -- (mu.210);
\draw[dashed] (0.5,0.25) .. controls (1.5,0.5) .. (1.75,0.75);
\draw[dashed] (2,3.5) .. controls (2.5,3.5) .. (3.4,4);
\draw[dashed] (4,6.5) .. controls (4,7) .. ($(mu2.90) + (0,1)$);
\node at (0.5,-0.3) {$\mM_{0i}^\ast$};
\node at (2.6,1.5) {$\mM_{i0}^\ast$};
\node at (1.7,4) {$\mM_{i0}$};
\node at (2.75,4.25) {$\scriptstyle{\mM_{0i}}$};
\node at (3.45,5) {$\scriptstyle{\mM_{0i}^\ast}$};
\node at (4.25,-0.3) {${\mM_{0i}}$};
\node at (-0.3,2.25) {$\aA$};
\end{tikzpicture}
=\begin{tikzpicture}
[scale = 1, baseline = {(current bounding box.center)}, line width=0.75pt]
\node (mu) at (1.25,4.5) [draw,minimum width=30pt,minimum height=10pt,thick, fill=white] {$\mu_\aA$};
\node (th1) at (0.75,1.7) [circle, draw=black, fill=white] {$\scriptstyle{\theta^{-1}}$};
\draw[white, double=black, line width = 3pt] (00,0) .. controls (0,2.75) .. (0.5,3.25) .. controls (1,3.0)
.. (1.75,2.25) -- (1.75,1.25) .. controls (1.75,0.25) and (0.75,0.25) .. (0.75,1.25);
\draw[white, double=black, line width = 3pt] (th1.90) .. controls (0.75,2.5)  .. (2,3.25) .. controls (2.5,2.5) .. (2.5,0);
\draw (0.5,3.25) -- (mu.210);
\draw (2,3.25) -- (mu.330);
\draw (mu.90) -- ($(mu.90)+(0,1)$);
\draw[dashed] (0,0.25) .. controls (1,0.25) .. (1.25,0.5);
\node at (0,-0.3) {$\mM_{0i}^\ast$};
\node at (2.5,-0.3) {$\mM_{0i}$};
\node at (0.5,0.9) {$\scriptstyle{\mM_{i0}}$};
\node at (2,0.9) {$\scriptstyle{\mM_{i0}^\ast}$};
\end{tikzpicture}
=\begin{tikzpicture}
[scale = 1, baseline = {(current bounding box.center)}, line width=0.75pt]
\node (th1) at (0.75,1.7) [circle, draw=black, fill=white] {$\scriptstyle{\theta^{-1}}$};
\draw[white, double=black, line width = 3pt] (00,0) .. controls (0,2.75) .. (0.5,3.25) .. controls (1,3.0)
.. (1.75,2.25) -- (1.75,1.25) .. controls (1.75,0.25) and (0.75,0.25) .. (0.75,1.25);
\draw[white, double=black, line width = 3pt] (th1.90) .. controls (0.75,2.5)  .. (2,3.25) .. controls (2.5,2.5) .. (2.5,0);
\draw (0.5,3.25) -- (0.5,4) .. controls (0.5,4.5) and (2,4.5) .. (2,4) -- (2,3.25);
\draw[dashed] (0,0.25) .. controls (1,0.25) .. (1.25,0.5);
\draw[dashed] (1.25,4.4) -- (1.25,5.25);
\draw (1.25,5.25) -- (1.25,6);
\node at (0,-0.3) {$\mM_{0i}^\ast$};
\node at (2.5,-0.3) {$\mM_{0i}$};
\node at (0.5,0.9) {$\scriptstyle{\mM_{i0}}$};
\node at (2,0.9) {$\scriptstyle{\mM_{i0}^\ast}$};
\end{tikzpicture}
\end{align*}
\begin{align*}
&= 
\begin{tikzpicture}
[scale = 1, baseline = {(current bounding box.center)}, line width=0.75pt]
\node (th1) at (1,1.7) [circle, draw=black, fill=white] {$\scriptstyle{\theta^{-1}}$};
\draw[white, double=black, line width = 3pt] (2.5,0) -- (2.5,3.25) .. controls (2.5,4) and (3.25,4) .. (3.25,3.25) -- (3.25,0);
\draw[white, double=black, line width=3pt] (th1.90) .. controls (1,2.75) and (2,3) .. (2,3.25);
\draw[white, double=black, line width = 3pt] (th1.270) .. controls (1,0.5) and (2,0.5) .. (2,1.25) -- (2,2.25) .. controls (2,2.75) and (1,2.75) .. (1,3.25) .. controls (1,3.75) and (2,3.75).. (2,3.25);
\draw[dashed] (2.5,0.25) .. controls (1.5,0.25) .. (1.5,0.7);
\draw[dashed] (1.5,3.6) .. controls (1.5,4) and (2.25,4) .. (2.25,4.5) .. controls (2.25,4) and
(3,4) .. (2.9,3.8);
\draw[dashed] (2.25,4.6) -- (2.25,5.25);
\draw (2.25,5.25) -- (2.25,6);
\node at (2.5,-0.3) {$\mM_{0i}^\ast$};
\node at (3.25,-0.3) {$\mM_{0i}$};
\node at (0.9,0.6) {$\scriptstyle{\mM_{i0}}$};
\node at (2.1,0.6) {$\scriptstyle{\mM_{i0}^\ast}$};
\end{tikzpicture}
= (\dim_{\cU}\mU_{i})\,
\begin{tikzpicture}
[scale = 1, baseline = {(current bounding box.center)}, line width=0.75pt]
\draw[white, double=black, line width = 3pt] (2.5,0) -- (2.5,3.25) .. controls (2.5,4) and (3.25,4) .. (3.25,3.25) -- (3.25,0);
\draw[dashed] (2.87,3.9)--(2.87,5);
\draw (2.87,5) -- (2.87,6);
\node at (2.5,-0.3) {$\mM_{0i}^\ast$};
\node at (3.25,-0.3) {$\mM_{0i}$};
\end{tikzpicture}.
\end{align*}
We start with the following map in $\Hom{\cC}{\mM_{0i}^*\otimes \mM_{0i}, \aA}$, omitting subscripts from identity morphisms to save space:
\begin{align*}
 & \hspace{-.3em} \mM_{0i}^\ast  \otimes \mM_{0i} 
\xrightarrow{l_{\mM_{0i}^*}^{-1}\otimes\Id} (\vac\otimes\mM_{0i}^*)\otimes\mM_{0i} \xrightarrow{(\iota_\aA\otimes \Id)\otimes \Id} (\aA\otimes \mM_{0i}^\ast)\otimes \mM_{0i}\nonumber\\ &\xrightarrow{r_{\aA\otimes\mM_{0i}^*}^{-1}\otimes\Id} ((\aA\otimes\mM_{0i}^*)\otimes\vac)\otimes\mM_{0i}\xrightarrow{(\Id\otimes\widetilde{i}_{\mM_{i0}^\ast})\otimes \Id} ((\aA\otimes \mM_{0i}^\ast)\otimes(\mM_{i0}^\ast\otimes \mM_{i0}))\otimes \mM_{0i}\nonumber\\
&\xrightarrow{assoc.} ((\aA\otimes(\mM_{0i}^\ast\otimes \mM_{i0}^\ast))\otimes \mM_{i0})\otimes \mM_{0i}\xrightarrow{((\Id\otimes\varphi_i^{-1}\circ(l_{\mU_i^*}\fus r_{\mW_i^*}))\otimes \Id)\otimes \Id} ((\aA\otimes\mM_{i'i'})\otimes \mM_{i0})\otimes \mM_{0i}\nonumber\\
& \xrightarrow{(\mu_\aA\otimes \Id)\otimes \Id} (\aA\otimes \mM_{i0})\otimes \mM_{0i}
\xrightarrow{r_{\aA\otimes\mM_{i0}}^{-1}\otimes\Id} ((\aA\otimes\mM_{i0})\otimes\vac)\otimes\mM_{0i}\nonumber\\ &\xrightarrow{(\Id\otimes i_{\mM_{0i}})\otimes \Id } ((\aA\otimes \mM_{i0})\otimes(\mM_{0i}\otimes \mM_{0i}^\ast))\otimes \mM_{0i}\xrightarrow{assoc.} ((\aA\otimes(\mM_{i0}\otimes \mM_{0i}))\otimes \mM_{0i}^\ast)\otimes \mM_{0i} \nonumber\\
&\xrightarrow{((\Id\otimes(r_{\mU_i}\fus l_{\mW_i}))\otimes\Id)\otimes\Id} ((\aA\otimes\mM_{ii})\otimes\mM_{0i}^*)\otimes\mM_{0i}\xrightarrow{(\mu_\aA\otimes\Id)\otimes\Id} (\aA\otimes\mM_{0i}^*)\otimes\mM_{0i}\nonumber\\
& \xrightarrow{\cA^{-1}_{\aA, \mM_{0i}^\ast, \mM_{0i}}} \aA\otimes(\mM_{0i}^\ast\otimes \mM_{0i})\xrightarrow{\Id\otimes e_{\mM_{0i}}} \aA\otimes\vac\xrightarrow{r_\aA} \aA.
\end{align*}
The first two simplifications to this composition come from the 
unit property of $\mu_\aA$ and the rigidity of $\mM_{0i}$. To achieve these simplifications, we first apply naturalities move $(\iota_\aA\otimes\Id)\circ l_{\mM_{0i}^*}^{-1}$ over several arrows in the composition before applying the unit property. We also apply the triangle axiom to $r_{\aA\otimes\mM_{i0}}^{-1}$ and then naturality of associativity to collect all associativity isomorphisms from the latter half of the composition:
\begin{align*}
 & \mM_{0i}^*  \otimes\mM_{0i} \xrightarrow{r_{\mM_{0i}^*}^{-1}\otimes\Id} (\mM_{0i}^*\otimes\vac)\otimes\mM_{0i} \xrightarrow{(\Id\otimes\widetilde{i}_{\mM_{i0}^*})\otimes\Id} (\mM_{0i}^*\otimes(\mM_{i0}^*\otimes\mM_{i0}))\otimes\mM_{0i}\nonumber\\
 &\xrightarrow{\cA_{\mM_{0i}^*,\mM_{i0}^*,\mM_{i0}}\otimes\Id} ((\mM_{0i}^*\otimes\mM_{i0}^*)\otimes\mM_{i0})\otimes\mM_{0i} \xrightarrow{((l_{\mM_{0i}^*}^{-1}\otimes\Id)\otimes\Id)\otimes\Id} (((\vac\otimes\mM_{0i}^*)\otimes\mM_{i0}^*)\otimes\mM_{i0})\otimes\mM_{0i}\nonumber\\
 &\xrightarrow{(\cA_{\vac,\mM_{0i}^*,\mM_{i0}^*}\otimes\Id)\otimes\Id} ((\vac\otimes(\mM_{0i}^*\otimes\mM_{i0}^*))\otimes\mM_{i0})\otimes\mM_{0i}\nonumber\\ &\xrightarrow{((\Id\otimes\varphi_i^{-1}\circ(l_{\mU_i^*}\fus r_{\mW_i^*}))\otimes \Id)\otimes \Id} ((\vac\otimes\mM_{i'i'})\otimes \mM_{i0})\otimes \mM_{0i}\xrightarrow{(l_{\mM_{i'i'}}\otimes\Id)\otimes\Id} (\mM_{i'i'}\otimes\mM_{i0})\otimes\mM_{0i}\nonumber\\
 &\xrightarrow{\Id\otimes l_{\mM_{0i}}^{-1}} (\mM_{i'i'}\otimes\mM_{i0})\otimes(\vac\otimes\mM_{0i}) \xrightarrow{\Id\otimes(i_{\mM_{0i}}\otimes\Id)} (\mM_{i'i'}\otimes\mM_{i0})\otimes((\mM_{0i}\otimes\mM_{0i}^*)\otimes\mM_{0i})\nonumber\\
& \xrightarrow{assoc.} (\mM_{i'i'}\otimes(\mM_{i0}\otimes\mM_{0i}))\otimes(\mM_{0i}^*\otimes\mM_{0i})\xrightarrow{\mu_\aA\circ(\Id\otimes(r_{\mU_i}\fus l_{\mW_i}))\otimes e_{\mM_{0i}}} \aA\otimes\vac\xrightarrow{r_\aA} \aA.
\end{align*}
Now, the fourth and fifth arrows here are simply $l_{\mM_{0i}^*\otimes\mM_{i0}^*}^{-1}$, and then we can use naturality to cancel $l_{\mM_{i'i'}}$ with its inverse. Meanwhile, the pentagon axiom allows us to choose the first isomorphism in the arrow marked $assoc.$ to be $(\Id\otimes\Id)\otimes\cA^{-1}_{\mM_{0i},\mM_{0i}^*,\mM_{0i}}$, with the remaining associativity isomorphisms respecting the factor of $\mM_{0i}^*\otimes\mM_{0i}$ so that we can applying naturality of associativity to $e_{\mM_{0i}}$:
\begin{align*}
 & \mM_{0i}^*  \otimes\mM_{0i} \xrightarrow{r_{\mM_{0i}^*}^{-1}\otimes\Id} (\mM_{0i}^*\otimes\vac)\otimes\mM_{0i} \xrightarrow{(\Id\otimes\widetilde{i}_{\mM_{i0}^*})\otimes\Id} (\mM_{0i}^*\otimes(\mM_{i0}^*\otimes\mM_{i0}))\otimes\mM_{0i}\nonumber\\
 &\xrightarrow{\cA_{\mM_{0i}^*,\mM_{i0}^*,\mM_{i0}}\otimes\Id} ((\mM_{0i}^*\otimes\mM_{i0}^*)\otimes\mM_{i0})\otimes\mM_{0i} \xrightarrow{(\varphi_i^{-1}\circ(l_{\mU_i^*}\fus r_{\mW_i^*}))\otimes \Id)\otimes \Id} (\mM_{i'i'}\otimes\mM_{i0})\otimes\mM_{0i}\nonumber\\
 &\xrightarrow{\Id\otimes l_{\mM_{0i}}^{-1}} (\mM_{i'i'}\otimes\mM_{i0})\otimes(\vac\otimes\mM_{0i}) \xrightarrow{\Id\otimes(i_{\mM_{0i}}\otimes\Id)} (\mM_{i'i'}\otimes\mM_{i0})\otimes((\mM_{0i}\otimes\mM_{0i}^*)\otimes\mM_{0i})\nonumber\\
 & \xrightarrow{\Id\otimes\cA^{-1}_{\mM_{0i},\mM_{0i}^*,\mM_{0i}}} (\mM_{i'i'}\otimes\mM_{i0})\otimes(\mM_{0i}\otimes(\mM_{0i}^*\otimes\mM_{0i}))\xrightarrow{\Id\otimes(\Id\otimes e_{\mM_{0i}})} (\mM_{i'i'}\otimes\mM_{i0})\otimes(\mM_{0i}\otimes\vac) \nonumber\\
 &\xrightarrow{assoc.} (\mM_{i'i'}\otimes(\mM_{i0}\otimes\mM_{0i}))\otimes\vac \xrightarrow{\mu_\aA\circ(\Id\otimes(r_{\mU_i}\fus l_{\mW_i}))\otimes\Id} \aA\otimes\vac\xrightarrow{r_\aA} \aA.
\end{align*}
Now the rigidity of $\mM_{0i}$ implies that the fifth through eigth arrows above collapse to $\Id\otimes r_{\mM_{0i}}^{-1}$, and further simplifications coming from properties of the right unit isomorphisms give:
\begin{align*}
 & \mM_{0i}^*  \otimes\mM_{0i} \xrightarrow{r_{\mM_{0i}^*}^{-1}\otimes\Id} (\mM_{0i}^*\otimes\vac)\otimes\mM_{0i} \xrightarrow{(\Id\otimes\widetilde{i}_{\mM_{i0}^*})\otimes\Id} (\mM_{0i}^*\otimes(\mM_{i0}^*\otimes\mM_{i0}))\otimes\mM_{0i}\nonumber\\
 &\xrightarrow{\cA_{\mM_{0i}^*,\mM_{i0}^*,\mM_{i0}}\otimes\Id} ((\mM_{0i}^*\otimes\mM_{i0}^*)\otimes\mM_{i0})\otimes\mM_{0i} \xrightarrow{(\varphi_i^{-1}\circ(l_{\mU_i^*}\fus r_{\mW_i^*}))\otimes \Id)\otimes \Id} (\mM_{i'i'}\otimes\mM_{i0})\otimes\mM_{0i}\nonumber\\
 & \xrightarrow{\cA_{\mM_{i'i'},\mM_{i0},\mM_{0i}}^{-1}} \mM_{i'i'}\otimes(\mM_{i0}\otimes\mM_{0i})\xrightarrow{\mu_\aA\circ(\Id\otimes(r_{\mU_i}\fus l_{\mW_i}))} \aA.
\end{align*}

Now observe that the entire composition is a morphism in
\begin{align*}
 \HHom_\cC(\mM_{0i}^*\otimes\mM_{0i},\aA) & \cong\HHom_{\cU\fus\cW} (\vac_\cU\fus(\mW_i^*\otimes\mW_i),\aA)\nonumber\\
 &\cong\bigoplus_{j\in I} \HHom_\cU(\vac_\cU,\mU_j)\otimes\HHom_\cW(\mW_i^*\otimes\mW_i,\mW_j).
\end{align*}
Since $\dim\HHom_\cU(\vac_\cU,\mU_j)=\delta_{0,j}$, it follows that the image must be contained in $\mU_0\fus\mW_0=\vac$.
Consequently, post-composing with $\iota_\aA\circ\varepsilon_\aA$ has no effect on the composition, and we may use Lemma \ref{lem:AiBiduals} and naturality of the associativity isomorphisms to reduce to
\begin{align*}
 \mM_{0i}^* & \otimes\mM_{0i} \xrightarrow{r_{\mM_{0i}^*}^{-1}\otimes\Id} (\mM_{0i}^*\otimes\vac)\otimes\mM_{0i} \xrightarrow{(\Id\otimes\widetilde{i}_{\mM_{i0}^*})\otimes\Id} (\mM_{0i}^*\otimes(\mM_{i0}^*\otimes\mM_{i0}))\otimes\mM_{0i}\nonumber\\
 & \xrightarrow{assoc.} (\mM_{i0}^*\otimes\mM_{0i}^*)\otimes(\mM_{i0}\otimes\mM_{0i}) \xrightarrow{(l_{\mU_i^*}\fus r_{\mW_i^*})\otimes(r_{\mU_i}\otimes l_{\mW_i})} \mM_{ii}^*\otimes\mM_{ii}\xrightarrow{e_{\mM_{ii}}} \vac\xrightarrow{\iota_\aA} \aA.
\end{align*}
Now using Lemma \ref{lem:tensprodevals}, we get
\begin{align*}
  \mM_{0i}^*  \otimes\mM_{0i} \xrightarrow{r_{\mM_{0i}^*}^{-1}\otimes\Id} &(\mM_{0i}^*\otimes\vac)\otimes\mM_{0i} \xrightarrow{(\Id\otimes\widetilde{i}_{\mM_{i0}^*})\otimes\Id} (\mM_{0i}^*\otimes(\mM_{i0}^*\otimes\mM_{i0}))\otimes\mM_{0i}\nonumber\\
  & \xrightarrow{(\Id\otimes e_{\mM_{i0}})\otimes\Id} (\mM_{0i}^*\otimes\vac)\otimes\mM_{0i}\xrightarrow{r_{\mM_{0i}^*}\otimes\Id} \mM_{0i}^*\otimes\mM_{0i}\xrightarrow{e_{\mM_{0i}}} \vac\xrightarrow{\iota_\aA} \aA.
\end{align*}
Since by definition 
\begin{align}
   e_{\mM_{i0}}\circ\widetilde{i}_{\mM_{i0}^\ast} = \dim_{\cC} \mM_{i0},
\end{align}
and since $\dim_\cU \mU_i=\dim_\cC \mM_{0i}$, it follows that we get $(\dim_{\cU} \mU_i)(\iota_\aA\circ e_{\mM_{0i}})$, completing the proof that $\Phi\circ\Psi =(\dim_{\cU} \mU_{i})\cdot \Id_{\cF(\mM_{0i}^{\ast})}$.

Now we calculate $\Psi\circ\Phi$ by considering equation \eqref{eqn:PsiPhi} with $G=\Psi\circ\Phi$. 
We use the following diagrams as a guide:
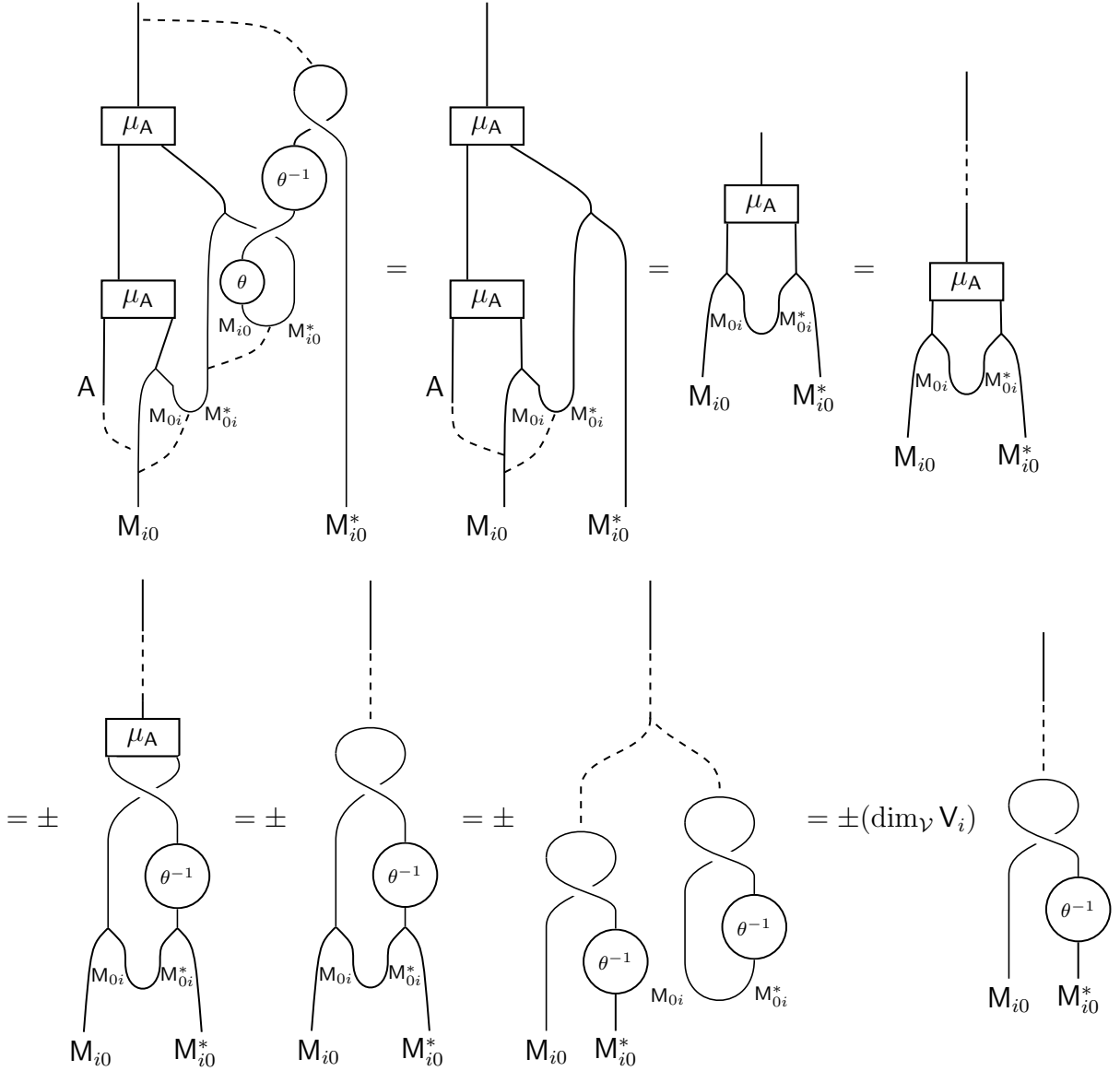
\begin{align*}
&\begin{tikzpicture}
[scale = 1, baseline = {(current bounding box.center)}, line width=0.75pt]
\node (mu) at (0.5,3) [draw,minimum width=30pt,minimum height=10pt,thick, fill=white] {$\mu_\aA$};
\node (mu2) at (0.5,5.5) [draw,minimum width=30pt,minimum height=10pt,thick, fill=white] {$\mu_\aA$};
\node (th1) at (2,3.25) [circle, draw=black, fill=white] {$\scriptstyle{\theta}$};
\node (th2) at (2.75,4.75) [circle, draw=black, fill=white] {$\scriptstyle{\theta^{-1}}$};
\draw (0,1.5) -- (mu.210);
\draw[dashed] (0.6,0.75) .. controls (0,1) ..  (0,1.5);
\draw[white, double=black, line width = 3pt] (0.5,0) .. controls (0.5,1.75) .. (0.75,2)   
-- (1,1.75) .. controls (1, 1.25) and (1.5,1.25) .. (1.5,1.75) .. controls (1.5,4) ..
 (1.75, 4.25) .. controls (2,4) and (2.75,4).. (2.75,3.5) -- (2.75,3) .. controls (2.75,2.5) and (2,2.5) .. (th1.270);
\draw[white, double=black, line width = 3pt] (th1.90) .. controls (2,4) and (2.75,4) .. (th2.270);
\draw (2.75, 6) .. controls (2.75,6.5) and (3.5,6.5) .. (3.5,6) .. controls  (3.5,5.5) and (2.75, 5.5) ..  (th2.90);
\draw[white, double=black, line width=3pt] (3.5,0) -- (3.5,5) .. controls (3.5,5.5) and (2.75,5.5) .. (2.75, 6);
\draw[dashed] (0.5,0.5) .. controls (1,0.75) .. (1.25,1.4);
\draw[dashed] (1.5,2) .. controls (2.25,2.25) .. (2.4,2.6);
\draw (0.75,2) -- (mu.330);
\draw (mu.135) -- (mu2.225);
\draw (1.75, 4.25) .. controls (1.75,4.5)  .. (mu2.320);
\draw (mu2.90)  -- ($(mu2.90) + (0,1.5)$);
\draw[dashed] (3,6.4) .. controls (2.5,7) .. ($(mu2.90) + (0,1.25)$);
\node at (-0.25,1.75) {$\aA$};
\node at (0.5,-0.3) {$\mM_{i0}$};
\node at (0.9,1.3) {$\scriptstyle{\mM_{0i}}$};
\node at (1.7,1.3) {$\scriptstyle{\mM_{0i}^\ast}$};
\node at (1.9,2.6) {$\scriptstyle{\mM_{i0}}$};
\node at (2.9,2.5) {$\scriptstyle{\mM_{i0}^\ast}$};
\node at (3.5,-0.3) {${\mM_{i0}^\ast}$};
\end{tikzpicture}
=
\begin{tikzpicture}
[scale = 1, baseline = {(current bounding box.center)}, line width=0.75pt]
\node (mu) at (0.5,3) [draw,minimum width=30pt,minimum height=10pt,thick, fill=white] {$\mu_\aA$};
\node (mu2) at (0.5,5.5) [draw,minimum width=30pt,minimum height=10pt,thick, fill=white] {$\mu_\aA$};
\draw (0,1.5) -- (mu.210);
\draw[dashed] (0.75,0.75) .. controls (0,1) ..  (0,1.5);
\draw (0.75,0) .. controls (0.75,1.75) .. (1,2)   
-- (1.25,1.75) .. controls (1.25, 1.25) and (1.75,1.25) .. (1.75,1.75) .. controls (1.75,4) ..
 (2, 4.25) .. controls (2.25,4) and (2.5,4) .. (2.5,3.5) -- (2.5,3) .. controls (2.5,2.5) and (2.5,2.5) .. (2.5,0);
\draw[dashed] (0.75,0.5) .. controls (1.25,0.75) .. (1.5,1.4);
\draw (1,2) -- (mu.330);
\draw (mu.135) -- (mu2.225);
\draw (2, 4.25) .. controls (2,4.5)  .. (mu2.320);
\draw (mu2.90)  -- ($(mu2.90) + (0,1.5)$);
\node at (-0.25,1.75) {$\aA$};
\node at (0.5,-0.3) {$\mM_{i0}$};
\node at (1.1,1.3) {$\scriptstyle{\mM_{0i}}$};
\node at (1.95,1.3) {$\scriptstyle{\mM_{0i}^\ast}$};
\node at (2.25,-0.3) {${\mM_{i0}^\ast}$};
\end{tikzpicture}
=
\begin{tikzpicture}
[scale = 1, baseline = {(current bounding box.center)}, line width=0.75pt]
\node (mu) at (0.75,2.5) [draw,minimum width=30pt,minimum height=10pt,thick, fill=white] {$\mu_\aA$};
\draw (-0.1,0) .. controls (0,1.25) .. (0.25,1.5) ..controls (0.5,1.25) .. (0.5,1) .. controls (0.5,0.5) and (1,0.5) .. (1,1) .. controls (1,1.25) .. (1.25,1.5) .. controls (1.5,1.25) .. (1.6,0);
\draw (0.25,1.5) -- (mu.210);
\draw (1.25,1.5) -- (mu.330);
\draw (mu.90) -- ($(mu.90)+(0,0.75)$);
\node at (0,-0.3) {$\mM_{i0}$};
\node at (0.25,0.8) {$\scriptstyle{\mM_{0i}}$};
\node at (1.25,0.8) {$\scriptstyle{\mM_{0i}^\ast}$};
\node at (1.5,-0.3) {${\mM_{i0}^\ast}$};
\end{tikzpicture}
= 
\begin{tikzpicture}
[scale = 1, baseline = {(current bounding box.center)}, line width=0.75pt]
\node (mu) at (0.75,2.25) [draw,minimum width=30pt,minimum height=10pt,thick, fill=white] {$\mu_\aA$};
\draw (-0.1,0) .. controls (0,1.25) .. (0.25,1.5) ..controls (0.5,1.25) .. (0.5,1) .. controls (0.5,0.5) and (1,0.5) .. (1,1) .. controls (1,1.25) .. (1.25,1.5) .. controls (1.5,1.25) .. (1.6,0);
\draw (0.25,1.5) -- (mu.210);
\draw (1.25,1.5) -- (mu.330);
\draw (mu.90) -- ($(mu.90)+(0,0.75)$);
\draw[dashed]  ($(mu.90)+(0,0.75)$) -- ($(mu.90)+(0,1.75)$);
\draw  ($(mu.90)+(0,1.75)$) -- ($(mu.90)+(0,2.75)$);
\node at (0,-0.3) {$\mM_{i0}$};
\node at (0.25,0.8) {$\scriptstyle{\mM_{0i}}$};
\node at (1.25,0.8) {$\scriptstyle{\mM_{0i}^\ast}$};
\node at (1.5,-0.3) {${\mM_{i0}^\ast}$};
\end{tikzpicture}
\end{align*}
\begin{align*}
=
\pm\begin{tikzpicture}
[scale = 1, baseline = {(current bounding box.center)}, line width=0.75pt]
\node (mu) at (0.75,4.25) [draw,minimum width=30pt,minimum height=10pt,thick, fill=white] {$\mu_\aA$};
\node (th1) at (1.25,2.25) [circle,draw=black, fill=white] {$\scriptstyle{\theta^{-1}}$};
\draw (-0.1,0) .. controls (0,1.25) .. (0.25,1.5) ..controls (0.5,1.25) .. (0.5,1) .. controls (0.5,0.5) and (1,0.5) .. (1,1) .. controls (1,1.25) .. (1.25,1.5) .. controls (1.5,1.25) .. (1.6,0);
\draw[white, double=black, line width = 3pt] (0.25,1.5) -- (0.25, 2.75).. controls (0.25,3.5) and (1.5,3.5) ..  (mu.330);
\draw[white, double=black, line width = 3pt] (1.25,1.5) 
-- (th1.270);
\draw (mu.90) -- ($(mu.90)+(0,0.25)$);
\draw[dashed]  ($(mu.90)+(0,0.25)$) -- ($(mu.90)+(0,1.25)$);
\draw  ($(mu.90)+(0,1.25)$) -- ($(mu.90)+(0,2)$);
\draw[white, double=black, line width = 3pt] (th1.90) -- ($(th1.90)+(0,0.25)$) .. controls (1.25,3.5) and (0.25,3.5) .. (mu.210);
\node at (0,-0.3) {$\mM_{i0}$};
\node at (0.25,0.8) {$\scriptstyle{\mM_{0i}}$};
\node at (1.25,0.8) {$\scriptstyle{\mM_{0i}^\ast}$};
\node at (1.5,-0.3) {${\mM_{i0}^\ast}$};
\end{tikzpicture}
=
\pm\begin{tikzpicture}
[scale = 1, baseline = {(current bounding box.center)}, line width=0.75pt]
\node (th1) at (1.25,2.25) [circle,draw=black, fill=white] {$\scriptstyle{\theta^{-1}}$};
\draw (-0.1,0) .. controls (0,1.25) .. (0.25,1.5) ..controls (0.5,1.25) .. (0.5,1) .. controls (0.5,0.5) and (1,0.5) .. (1,1) .. controls (1,1.25) .. (1.25,1.5) .. controls (1.5,1.25) .. (1.6,0);
\draw[white, double=black, line width = 3pt] (0.25,1.5) -- (0.25, 2.75).. controls (0.25,3.5) and (1.25,3.5) ..  (1.25,4) .. controls (1.25,4.5) and (0.25,4.5) .. (0.25,4);
\draw[white, double=black, line width = 3pt] (1.25,1.5)
-- (th1.270);
\draw[dashed]  (0.75,4.5) -- (0.75,5.5);
\draw  (0.75,5.5) -- (0.75,6.5);
\draw[white, double=black, line width = 3pt] (th1.90) -- ($(th1.90)+(0,0.25)$) .. controls (1.25,3.5) and (0.25,3.5) .. (0.25,4);
\node at (0,-0.3) {$\mM_{i0}$};
\node at (0.25,0.8) {$\scriptstyle{\mM_{0i}}$};
\node at (1.25,0.8) {$\scriptstyle{\mM_{0i}^\ast}$};
\node at (1.5,-0.3) {${\mM_{i0}^\ast}$};
\end{tikzpicture}
=
\pm\begin{tikzpicture}
[scale = 1, baseline = {(current bounding box.center)}, line width=0.75pt]
\node (th1) at (1,1) [circle,draw=black, fill=white] {$\scriptstyle{\theta^{-1}}$};
\node (th2) at (3,1.5) [circle,draw=black, fill=white] {$\scriptstyle{\theta^{-1}}$};
\draw[white, double=black, line width = 3pt] (0,2.5) .. controls (0,3) and (1,3) .. (1,2.5) .. controls (1,2) and (0,2) .. (0,1.5) -- (0,0);
\draw[white, double=black, line width = 3pt] (th1.90) -- ($(th1.90)+(0,0.25)$) .. controls (1,2) and (0,2) .. (0,2.5);
\draw (1,0) -- (th1.270);
\draw[white, double=black, line width = 3pt] (2,3) .. controls (2,3.5) and (3,3.5) .. (3,3) .. controls (3,2.5) and (2,2.5) .. (2,2) -- (2,1) .. controls (2,0.25) and (3,0.25) .. (th2.270);
\draw[white, double=black, line width = 3pt] (th2.90) -- ($(th2.90)+(0,0.25)$) .. controls (3,2.5) and (2,2.5) .. (2,3);
\draw[dashed] (0.5,3) -- (0.5,3.5) .. controls (0.5,4) and (1.5,4) .. (1.5,4.5) .. 
controls (2,4) and (2.5,4) .. (2.5,3.5);
\draw[dashed] (1.5,4.5) -- (1.5,5.5);
\draw (1.5,5.5) -- (1.5,6.5);
\node at (0,-0.3) {$\mM_{i0}$};
\node at (1,-0.3) {$\mM_{i0}^\ast$};
\node at (1.75,0.5) {$\scriptstyle{\mM_{0i}}$};
\node at (3.25,0.5) {$\scriptstyle{\mM_{0i}^\ast}$};
\end{tikzpicture}
=\pm(\dim_{\cW} \mW_{i})\begin{tikzpicture}
[scale = 1, baseline = {(current bounding box.center)}, line width=0.75pt]
\node (th1) at (1,1) [circle,draw=black, fill=white] {$\scriptstyle{\theta^{-1}}$};
\draw[white, double=black, line width = 3pt] (0,2.5) .. controls (0,3) and (1,3) .. (1,2.5) .. controls (1,2) and (0,2) .. (0,1.5) -- (0,0);
\draw[white, double=black, line width = 3pt] (th1.90) -- ($(th1.90)+(0,0.25)$) .. controls (1,2) and (0,2) .. (0,2.5);
\draw (1,0) -- (th1.270);
\draw[dashed] (0.5,3) -- (0.5,4);
\draw (0.5,4) -- (0.5,5);
\node at (0,-0.3) {$\mM_{i0}$};
\node at (1,-0.3) {$\mM_{i0}^\ast$};
\end{tikzpicture} 
\end{align*}
In this case, the relevant composition is:
\begin{align*}
 & \hspace{-.3em}\mM_{i0}\otimes \mM_{i0}^\ast \xrightarrow{l_{\mM_{i0}}^{-1}\otimes\Id} (\vac\otimes\mM_{i0})\otimes\mM_{i0}^*  \xrightarrow{(\iota_\aA\otimes \Id)\otimes\Id} (\aA\otimes \mM_{i0})\otimes \mM_{i0}^\ast\nonumber\\ &\xrightarrow{r_{\aA\otimes\mM_{i0}}^{-1}\otimes\Id} ((\aA\otimes\mM_{i0})\otimes\vac)\otimes\mM_{i0}^*\xrightarrow{(\Id\otimes i_{\mM_{0i}})\otimes\Id} ((\aA\otimes \mM_{i0})\otimes(\mM_{0i}\otimes \mM_{0i}^\ast))\otimes \mM_{i0}^\ast\nonumber\\
 &\xrightarrow{assoc.} ((\aA\otimes(\mM_{i0}\otimes \mM_{0i}))\otimes \mM_{0i}^\ast)\otimes \mM_{i0}^\ast \xrightarrow{((\Id\otimes(r_{\mU_i}\fus l_{\mW_i}))\otimes\Id)\otimes\Id} ((\aA\otimes\mM_{ii})\otimes\mM_{0i}^*)\otimes\mM_{i0}^*\nonumber\\ &\xrightarrow{(\mu_\aA\otimes \Id)\otimes \Id} (\aA\otimes \mM_{0i}^\ast)\otimes \mM_{i0}^\ast \xrightarrow{r_{\aA\otimes\mM_{0i}^*}^{-1}\otimes\Id} ((\aA\otimes\mM_{0i}^*)\otimes\vac)\otimes\mM_{i0}^*\nonumber\\
 & \xrightarrow{(\Id\otimes\widetilde{i}_{\mM_{i0}^\ast})\otimes \Id} ((\aA\otimes \mM_{0i}^\ast)\otimes(\mM_{i0}^\ast\otimes \mM_{i0}))\otimes \mM_{i0}^\ast\xrightarrow{assoc.} ((\aA\otimes(\mM_{0i}^\ast\otimes \mM_{i0}^\ast))\otimes \mM_{i0})\otimes \mM_{i0}^\ast\nonumber\\
 & \xrightarrow{((\Id\otimes\varphi_i^{-1}\circ(r_{\mU_i^*}\fus l_{\mW_i^*}))\otimes \Id)\otimes \Id} ((\aA\otimes\mM_{i'i'})\otimes \mM_{i0})\otimes \mM_{i0}^\ast \xrightarrow{(\mu_\aA\otimes \Id)\otimes \Id} (\aA\otimes \mM_{i0})\otimes \mM_{i0}^\ast\nonumber\\
 &\xrightarrow{\cA_{\aA,\mM_{i0},\mM_{i0}^\ast}^{-1}} \aA\otimes(\mM_{i0}\otimes \mM_{i0}^\ast)\xrightarrow{1_\aA\otimes\widetilde{e}_{\mM_{i0}^\ast}} \aA\otimes\vac\xrightarrow{r_\aA} \aA.
\end{align*}
As before, we can simplify using the unit property of $\aA$ and the rigidity of $\mM_{i0}^\ast$ (with evaluation $\widetilde{e}_{\mM_{i0}^\ast}$ and coevaluation $\widetilde{i}_{\mM_{i0}^\ast}$) to get:
\begin{align*}
 \mM_{i0} & \otimes\mM_{i0}^*\xrightarrow{r_{\mM_{i0}}^{-1}\otimes\Id} (\mM_{i0}\otimes\vac)\otimes\mM_{i0}^*\xrightarrow{(\Id\otimes i_{\mM_{0i}})\otimes \Id} (\mM_{i0}\otimes(\mM_{0i}\otimes \mM_{0i}^\ast))\otimes \mM_{i0}^\ast\nonumber\\
& \xrightarrow{\cA_{\mM_{i0},\mM_{0i},\mM_{0i}^\ast}\otimes \Id} ((\mM_{i0}\otimes \mM_{0i})\otimes \mM_{0i}^\ast)\otimes \mM_{i0}^\ast\xrightarrow{((r_{\mU_i}\fus l_{\mW_i})\otimes\Id)\otimes\Id} (\mM_{ii}\otimes\mM_{0i}^*)\otimes\mM_{i0}^*\nonumber\\
& \xrightarrow{\cA^{-1}_{\mM_{ii}, \mM_{0i}^\ast,\mM_{i0}^\ast}} \mM_{ii}\otimes(\mM_{i0}^\ast\otimes \mM_{0i}^\ast) \xrightarrow{\Id\otimes\varphi_i^{-1}\circ(r_{\mU_i^*}\fus l_{\mW_i^*})} \mM_{ii}\otimes\mM_{i'i'}\xrightarrow{\mu_\aA} \aA.
\end{align*}
Also as before, we post-compose with $\iota_\aA\circ\varepsilon_\aA$, which does not change the above composition because $\dim\HHom_{\cW}(\vac_\cW,\mW_j)=\delta_{0,j}$.

Now we use the (super)commutativity of $\mu_\aA$ and properties of the twist $\theta_\aA$ to replace
\begin{equation*}
 \mu_\aA\vert_{\mM_{ii}\otimes\mM_{i'i'}} =\pm\mu_\aA\circ\cR^{-1}_{\mM_{i'i'}, \mM_{ii}}\circ(\Id\otimes\theta^{-1}_{\mM_{i'i'}}),
\end{equation*}
where the minus sign occurs precisely when $\aA$ is a (super)algebra of wrong statistics and $i\in I^1$ (recall Remark \ref{i_i'_same_parity}). Next applying naturality of $\cR$ and $\theta$ to $\varphi_i^{-1}$, and then using Lemma \ref{lem:AiBiduals}, our composition becomes, up to sign,
\begin{align*}
 &\mM_{i0}  \otimes\mM_{i0}^*\xrightarrow{r_{\mM_{i0}}^{-1}\otimes\Id} (\mM_{i0}\otimes\vac)\otimes\mM_{i0}^*\xrightarrow{(\Id\otimes i_{\mM_{0i}})\otimes \Id} (\mM_{i0}\otimes(\mM_{0i}\otimes \mM_{0i}^\ast))\otimes \mM_{i0}^\ast\nonumber\\
& \xrightarrow{\cA_{\mM_{i0},\mM_{0i},\mM_{0i}^\ast}\otimes \Id} ((\mM_{i0}\otimes \mM_{0i})\otimes \mM_{0i}^\ast)\otimes \mM_{i0}^\ast\xrightarrow{((r_{\mU_i}\fus l_{\mW_i})\otimes\Id)\otimes\Id} (\mM_{ii}\otimes\mM_{0i}^*)\otimes\mM_{i0}^*\nonumber\\
& \xrightarrow{\cA^{-1}_{\mM_{ii}, \mM_{0i}^\ast,\mM_{i0}^\ast}} \mM_{ii}\otimes(\mM_{i0}^\ast\otimes \mM_{0i}^\ast)\xrightarrow{\Id\otimes(r_{\mu_i^*}\fus l_{\mW_i^*})} \mM_{ii}\otimes\mM_{ii}^*\xrightarrow{\cR_{\mM_{ii}^*,\mM_{ii}}^{-1}\circ(\Id\otimes\theta_{\mM_{ii}^*})} \mM_{ii}^*\otimes\mM_{ii}\xrightarrow{e_{\mM_{ii}}} \vac\xrightarrow{\iota_\aA} \aA.
\end{align*}
We now have the Deligne product of two morphisms in $\cU$ and $\cW$, which we can calculate individually. On the $\cU$ side, we have
\begin{align*}
 \mU_i & \otimes\mU_i^*\xrightarrow{ r_{\mU_i}^{-1}\otimes\Id} (\mU_i\otimes\vac_\cU)\otimes\mU_i^*\xrightarrow{(\Id\otimes(l_{\vac_\cU}^{-1}=r_{\vac_\cU}^{-1}))\otimes\Id} (\mU_i\otimes(\vac_\cU\otimes\vac_\cU))\otimes\mU_i^*\nonumber\\
 &\xrightarrow{\cA_{\mU_i,\vac_\cU,\vac_\cU}\otimes\Id} ((\mU_i\otimes\vac_\cU)\otimes\vac_\cU)\otimes\mU_i^*\xrightarrow{(r_{\mU_i}\otimes\Id)\otimes\Id} (\mU_i\otimes\vac_\cU)\otimes\mU_i^*\xrightarrow{\cA^{-1}_{\mU_i,\vac_\cU,\mU_i^*}} \mU_i\otimes(\vac_\cU\otimes\mU_i^*)\nonumber\\
 & \xrightarrow{\Id\otimes l_{\mU_i^*}} \mU_i\otimes\mU_i^*\xrightarrow{\cR_{\mU_i^*,\mU_i}^{-1}\circ(\Id\otimes\theta_{\mU_i^*})} \mU_i^*\otimes\mU_i \xrightarrow{e_{\mU_i}}\vac_\cU\xrightarrow{\sim}\mU_0.
\end{align*}
Now, the first six arrows collapse to the identity using the triangle axiom, and then we can calculate the rest using properties of twists and duals:
\begin{align*}
 e_{\mU_i}\circ\cR^{-1}_{\mU_i^\ast,\mU_i}\circ(\Id\otimes\theta_{\mU_i^\ast}^{-1}) & = e_{\mU_i}\circ(\theta_{\mU_i^\ast}^{-1}\otimes \Id)\circ\cR_{\mU_i^\ast,\mU_i}^{-1}\nonumber\\
 & = e_{\mU_i}\circ\left( (\theta_{\mU_i}^{-1})^\ast\otimes \Id\right)\circ\cR_{\mU_i^\ast,\mU_i}^{-1}\nonumber\\
 & = e_{\mU_i}\circ(\Id\otimes \theta_{\mU_i}^{-1})\circ\cR_{\mU_i^\ast,\mU_i}^{-1} = \widetilde{e}_{\mU_i^\ast}.
\end{align*}
On the $\cW$ side, we have the composition:
\begin{align*}
 \vac_\cW & \otimes\vac_\cW\xrightarrow{r_{\vac_\cW}^{-1}\otimes\Id} (\vac_\cW\otimes\vac_\cW)\otimes\vac_\cW\xrightarrow{(\Id\otimes i_{\mW_i})\otimes\Id} (\vac_\cW\otimes(\mW_i\otimes\mW_i^*))\otimes\vac_\cW\nonumber\\
 & \xrightarrow{\cA_{\vac_\cW,\mW_i,\mW_i^*}\otimes\Id} ((\vac_\cW\otimes\mW_i)\otimes\mW_i^*)\otimes\vac_\cW\xrightarrow{(l_{\mW_i}\otimes\Id)\otimes\Id} (\mW_i\otimes\mW_i^*)\otimes\vac_\cW\nonumber\\
 &\xrightarrow{\cA^{-1}_{\mW_i,\mW_i^*,\vac_\cW}} \mW_i\otimes(\mW_i^*\otimes\vac_\cW)\xrightarrow{\Id\otimes r_{\mW_i^*}} \mW_i\otimes\mW_i^*\xrightarrow{\cR_{\mW_i^*,\mW_i}^{-1}\circ(\Id\otimes\theta_{\mW_i^*})} \mW_i^*\otimes\mW_i \xrightarrow{e_{\mW_i}}\vac_\cU\xrightarrow{\sim}\mW_0.
\end{align*}
The third and fourth arrows here simplify to $l_{\mW_i\otimes\mW_i^*}\otimes\Id$, and then we can use naturality to cancel this left unit isomorphism with the first arrow of the composition. Moreover, the fifth and sixth arrows simplify to $r_{\mW_i\otimes\mW_i^*}$, and then we can use naturality to move this right unit isomorphism to the beginning of the composition:
\begin{align*}
 \vac_\cW\otimes\vac_\cW\xrightarrow{r_{\vac_\cW}=e_{\vac_\cW}}\vac_\cW\xrightarrow{i_{\mW_i}}\mW_i\otimes\mW_i^*\xrightarrow{\cR_{\mW_i^*,\mW_i}^{-1}\circ(\Id\otimes\theta_{\mW_i^*})} \mW_i^*\otimes\mW_i \xrightarrow{e_{\mW_i}}\vac_\cU\xrightarrow{\sim}\mW_0.
\end{align*}
Now using the balancing equation, $\theta_{\vac_\cW}=\Id$, and $e_{\vac_\cW}=\widetilde{e}_{\vac_\cW}$, we calculate
\begin{align*}
 e_{\mW_i}\circ\cR_{\mW_i^\ast, \mW_i}^{-1}\circ(\Id\otimes\theta^{-1}_{\mW_i^\ast})\circ i_{\mW_i}\circ e_{\vac_\cW} & = e_{\mW_i}\circ\theta_{\mW_i^\ast\otimes \mW_i}^{-1}\circ\cR_{\mW_i,\mW_i^\ast}\circ(\theta_{\mW_i}\otimes \Id)\circ i_{\mW_i}\circ\widetilde{e}_{\vac_\cW}\nonumber\\
 & = \theta_{\one_{\cW}}^{-1}\circ e_{\mW_i}\circ\widetilde{i}_{\mW_i^\ast}\circ\widetilde{e}_{\vac_\cW} =(\dim_{\cW} \mW_i)\widetilde{e}_{\vac_\cW}.
\end{align*}
In conclusion, we have shown that the composition in \eqref{eqn:PsiPhi} with $G=\Psi\circ\Phi$ equals $$\pm(\dim_{\cW} \mW_i)(\iota_\aA\circ\widetilde{e}_{\mM_{i0}}).$$ This completes the proof that $\Phi$ and $\Psi$ are both isomorphisms in $\repA$.
\end{proof}

Finally we prove Proposition \ref{prop:UA_WA_ribbon}:
\begin{proof}
 Using Key Lemma \ref{lemma:keyiso}, we have for $ i,j\in I$ the following isomorphisms in $\repA$ and/or $\cC$:
 \begin{align}
  \cF(\mU_i\boxtimes \mW_j^\ast) & \cong \aA\otimes((\mU_i\boxtimes \one_{\cW})\otimes(\one_{\cU}\boxtimes \mW_j^\ast))\nonumber\\
  & \cong (\aA\otimes (\mU_i\boxtimes \one_{\cW}))\otimes (\one_{\cU}\boxtimes \mW_j^\ast)\nonumber\\
  & \cong (\aA\otimes(\one_{\cU}\boxtimes \mW_i^\ast))\otimes(\one_{\cU}\boxtimes \mW_j^\ast)\nonumber\\
  & \cong \aA\otimes(\one_{\cU}\boxtimes(\mW_i^\ast\otimes \mW_j^\ast))\nonumber\\
  & \cong\bigoplus_{k\in I} \mU_k\boxtimes( \mW_k\otimes(\mW_i^\ast\otimes \mW_j^\ast)),
  \label{eqn:fusion1}
 \end{align}
as well as
\begin{align}
 \cF(\mU_i\boxtimes \mW_j^\ast) & \cong\aA\otimes((\one_{\cU}\boxtimes\mW_j^\ast)\otimes (\mU_i\boxtimes \one_{\cW}))\nonumber\\
 & \cong (\aA\otimes(\one_{\cU}\boxtimes \mW_j^\ast))\otimes (\mU_i\boxtimes \one_{\cW})\nonumber\\
 &\cong (\aA\otimes(\mU_j\boxtimes \one_{\cW}))\otimes(\mU_i\boxtimes \one_{\cW})\nonumber\\
 & \cong \aA\otimes ((\mU_j\otimes \mU_i)\boxtimes \one_{\cW})\nonumber\\
 &\cong\bigoplus_{k\in I} (\mU_k\otimes(\mU_j\otimes \mU_i))\boxtimes \mW_k.
   \label{eqn:fusion2}
\end{align}
The first sequence of isomorphisms shows that $\cF(\mU_i\boxtimes \mW_j^*)$ is an object of $(\cU_\aA\fus\cW)_\oplus$, and then taking $k=0$ in the second sequence of isomorphisms shows that $\cF(\mU_i\fus\mW_j^*)$ contains $(\mU_j\otimes\mU_i)\fus\vac_\cW$ as a subobject. This means $(\mU_j\otimes\mU_i)\fus\vac_\cW$ is an object of $\cU_\aA\fus\cW$, that is, $\mU_j\otimes\mU_i$ is an object of $\cU_\aA$ for all $i,j\in I$. This shows that $\cU_\aA$ is a tensor subcategory of $\cU$. Lemma \ref{lem:AiBiduals} then shows that $\cU_\aA$ is closed under duals, and hence is a ribbon subcategory of $\cU$.

Now we examine the summand of $\cF(\mU_i\boxtimes \mW_j^\ast)$ corresponding to $\one_{\cU}$. On the one hand, \eqref{eqn:fusion1} shows this is $\vac_\cU\fus(\mW_i^\ast\otimes \mW_j^\ast)$. On the other hand, \eqref{eqn:fusion2} combined with semisimplicity of $\cU$ implies it is isomorphic to $\vac_\cU\fus\bigoplus\limits_{k\in I} N^{k^*}_{j, i} \mW_k$, where $N^{k^*}_{j,i}$ is the multiplicity of $\mU_{k}^*$ in $\mU_j\otimes \mU_i$. In other words, we have
\begin{equation*}
 \mW_{i}^\ast\otimes \mW_j^\ast\cong\bigoplus_{k\in I} N^{k^*}_{j,i}\, \mW_k
\end{equation*}
in $\cW$, or equivalently
\begin{equation}
 {\mW_{j}}\otimes \mW_i\cong\bigoplus_{k\in I} N^{k^*}_{j,i}\, \mW_{k'}.
 \label{eqn:Wfus}
\end{equation}
This shows that $\cW_\aA$ is closed under tensor products (and also duals by Lemma \ref{lem:AiBiduals}) and thus is a ribbon subcategory of $\cW$.

Now to show that $\cW_\aA$ is semisimple with the $\mW_i$ as distinct simple objects, we need to show that $\dim\Hom{\cW}{\mW_i,\mW_j} =\delta_{i,j}$ for $i,j\in I$. For this, we calculate
\begin{align*}
 \dim\Hom{\cW}{\mW_{i},\mW_j} & = \dim\Hom{\cW}{\one_{\cW}\otimes \mW_i, \mW_j} =\dim\Hom{\cW}{\one_{\cW}, \mW_j\otimes \mW_{i'}}\nonumber\\
 & =\sum_{k\in I} N^{k^*}_{j,i'}\dim\Hom{\cW}{\one_{\cW}, \mW_{k'}} = N^0_{j,i'} =\delta_{i,j},
\end{align*}
using properties of duals, equation \eqref{eqn:Wfus}, the fact that $\one_{\cU}$ and $\one_{\cW}$ form a dual pair in $\aA$, and the simplicity and mutual inequivalence of the $\mU_i$. 
\end{proof}

\begin{rema}
 The above proof only shows that the $\mW_i$ are simple in $\cW_\aA$, not that they are necessarily simple in the possibly larger category $\cW$. For example, it is conceivable that some $\mW_i$ admits a non-trivial simple quotient $\mW_i/\widetilde{\mW}$, provided that $\mW_i/\widetilde{\mW}$ does not occur as a submodule of any other $\mW_j$.
\end{rema}

\section{From tensor categories to vertex operator algebras}

Here we interpret the categorical theorems of the previous sections as theorems for vertex operator algebras.

\subsection{Vertex tensor categories}

Here to establish notation and terminology, we recall some features and structures in the notion of vertex tensor category as formulated and developed in \cite{HL-VTC}, \cite{HL-tensor1}-\cite{HL-tensor3}, \cite{H-tensor4}, \cite{HLZ1}-\cite{HLZ8}; see also the exposition in \cite[Section 3.1]{CKM}. In contrast to the preceding sections, here we will need to use the symbol $\otimes$ exclusively for vector space tensor products (over $\CC$).

We use the definitions of vertex operator algebra and module for a vertex operator algebra from \cite{LL}, except that we typically allow the Virasoro operator $L(0)$ to act non-semisimply on a module. Such modules are called \textbf{grading-restricted generalized modules} in \cite{HLZ1}. To be more specific, a grading-restricted generalized module $\mX$ has a $\CC$-grading $\mX=\bigoplus_{h\in\CC} \mX_{[h]}$, where $\mX_{[h]}$ is the generalized $L(0)$-eigenspace with generalized eigenvalue $h$, satisfying the two \textbf{grading restriction conditions}:
\begin{enumerate}
 \item Each $\mX_{[h]}$ is finite-dimensional.
 \item For any $h\in\CC$, $\mX_{[h-n]}=0$ for $n\in\NN$ sufficiently large.
\end{enumerate}
The (vector space) tensor product of two vertex operator algebras $\aU$ and $\aW$ is a vertex operator algebra \cite{FHL}, and if $\mX$ and $\mY$ are grading-restricted, generalized $\aU$- and $\aW$-modules, respectively, then $\mX\otimes\mY$ is a generalized $\aU\otimes\aW$-module with
\begin{equation*}
 (\mX\otimes \mY)_{[h]} = \bigoplus_{h_{\aU}+h_{\aW}=h} \mX_{[h_\aU]}\otimes \mY_{[h_\aW]}.
\end{equation*}
The module $\mX\otimes \mY$ will also satisfy the grading-restriction conditions if at least one of $\mX$ and $\mY$ is \textbf{strongly} grading-restricted in the sense that there are finitely many cosets $\lbrace h_i+\ZZ\rbrace$ in $\CC/\ZZ$ such that $\mX_{[h]}\neq 0$ (or $\mY_{[h]}\neq 0$) only if $h\in h_i+\ZZ$ for some $i$. From now on, we will refer to such strongly grading-restricted, generalized modules simply as \textbf{modules}.

A key feature of the tensor product theory of modules for a vertex operator algebra is a tensor product for each conformal equivalence class of spheres with two positively oriented punctures, one negatively oriented puncture, and local coordinates at the punctures. But to obtain braided tensor categories of vertex operator algebra modules, it is sufficient (see for instance \cite{HLZ8}) to focus on $P(z)$-tensor products, where $P(z)$ is the sphere with positively-oriented punctures at $0$ and $z\in\CC^\times$, a negatively-oriented puncture at $\infty$, and local coordinates $w\mapsto w$, $w\mapsto w-z$, $w\mapsto 1/w$, respectively. Therefore, we shall here abuse terminology and use the term ``vertex tensor category structure'' to refer only to the tensor product functors and natural isomorphisms corresponding to the spheres $P(z)$ for $z\in\CC^\times$.

Given a vertex operator algebra $\aV$ and a category $\cC$ of $\aV$-modules, the $P(z)$-tensor product of modules in $\cC$ is defined in terms of $P(z)$-intertwining maps. By \cite[Proposition 4.8]{HLZ3}, these are precisely maps of the form $\cY(\cdot,z)\cdot$, where
\begin{equation*}
 \cY: \mX_1\otimes\mX_2\rightarrow\mX_3[\log x]\lbrace x\rbrace
\end{equation*}
is a (logarithmic) intertwining operator of type $\binom{\mX_3}{\mX_1\,\mX_2}$ for modules $\mX_1$, $\mX_2$, $\mX_3$ in $\cC$ and the formal variable $x$ is specialized to $z\in\CC^\times$ using a choice of branch of logarithm. The range of a $P(z)$-intertwining map is the \textbf{algebraic completion} $\overline{\mX_3}$, defined as the direct product (rather than direct sum) of the homogeneous graded subspaces of the module $\mX_3$.

The $P(z)$-tensor product of two modules $\mX_1$, $\mX_2$ in $\cC$ is defined to be a representing object for the functor $\mathcal{V}[P(z)]^{\bullet}_{\mX_1, \mX_2}: \cC\rightarrow\mathcal{V}ec_\CC$,  where for a module $\mX$ in $\cC$, $\mathcal{V}[P(z)]^{\mX}_{\mX_1, \mX_2}$ is the space of $P(z)$-intertwining maps of type $\binom{\mX}{\mX_1\,\mX_2}$. That is, there are natural isomorphisms
\begin{equation*}
 \mathcal{V}[P(z)]^{\mX}_{\mX_1, \mX_2}\xrightarrow{\sim}\HHom_\cC(\mX_1\fus_{P(z)}\mX_2,\mX).
\end{equation*}
for all objects $\mX$ in $\cC$. In particular, there is a distinguished $P(z)$-intertwining map $\cdot\fus_{P(z)}\cdot$ of type $\binom{\mX_1\fus_{P(z)}\mX_2}{\mX_1\,\mX_2}$ (corresponding to $\Id_{\mX_1\fus_{P(z)}\mX_2}$) such that if $\cY$ is any intertwining operator of type $\binom{\mX}{\mX_1\,\mX_2}$ for some $\aV$-module $\mX$ in $\cC$, then there is a unique $\aV$-module homomorphism $\eta_\cY: \mX_1\fus_{P(z)}\mX_2\rightarrow\mX$ such that
\begin{equation*}
 \overline{\eta_\cY}(w_1\fus_{P(z)} w_2)=\cY(w_1,z)w_2
\end{equation*}
for $w_1\in\mX_1$, $w_2\in\mX_2$, where $\overline{\eta_\cY}$ is the natural extension of $\eta_\cY$ to the algebraic completions of $\mX_1\fus_{P(z)}\mX_2$ and $\mX$.

In addition to the $P(z)$-tensor product functors on the category $\cC$ of $\aV$-modules, vertex tensor category structure on $\cC$ includes the following natural isomorphisms:
\begin{enumerate}
 \item For continuous paths $\gamma$ in $\CC^\times$ beginning at $z_1$ and ending at $z_2$, \textbf{parallel transport isomorphisms} $T_{\gamma; \mX_1,\mX_2}: \mX_1\fus_{P(z_1)}\mX_2\rightarrow\mX_1\fus_{P(z_2)}\mX_2$.
 
 \item For $z\in\CC^\times$, \textbf{$P(z)$-unit isomorphisms} $l_{P(z),\Mod{X}}: \aV\fus_{P(z)}\Mod{X}\rightarrow\Mod{X}$ and $r_{P(z); \Mod{X}}: \Mod{X}\fus_{P(z)}\aV\rightarrow\Mod{X}$.
 
 \item For $z_1,z_2\in\CC^\times$ satisfying $\vert z_1\vert>\vert z_2\vert>\vert z_1-z_2\vert>0$, \textbf{$P(z_1,z_2)$-associativity isomorphisms}
 \begin{equation*}
  \cA_{P(z_1,z_2); \mX_1,\mX_2,\mX_3}: \mX_1\fus_{P(z_1)}(\mX_2\fus_{P(z_2)}\mX_3)\rightarrow(\mX_1\fus_{P(z_1-z_2)}\mX_2)\fus_{P(z_2)}\mX_3.
 \end{equation*}

\item For $z\in\CC^\times$, \textbf{$P(z)$-braiding isomorphisms} $\cR_{P(z); \mX_1,\mX_2}: \mX_1\fus_{P(z)}\mX_2\rightarrow\mX_2\fus_{P(-z)}\mX_1$.
\end{enumerate}

\begin{rema}
 The sphere $P(z_1,z_2)$ in the associativity isomorphisms has three positively oriented punctures at $0$, $z_1$, $z_2$ and one negatively oriented puncture at $\infty$. It can be obtained either by sewing together $P(z_1)$ and $P(z_2)$ spheres at the punctures $0$ and $\infty$, respectively, provided $\vert z_1\vert>\vert z_2\vert$, or by sewing together $P(z_1-z_2)$ and $P(z_2)$ spheres at the punctures $\infty$ and $z_2$, respectively, provided $\vert z_2\vert>\vert z_1-z_2\vert$. Thus the natural associativity isomorphisms in a vertex tensor category reflect the fact that these two sewing procedures yield conformally equivalent spheres with punctures and local coordinates.
\end{rema}

For conditions on $\cC$ guaranteeing the existence of these isomorphisms and details of their construction, see \cite{HLZ1}-\cite{HLZ8}; see also the expository article \cite{HL-rev} and \cite[Section 3.1]{CKM}. In order to obtain a braided tensor category structure from the vertex tensor category structure, one selects a particular tensor product functor, typically $\fus_{P(1)}$ which we shall denote simply by $\fus$ (it will be clear from the context when $\fus$ denotes a Deligne product category and when $\fus$ denotes a $P(1)$-tensor product). To obtain associativity and braiding isomorphisms for the single $P(1)$-tensor product, one needs to modify $P(1)$-braiding and $P(z_1,z_2)$-associativity isomorphisms using parallel transport. For details, see \cite{HLZ8}.

In general, it is difficult to show that a vertex tensor category $\cC$ is rigid, but it will frequently have a contragredient functor. Given a $\aV$-module $\mX=\bigoplus_{h\in\CC}\mX_{[h]}$, the graded dual vector space $\mX'=\bigoplus_{h\in\CC}\mX_{[h]}^*$ admits a $\aV$-module structure called the \textbf{contragredient module} \cite{FHL}. If $\aV$ is self-contragredient, that is, $\aV\cong\aV'$ as a $\aV$-module, then $\mX\mapsto\mX'$ defines a contragredient functor. By \cite[Proposition 5.3.2]{FHL}, $\mX'$ is simple if and only if $\mX$ is, and we have natural isomorphisms
\begin{align*}
 \HHom_{\cC}(\mX\fus\mY,\aV) & \cong\mathcal{V}[P(1)]^{\aV}_{\mX,\mY}\cong\mathcal{V}[P(1)]^{\mY'}_{\mX,\aV'}\nonumber\\
 &\cong\mathcal{V}[P(1)]^{\mY'}_{\mX,\aV}\cong\HHom_\cC(\mX\fus\aV,\mY')\cong\HHom_\cC(\mX,\mY')
\end{align*}
given by the definition of the $P(1)$-tensor product, symmetries of intertwining operators (see for example \cite[Proposition 5.5.2]{FHL}), the isomorphism $\aV\cong\aV'$, and the right unit isomorphisms.

The twist on a braided tensor category of modules is given by $e^{2\pi i L(0)}$. In particular, $\theta_\mX$ is a scalar precisely when $\mX$ is graded by a single coset of $\CC/\ZZ$.

\subsection{Deligne products of vertex algebraic tensor categories}

In this section, we show that under mild conditions, the Deligne product of braided tensor categories of modules for two vertex operator algebras is a braided tensor category of modules for the tensor product vertex operator algebra. Let $\aU$ and $\aW$ be vertex operator algebras, and let $\cU$ and $\cW$ be module categories for $\aU$ and $\aW$, respectively, that admit vertex tensor category structure. We first consider when we can obtain vertex tensor category structure on a suitable category of $\aU\otimes \aW$-modules.

It is natural to consider the full subcategory $\cC$ of $\aU\otimes \aW$-modules whose objects are (isomorphic to) direct sums of modules $\mX\otimes \mY$ where $\mX$ is a module in $\cU$ and $\mY$ is a module in $\cW$. For the following theorem, we make fairly minimal assumptions on the vertex operator algebras $\aU$ and $\aW$; for similar results along these lines see for instance \cite[Lemma 2.16]{lin} and \cite[Proposition 3.3]{CKLR}.
\begin{thm}\label{Cvrtxtenscat}
 If all fusion rules among modules in either $\cU$ or $\cW$ are finite, then the category $\cC$ of $\aU\otimes \aW$-modules admits the vertex tensor category structure given in \cite{HLZ1}-\cite{HLZ8}. Specifically, for modules  $\mU_i$ in $\cU$ and $\mW_i$ in $\cW$:
 \begin{enumerate}
  \item For $z\in\CC^\times$, $P(z)$-tensor products in $\cC$ are given by
  \begin{equation*}
   (\mU_1\otimes \mW_1)\pfus{z}(\mU_2\otimes \mW_2) = (\mU_1\pfus{z}\mU_2)\otimes(\mW_1\pfus{z}\mW_2),
  \end{equation*}
with tensor product $P(z)$-intertwining map
\begin{equation*}
 \pfus{z}=\pfus{z}\otimes\pfus{z}.
\end{equation*}

\item For a continuous path $\gamma$ in $\CC^\times$ beginning at $z_1$ and ending at $z_2$, the parallel transport isomorphism is
 \begin{equation*}
  T_{\gamma; \mU_1\otimes \mW_1, \mU_2\otimes \mW_2} = T_{\gamma; \mU_1, \mU_2}\otimes T_{\gamma; \mW_1, \mW_2}.
 \end{equation*}
 
 \item For $z\in\CC^\times$, the $P(z)$-unit isomorphisms are
 \begin{equation*}
  l_{P(z); \mU_i\otimes \mW_j} =l_{P(z);\mU_i}\otimes l_{P(z); \mW_j}
 \end{equation*}
and
\begin{equation*}
 r_{P(z); \mU_i\otimes \mW_j} = r_{P(z); \mU_i}\otimes r_{P(z); \mW_j}.
\end{equation*}

\item For $z_1,z_2\in\CC^\times$ such that $\vert z_1\vert>\vert z_2\vert>\vert z_1-z_2\vert>0$, the $P(z_1,z_2)$-associativity isomorphism is
\begin{equation*}
 \cA_{P(z_1,z_2); \mU_1\otimes \mW_1, \mU_2\otimes \mW_2, \mU_3\otimes \mW_3} =\cA_{P(z_1,z_2); \mU_1,\mU_2,\mU_3}\otimes\cA_{P(z_1, z_2); \mW_1, \mW_2, \mW_3}.
\end{equation*}

\item For $z\in\CC^\times$, the $P(z)$-braiding isomorphism is
\begin{equation*}
 \cR_{P(z); \mU_1\otimes \mW_1, \mU_2\otimes \mW_2} = \cR_{P(z); \mU_1,\mU_2}\otimes\cR_{P(z); \mW_1,\mW_2}.
\end{equation*}
 \end{enumerate}
 Moreover, if one of the categories $\cU$ or $\cW$ is semisimple and the other is closed under submodules and quotients, the category $\cC$ is abelian and thus is a braided tensor category.
\end{thm}
\begin{proof}
 Since the parallel transport and $P(z_1,z_2)$-associativity isomorphisms in a vertex tensor category of modules for a vertex operator algebra are entirely characterized in terms of tensor product intertwining maps (see \cite{HLZ8} for details), the indicated formulas for these isomorphisms in $\cC$ follow directly from the indicated identification of $P(z)$-tensor products and tensor product $P(z)$-intertwining maps in $\cC$. The formulas for the $P(z)$-unit isomorphisms and $P(z)$-braiding isomorphisms also follow from these identifications, together with the formulas $Y_{\mU_i\otimes \mW_j}=Y_{\mU_i}\otimes Y_{\mW_j}$ (from the definition in \cite[Section 4.6]{FHL}) and $e^{z L_{\mU\otimes \mW}(-1)} = e^{z L_{\mU}(-1)}\otimes e^{zL_{\mW}(-1)}$ (from the Leibniz formula). Moreover, all the isomorphisms indicated in the statement of the theorem are well defined; in particular, the convergence of compositions of intertwining maps in $\cC$ needed for the associativity isomorphisms follows from the convergence of intertwining maps in $\cU$ and $\cW$. Moreover, all coherence properties needed for a vertex tensor category follow from the corresponding coherence properties satisfied in $\cU$ and $\cW$.
 
To show that $\cC$ admits vertex tensor category structure, then, it remains to show that for modules $\mU_1$, $\mU_2$ in $\cU$ and $\mW_1$, $\mW_2$ in $\cW$, the pair
\begin{equation*}
 \left( (\mU_1\pfus{z}\mU_2)\otimes(\mW_1\pfus{z}\mW_2), \pfus{z}\otimes\pfus{z} \right)
\end{equation*}
indeed satisfies the the universal property of a $P(z)$-tensor product in $\cC$. For this, suppose $\mX$ is any module in $\cC$ and $\cI$ is any $P(z)$-intertwining map of type $\binom{\mX}{\mU_1\otimes \mW_1\,\mU_2\otimes \mW_2}$. We may identify $\mX$ with a (finite) direct sum $\mX=\bigoplus_i \mU^{(i)}\otimes \mW^{(i)}$ where the $\mU^{(i)}$ are modules in $\cU$ and the $\mW^{(i)}$ are modules in $\cW$. Under this identification, \cite[Theorem 2.10]{ADL} implies that the intertwining map $\cI$ may be identified with a (sum of) tensor products of intertwining maps: without loss of generality, we may assume that fusion rules in $\cU$ are finite, so for any $i$, $\lbrace \cI^{(1)}_{i,j}\rbrace_{j=1}^{J_i}$ is a basis for the space of $P(z)$-intertwining maps of type $\binom{\mU^{(i)}}{\mU_1\,\mU_2}$. Then
\begin{equation*}
 \cI=\sum_i \sum_{j=1}^{J_i} \cI^{(1)}_{i,j}\otimes \cI^{(2)}_{i,j},
\end{equation*}
where each $\cI^{(2)}_{i,j}$ is a $P(z)$-intertwining map of type $\binom{\mW^{(i)}}{\mW_1\,\mW_2}$.

Now the universal property of $P(z)$-tensor products in $\cU$ and $\cW$ imply that there are unique $\mU$-module homomorphisms
\begin{equation*}
 \eta^{(1)}_{i,j}: \mU_1\pfus{z}\mU_2\rightarrow \mU^{(i)}
\end{equation*}
such that $\cI^{(1)}_{i,j}=\overline{\eta^{(1)}_{i,j}}\circ\pfus{z}$, and there are unique $\mW$-module homomorphisms
\begin{equation*}
 \eta^{(2)}_{i,j}: \mW_1\pfus{z}\mW_2\rightarrow \mW^{(i)}
\end{equation*}
such that $\cI^{(2)}_{i,j}=\overline{\eta^{(2)}_{i,j}}\circ\pfus{z}$. Then the $\mU\otimes \mW$-module homomorphism $\eta=\sum_i\sum_{j=1}^{J_i} \eta^{(1)}_{i,j}\otimes\eta^{(2)}_{i,j}$ satisfies $\cI=\overline{\eta}\circ(\pfus{z}\otimes\pfus{z})$.

To show that $\eta$ is the unique $\mU\otimes \aW$-module homomorphism with this property, it suffices to show that $\pfus{z}\otimes\pfus{z}$ is a surjective intertwining map in the sense that the $\mU\otimes \mW$-module $(\mU_1\pfus{z}\mU_2)\otimes(\mW_1\pfus{z}\mW_2)$ is generated by projections to the conformal weight spaces of the form $\pi_h\left((u_1\pfus{z}u_2)\otimes(w_1\pfus{z}w_2)\right)$ for $h\in\CC$, $u_1\in \mU_1$, $u_2\in \mU_2$, $v_1\in \mW_1$, and $v_2\in \mW_2$. In fact,
\begin{equation*}
 \pi_h\left((u_1\pfus{z}u_2)\otimes(v_1\pfus{z}v_2)\right) =\sum_{h_{\mU}+h_{\mW}=h} \pi_{h_{\mU}}(u_1\pfus{z}u_2)\otimes\pi_{h_{\mW}}(v_1\pfus{z}v_2),
\end{equation*}
where the sum is finite. Since the $\mU\otimes \mW$-module generated by such projections is stable under $L_{\mU}(0)$ (and under $L_{\mW}(0)$), this submodule contains each $\pi_{h_{\mU}}(u_1\pfus{z}u_2)\otimes\pi_{h_{\mW}}(w_1\pfus{z}w_2)$ for $h_{\mU},h_{\mW}\in\CC$. Such vectors span $(\mU_1\pfus{z}\mU_2)\otimes(\mW_1\pfus{z}\mW_2)$ by \cite[Proposition 4.23]{HLZ3}, proving the desired the result. This completes the proof that $\cC$ admits the indicated vertex tensor category structure.

Now to show that $\cC$ is abelian, we may assume that $\cU$ is semisimple and that $\cW$ is closed under submodules and quotients. Since $\cC$ by definition includes direct sums, we just need to show that every $\mU\otimes \mW$-module homomorphism between modules in $\cC$ has kernel and cokernel in $\cC$. For this, it is sufficient to show that $\cC$ is closed under submodules and quotients.

First, we note that by definition of $\cC$ and semisimplicity of $\cU$, every module in $\cC$ is completely reducible as a weak $\mU$-module, and every weak $\mU$-submodule of a module in $\cC$ is also completely reducible. Specifically, if $\mX$ is a module in $\cC$, then $\mX\cong\bigoplus_i \mU^{(i)}\otimes \mW^{(i)}$ where the $\mU^{(i)}$ are distinct irreducible $\mU$-modules in $\cU$ and the $\mW^{(i)}$ are $\mW$-modules in $\cW$. Then if $\widetilde{\mX}\subseteq \mX$ is a weak $\mU$-submodule,
\begin{equation*}
 \widetilde{\mX}\cong\bigoplus_i \mU^{(i)}\otimes \widetilde{\mW}^{(i)},
\end{equation*}
where the $\widetilde{\mW}^{(i)}\subseteq \mW^{(i)}$ are subspaces (possibly equal to zero). If $\widetilde{\mX}$ is additionally a $\mU\otimes \mW$-submodule, then the $\widetilde{\mW}^{(i)}$ are $\mW$-submodules of $\mW^{(i)}$. Since $\cW$ is closed under submodules, it follows that the $\widetilde{\mW}^{(i)}$ are modules in $\cW$ and $\widetilde{\mX}$ is a module in $\cC$. Similarly, any quotient $\mX/\widetilde{\mX}$ where $\mX$ is a module in $\cC$ and $\widetilde{\mX}$ is an $\mU\otimes \mW$-submodule is isomorphic to
\begin{equation*}
\mX/\widetilde{\mX}\cong \bigoplus_i \mU^{(i)}\otimes \left(\mW^{(i)}/\widetilde{\mW}^{(i)}\right),
\end{equation*}
where the $\mU^{(i)}$ are distinct irreducible $\mU$-modules in $\cU$, the $\mW^{(i)}$ are $\mW$-modules in $\cW$, and the $\widetilde{\mW}^{(i)}$ are $\mW$-submodules. Since $\cW$ is closed under quotients, $\mX/\widetilde{\mX}$ is a module in $\cC$.
\end{proof}

\begin{rema}\label{Cbraidedtensor}
 Because the braided tensor category structure on $\cC$ derives from the vertex tensor category structure, we have the following identifications of structure isomorphisms in the braided tensor category structure:
\begin{enumerate}
 \item The unit isomorphisms are
 \begin{equation*}
  l_{\mU_i\otimes \mW_j}=l_{\mU_i}\otimes l_{\mW_j}
 \end{equation*}
and
\begin{equation*}
 r_{\mU_i\otimes \mW_j}=r_{\mU_i}\otimes r_{\mW_j}.
\end{equation*}

\item The associativity isomorphisms are
\begin{equation*}
 \cA_{\mU_1\otimes \mW_1, \mU_2\otimes \mW_2, \mU_3\otimes \mW_3}=\cA_{\mU_1,\mU_2,\mU_3}\otimes\cA_{\mW_1,\mW_2,\mW_3}.
\end{equation*}

\item The braiding isomorphisms are
\begin{equation*}
 \cR_{\mU_1\otimes \mW_1, \mU_2\otimes \mW_2} =\cR_{\mU_1,\mU_2}\otimes\cR_{\mW_1,\mW_2}.
\end{equation*}
\end{enumerate}
\end{rema}

\begin{rema}
 The identification of structure isomorphisms in $\cC$ with tensor products of structure isomorphisms in $\cU$ and $\cW$ do not follow simply from the existence of an isomorphism
 \begin{equation*}
  (\mU_1\otimes \mW_1)\fus_{P(z)}(\mU_2\otimes \mW_2)\cong(\mU_1\fus_{P(z)} \mU_2)\otimes(\mW_1\fus_{P(z)} \mW_2),
 \end{equation*}
but also from the identification
\begin{equation*}
 \fus_{P(z)}=\fus_{P(z)}\otimes\fus_{P(z)}
\end{equation*}
under this isomorphism.
\end{rema}

Now we can show that $\cC$ is actually the Deligne product of $\cU$ and $\cW$:
\begin{thm}\label{thm:CequalsDP}
 Suppose $\cU$ and $\cW$ are locally finite abelian categories, one of $\cU$ and $\cW$ is semisimple, and the other is closed under submodules and quotients. Then $\cC$ is braided tensor equivalent to the Deligne product category $\cU\fus\cW$.
\end{thm}
\begin{proof}
The functor $\cU\times\cW\rightarrow\cC$ given by $(\mX,\mY)\mapsto \mX\otimes \mY$ on objects and $(f,g)\mapsto f\otimes g$ on morphisms is right exact in both variables, so there is a unique functor
\begin{equation*}
 \cF: \cU\fus\cW\rightarrow\cC
\end{equation*}
determined on objects by $\cF(\mX\fus \mY)=\mX\otimes \mY$ and $\cF(f\fus g)=f\otimes g$ on morphisms. 

Conversely, we may assume that $\cU$ is semisimple and define a functor $\cG: \cC\rightarrow\cU\fus\cW$ on objects by $\cG(\mX\otimes\mY)=\mX\fus\mY$ for a simple object $\mX$ in $\cU$ and any object $\mY$ in $\cW$. For morphisms, we observe that if $\mX_1$ and $\mX_2$ are simple in $\cU$, then every morphism in
\begin{equation*}
 \HHom_\cC(\mX_1\otimes\mY_1, \mX_2\otimes\mY_2)
\end{equation*}
can be written as $f\otimes g$ where $f: \mX_1\rightarrow\mX_2$ is fixed (and is either an isomorphism or $0$) and $g: \mY_1\rightarrow\mY_2$ is some $\aV$-module homomorphism. Thus we can define $\cG$ on morphisms by $\cG(f\otimes g)=f\fus g$ for such $f$ and $g$.

Now, $\cF\circ\cG$ is naturally isomorphic to $\Id_\cC$ because it is an additive functor that equals the identity on indecomposable objects of $\cC$. For the other direction, the functor $\cG\circ\cF\circ\fus: \cU\times\cW\rightarrow\cU\fus\cW$  is right exact in both variables and sends $(\mX,\mY)$ to $\mX\fus\mY$ when $\mX$ is a simple module in $\cU$. Since the universal property of $\cU\fus\cW$ implies that the identity is the only endofunctor of $\cU\fus\cW$ with this property (up to natural isomorphism), we have $\cG\circ\cF\cong\Id_{\cU\fus\cW}$.

 Now because $\cU$ and $\cW$ are locally finite, spaces of intertwining operators are finite dimensional, so Theorem \ref{Cvrtxtenscat} applies showing $\cC$ is an (abelian) braided tensor category. Then Remark \ref{Cbraidedtensor} shows that $\cF$ is compatible with the braided tensor category structures on $\cU\fus\cW$ and $\cC$, and thus is an equivalence of braided tensor categories.
\end{proof}

\subsection{Algebras in vertex tensor categories}

The foundational theorem for algebras in braided tensor categories of modules for a vertex operator algebra is the following result of Huang, Kirillov, and Lepowsky:
\begin{thm}\label{thm:HKL}\cite[Theorem 3.2, Remark 3.3]{HKL}
Let $\cC$ be a category of modules for a vertex operator algebra $\aV$ that admits the vertex tensor category structure given in \cite{HLZ1}-\cite{HLZ8}, and thus also the braided tensor category structure of \cite{HLZ8}. Then the following two notions are equivalent:
\begin{enumerate}
 \item A vertex operator algebra $(\aA, Y_\aA,\vac,\omega)$ in $\cC$ (with the same vacuum and conformal vectors as $\aV$).
 \item A commutative associative algebra $(\aA,\mu_\aA,\iota_\aA)$ in $\cC$ with injective unit and trivial twist: $\theta_\aA=\Id_\aA$.
\end{enumerate}
\end{thm}
Since we are also concerned with algebras in $\cC_\oplus$, or more particularly algebras in $\cC_{\oplus}^{fin}$ when $\cC$ is semisimple, we need the generalization of this theorem to such algebras. We note that the conformal weight gradings of objects $\aA=\bigoplus_{s\in S} \aA_s$ will not necessarily satisfy grading restriction conditions, but such an object can still be a conformal vertex algebra in the sense of \cite{HLZ1}.
\begin{thm}\label{thm:HKL_for_Cplus}
 Let $\cC$ be a semisimple category of modules for a vertex operator algebra $\aV$ that admits the vertex tensor category structure given in \cite{HLZ1}-\cite{HLZ8}, and thus also the braided tensor category structure of \cite{HLZ8}. Then the following two notions are equivalent:
 \begin{enumerate}
  \item A conformal vertex algebra $(\aA,Y_\aA,\vac,\omega)$ in $\cC_\oplus^{fin}$ with the same vacuum and conformal vectors as $\aV$.
  \item A commutative associative algebra $(\aA,\mu_\aA,\iota_\aA)$ in $\cC_\oplus^{fin}$ with injective unit and $\theta_\aA=\Id_\aA$.
 \end{enumerate}
\end{thm}
\begin{proof}
Since the proof is the same as that of \cite[Theorem 3.2]{HKL} with minor changes, we only indicate how to obtain a vertex operator $Y_\aA$ from an algebra multiplication $\mu_\aA$, and vice versa.

Given a commutative associative algebra $\aA=\bigoplus_{s\in S} \aA_s$ with injective unit, trivial twist, and multiplication map
 \begin{equation*}
  \mu_\aA=\lbrace(\mu_\aA)_{s_1,s_2,t}\rbrace_{s_1,s_2,t\in S}\in\prod_{(s_1,s_2,t)\in S\times S\times S} \HHom_{\cC}(\aA_{s_1}\fus\aA_{s_2}, \aA_t),
 \end{equation*}
 the vertex operator $Y_\aA$ is defined to be unique the intertwining operator $Y_\aA$ of (weak) $\aV$-modules that satisfies
 \begin{equation*}
  Y_\aA(a_1, 1)a_2=\sum_{t\in S} \overline{(\mu_{\aA})_{s_1,s_2,t}}(a_1\fus a_2)
 \end{equation*}
for $a_1\in\aA_{s_1}$, $a_2\in\aA_{s_2}$. The sum is well defined because $\mu_\aA$ is a morphism in $\cC_\oplus$ and thus for fixed $s_1,s_2\in S$, $(\mu_\aA)_{s_1,s_2,t}=0$ for all but finitely many $t\in S$.

Conversely, given a conformal vertex algebra $\aA=\bigoplus_{s\in S} \aA_s$ in $\cC_\oplus^{fin}$ with the same vacuum and conformal vector as $\aV$, we would like to define
\begin{equation*}
 \mu_\aA\in\HHom_{\cC_\oplus} (\aA\fus\aA,\aA)\subseteq\prod_{(s_1,s_2, t)\in S\times S\times S} \HHom_\cC(\aA_{s_1}\fus \aA_{s_2}, \aA_t)
\end{equation*}
to be the tuple $\lbrace(\mu_\aA)_{s_1,s_2,t}\rbrace_{s_1,s_2,t\in S}$ where 
$$(\mu_\aA)_{s_1,s_2,t}: \aA_{s_1}\fus\aA_{s_2}\rightarrow\aA_t$$
 is the unique morphism such that
 \begin{equation*}
  \overline{(\mu_\aA)_{s_1,s_2,t}}(a_1\fus a_2)=\pi_t(Y_\aA(a_1, 1)a_2)
 \end{equation*}
for $a_1\in \aA_{s_1}$, $a_2\in \aA_{s_2}$, where $\pi_t$ is the canonical projection from $\overline{\aA}$ to $\aA_t$. However, we need to show that for fixed $s_1,s_2\in S$, we have $(\mu_\aA)_{s_1,s_2,t}=0$ for all but finitely many $t\in S$. In fact, this holds because $\aA$ is an object of $\cC_\oplus^{fin}$. For, if $\pi_t\circ Y_\aA\vert_{\aA_{s_1}\otimes\aA_{s_2}}\neq 0$, it is a
non-zero intertwining operator of type $\binom{\aA_t}{\aA_{s_1}\,\aA_{s_2}}$, and then $\HHom_{\cC}(\aA_{s_1}\fus\aA_{s_2},\aA_t)\neq 0$. But since $\aA_{s_1}\fus\aA_{s_2}$ is a direct sum of finitely many simple objects in $\cC$ and because these finitely many simple objects can occur in only finitely many $\aA_t$, this space of homomorphisms is non-zero for only finitely many $t$.
\end{proof}

\begin{rema}
 We may replace the condition $\theta_\aA=\Id_\aA$ with $\theta_\aA^2=\Id_\aA$ if we wish to allow $\frac{1}{2}\ZZ$-graded conformal vertex algebra extensions of $\aV$.
\end{rema}

\begin{rema}
 The conclusion of Theorem \ref{thm:HKL_for_Cplus} also applies when $\cC=\cU\fus\cW$ where $\cU$ is semisimple and $\cW$ is not, provided we restrict our attention to algebras of the form
 \begin{equation*}
  \aA=\bigoplus_{i\in I} \mU_i\otimes\mW_i
 \end{equation*}
where $\lbrace\mU_i\rbrace_{i\in I}$ is a set of simple modules in $\cU$ containing any given simple module of $\cU$ finitely many times.
\end{rema}

\subsection{The main theorems for vertex operator algebras}

We can now combine Theorem \ref{thm:canonicalalgebra}, Remark \ref{rema:canonicalalgebra}, Proposition \ref{prop:UA_WA_ribbon}, Theorem \ref{thm:main_bequiv}, Theorem \ref{thm:CequalsDP}, Theorem \ref{thm:HKL}, and Theorem \ref{thm:HKL_for_Cplus} into the following fundamental theorem relating conformal vertex algebra extensions of tensor product vertex operator algebras to braid-reversed equivalences:
\begin{thm}\label{thm:VOA}
 Let $\cU$ and $\cW$ be locally finite module categories for simple self-contragredient vertex operator algebras $\mU$ and $\mW$, respectively, that are closed under contragredients and admit the vertex tensor category structure given in \cite{HLZ1}-\cite{HLZ8}, and thus also the braided tensor category structure of \cite{HLZ8}. Assume moreover that $\cU$ is semisimple and $\cW$ is closed under submodules and quotients.
 \begin{enumerate}
  \item Suppose $\lbrace\mU_i\rbrace_{i\in I}$ is a set of representatives of equivalence classes of  simple modules in $\cU$ with $\mU_0=\mU$ and $\tau: \cU\rightarrow\cW$ is a braid-reversed tensor equivalence with twists satisfying $\theta_{\tau(\mU_i)} =\pm\tau(\theta_{\mU_i}^{-1})$ for $i\in I$. Then
  \begin{equation*}
   \aA=\bigoplus_{i\in I} \mU_i'\otimes\tau(\mU_i)
  \end{equation*}
is a $\frac{1}{2}\mathbb{Z}$-graded conformal vertex algebra extension of $\mU\otimes\mW$. Moreover, if $\cU$ is rigid, then $\aA$ is simple and the multiplication rules of $\aA$ satisfy $M_{\mU_i'\otimes\tau(\mU_i), \mU_j'\otimes\tau(\mU_j)}^{\mU_k'\otimes\tau(\mU_k)}=1$ if and only if $\mU_k$ occurs as a submodule of $\mU_i\fus\mU_j$.

\item Conversely, suppose $\cU$ and $\cW$ are both ribbon categories, $\lbrace\mU_i\rbrace_{i\in I}$ is a set of distinct simple modules in $\cU$ with $\mU_0=\mU$, and
\begin{equation*}
 \aA=\bigoplus_{i\in I} \mU_i\otimes\mW_i
\end{equation*}
is a simple $\frac{1}{2}\mathbb{Z}$-graded conformal vertex algebra extension of $\mU\otimes\mW$, where the $\mW_i$ are objects of $\cW$  satisfying
\begin{equation*}
 \dim\HHom_\cW(\mW,\mW_i)=\delta_{i,0}
\end{equation*}
and there is a partition $I=I^0\sqcup I^1$ of the index set with $0\in I^0$ and
\begin{equation*}
 \bigoplus_{i\in I^j} \mU_i\otimes\mW_i =\bigoplus_{n\in\frac{j}{2}+\ZZ} \aA_{(n)}
\end{equation*}
for $j=0,1$. Let $\cU_\aA\subseteq\cU$, respectively $\cW_\aA\subseteq\cW$, be the full subcategories whose objects are isomorphic to direct sums of the $\mU_i$, respectively of the $\mW_i$. Then:
\begin{enumerate}
\item $\cU_\aA$ and $\cW_\aA$ are ribbon subcategories of $\cU$ and $\cW$ respectively. Moreover, $\cW_\aA$ is semisimple with distinct simple objects $\lbrace\mW_i\rbrace_{i\in I}$.
\item There is a braid-reversed equivalence $\tau: \cU_\aA\rightarrow\cW_\aA$ such that $\tau(\mU_i)\cong\mW_i'$ for all $i\in I$.
 \end{enumerate}
 \end{enumerate}
\end{thm}
\begin{rema}
 Note that for part (2) of the theorem, we have $\dim\HHom_\cU(\mU,\mU_i)=\delta_{i,0}$ and $\dim_\cU\mU_i\neq 0$ for $i\in I$ because the $\mU_i$ are simple objects in a semisimple ribbon category. 
\end{rema}

As discussed in the Introduction, part (1) of Theorem \ref{thm:VOA} provides a partial answer to a question of Chongying Dong, while part (2) allows us to address a general question on the rationality of coset extensions of the form $\mU\otimes\mW\subseteq\aA$: If $\mU$ and $\mW$ are strongly rational vertex operator algebras (that is, simple, self-contragredient, CFT-type, $C_2$-cofinite, and rational), is the extension $\aA$ also strongly rational? In particular, is the category of grading-restricted, generalized $\aA$-modules semisimple? We answer these questions using results from \cite{KO} together with Theorem \ref{thm:VOA} and \cite[Theorem 2.3]{ENO}:
\begin{thm}\label{thm:dim}
 Suppose $\cU$ and $\cW$ are braided fusion categories of modules for simple self-contragredient vertex operator algebras $\mU$ and $\mW$, respectively, and
 \begin{equation*}
  \aA=\bigoplus_{i\in I} \mU_i\otimes\mW_i
 \end{equation*}
is a simple $\ZZ$-graded vertex operator algebra extension of $\mU\otimes\mW$ in $\cC=\cU\fus\cW$ where the $\mU_i$ are distinct simple modules in $\cU$ including $\mU_0=\mU$ and the $\mW_i$ are modules in $\cW$ such that
\begin{equation*}
 \dim\HHom_{\cW}(\mW,\mW_i)=\delta_{i,0}.
\end{equation*}
Then $\dim_\cC\aA>0$ and the category of (grading-restricted, generalized) $\aA$-modules in $\cC$ is a braided fusion category.
\end{thm}
\begin{proof}
 The rigidity and semisimplicity of the braided tensor category $\rep^0\aA$ of $\aA$-modules in $\cC$ (and indeed of the larger tensor category $\repA$) follow from \cite[Theorem 1.15]{KO} and \cite[Theorems 3.2 and 3.3]{KO} provided $\dim_\cC\aA\neq 0$. Then to see why $\rep^0\aA$ has finitely many isomorphism classes of simple modules, let $\lbrace\mM_j\rbrace_{j=1}^J$ for some $J\in\ZZ_+$ be a set of equivalence class representatives of simple modules in $\cC$. Because $\cC$ is semisimple, any irreducible module $\mX$ in $\rep^0\aA$ contains at least one such $\mM_j$, and the $\mU\otimes\mW$-module inclusion $\mM_j\hookrightarrow\mX$ together with Frobenius reciprocity imply there is a non-zero $\aA$-module homomorphism
 \begin{equation*}
  \cF\left(\bigoplus_{j=1}^J \mM_j\right)\rightarrow\mX,
 \end{equation*}
which is a surjection because $\mX$ is simple. Thus every irreducible $\aA$-module in $\cC$ is a quotient of $\cF\left(\bigoplus_{j=1}^J \mM_j\right)$, and it suffices to show that this module in $\repA$ has finitely many distinct irreducible quotients. Since $\repA$ is semisimple, it suffices to show that $\cF\left(\bigoplus_{j=1}^J \mM_j\right)$ is finitely generated. In fact, since the $\mM_j$ are simple modules in $\cC$, each $\cF(\mM_j)=\aA\fus\mM_j$ is singly-generated as an $\aA$-module by any non-zero $m_j\in\mM_j$.

It remains to show that $\dim_\cC\aA>0$. The braid-reversed tensor equivalence guaranteed by part (2) of Theorem \ref{thm:VOA} implies
 \begin{align*}
  \dim_\cC\aA = \sum_{i\in I} (\dim_{\cU}\mU_i)(\dim_{\cW}\mW_i) = \sum_{i\in I} (\dim_{\cU}\mU_i)(\dim_{\cU^{\rev}}\mU_i').
 \end{align*}
We note that since we assume $\aA$ is $\ZZ$-graded, each $\mU_i\otimes\mW_i$ must be $\mathbb{Z}$-graded, which means that the proper twist to use for calculating dimensions in $\cU^{\rev}$ is $\theta^{-1}$.

Now for each $i\in I$, recall the isomorphism $\delta_{\mU_i}: \mU_i\rightarrow\mU_i''$ of Remark \ref{rema:deltaX}. By \cite[Theorem 2.3]{ENO}, we have 
$$\mathrm{Tr}_{\mU_i}(\delta_{\mU_i})\mathrm{Tr}_{\mU_i'}((\delta_{\mU_i}^{-1})')>0,$$
where $\mathrm{Tr}_{\mU_i}(\delta_{\mU_i})$ and $\mathrm{Tr}_{\mU_i'}((\delta_{\mU_i}^{-1})')$ are the scalar multiples of $\Id_\mU$ determined by the compositions
\begin{equation*}
 \mU\xrightarrow{i_{\mU_i}} \mU_i\fus\mU_i'\xrightarrow{\delta_{\mU_i}\fus \Id_{\mU_i'}} \mU_i''\fus\mU_i'\xrightarrow{e_{\mU_i'}} \mU,\qquad \mU\xrightarrow{i_{\mU_i'}} \mU_i'\fus\mU_i''\xrightarrow{(\delta_{\mU_i}^{-1})'\fus \Id_{\mU_i''}} \mU_i'''\fus\mU_i''\xrightarrow{e_{\mU_i''}} \mU,
\end{equation*}
respectively. The definition of $\delta_{\mU_i}$ shows that $\mathrm{Tr}_{\mU_i}(\delta_{\mU_i})=\dim_{\cU}\mU_i$, so we just need to show
\begin{equation*}
 \mathrm{Tr}_{\mU_i'}((\delta_{\mU_i}^{-1})') = \dim_{\cU^{\rev}}\mU_i'.
\end{equation*}
We use the definitions of $\mathrm{Tr}_{\mU_i'}((\delta_{\mU_i}^{-1})')$ and dual of a homomorphism to obtain
\begin{equation*}
 \mathrm{Tr}_{\mU_i'}((\delta_{\mU_i}^{-1})')\cdot\Id_\mU= e_{\mU_i''}\circ((\delta_{\mU_i}^{-1})'\fus \Id_{\mU_i''})\circ i_{\mU_i'} = e_{\mU_i}\circ(\Id_{\mU_i'}\fus \delta_{\mU_i}^{-1})\circ i_{\mU_i'}.
\end{equation*}
On the other hand, using the definition of $\delta_{\mU_i}$ we have
\begin{align*}
 e_{\mU_i'} & = e_{\mU_i}\circ\cR_{\mU_i,\mU_i'}\circ(\theta_{\mU_i}\fus \Id_{\mU_i'})\circ(\delta_{\mU_i}^{-1}\fus\Id_{\mU_i'})\nonumber\\
  & = e_{\mU_i}\circ(\Id_{\mU_i'}\fus \delta_{\mU_i}^{-1})\circ\cR_{\mU_i'',\mU_i'}\circ(\theta_{\mU_i''}\fus\Id_{\mU_i'}),
\end{align*}
so that 
\begin{align*}
 e_{\mU_i}\circ(\Id_{\mU_i'}\fus \delta_{\mU_i}^{-1})\circ i_{\mU_i'} &  = e_{\mU_i'}\circ(\theta^{-1}_{\mU_i''}\fus\Id_{\mU_i'})\circ\cR_{\mU_i'',\mU_i'}^{-1}\circ i_{\mU_i'}\nonumber\\
  & =e_{\mU_i'}\circ((\theta^{-1}_{\mU_i'})'\fus\Id_{\mU_i'})\circ\cR_{\mU_i'',\mU_i'}^{-1}\circ i_{\mU_i'}\nonumber\\
    & =e_{\mU_i'}\circ(\Id_{\mU_i''}\fus\theta^{-1}_{\mU_i'})\circ\cR_{\mU_i'',\mU_i'}^{-1}\circ i_{\mU_i'}\nonumber\\
      & =e_{\mU_i'}\circ\cR_{\mU_i'',\mU_i'}^{-1}\circ(\theta^{-1}_{\mU_i'}\fus\Id_{\mU_i''})\circ i_{\mU_i'}\nonumber\\
      & =(\dim_{\cU^{\rev}}\mU_i')\cdot\Id_\mU,
\end{align*}
as required.
\end{proof}

If the vertex operator algebras $\mU$ and $\mW$ of the above theorem are strongly rational, we can now show that $\aA$ will also be strongly rational, provided it is CFT-type. Self-contragrediency and $C_2$-cofiniteness follow from the corresponding properties of $\mU$ and $\mW$ via \cite[Theorem 3.1]{Li-forms} and \cite[Proposition 5.2]{ABD}. Moreover, the argument in Lemma 3.6 and Proposition 3.7 of \cite{CM} (see also \cite[Proposition 4.15]{McRae}) shows that $\aA$ is rational provided the category of grading-restricted generalized $\aA$-modules is semsimple, which is the content of Theorem \ref{thm:dim}. Thus we have:
\begin{cor}\label{cor:dim}
In the setting of Theorem \ref{thm:dim}, suppose $\mU$ and $\mW$ are strongly rational vertex operator algebras. If $\aA$ is simple and CFT-type, then $\aA$ is strongly rational.
\end{cor}

\appendix

\section{Direct sum completion}
\label{app:dirsum}
In this Appendix, we gather the main constructions from \cite{AR} of the direct sum completion of a category. Given a category $\cC$ with possibly additional structures, $\cC_{\oplus}$ is essentially the smallest category closed under arbitrary direct sums. One may restrict to only countable direct sums, and this would be enough for our purposes. Even if $\cC$ is abelian, one cannot guarantee that $\cC_{\oplus}$ is abelian. Hence, we may wish to consider the smallest category containing $\cC$ closed under direct sums, kernels and cokernels; see \cite{CGR}. If $\cC$ is already semisimple, then $\cC_{\oplus}$ is also abelian (see for example Section 3.5 of \cite{J}).

If $\cC$ is a braided tensor category, then $\cC_\oplus$ also admits a braided tensor category structure. If $\cC$ has a system of isomorphisms $\theta_{\mX}$ that satisfy balancing, we get a system of balancing isomorphisms in $\cC_{\oplus}$ as well. However, even if $\cC$ is rigid, we cannot guarantee rigidity of $\cC_{\oplus}$. But we shall not need $\cC_\oplus$ to be rigid.

	Let $\cC$ be a $\CC$-linear additive category. We define the direct sum completion $\cC_\oplus$ as follows.
	The objects of $\cC_\oplus$ are:
	\begin{align}
	\obj(\cC_\oplus)=\bigg\lbrace\bigoplus_{s\in S}\mX_s \,\,\big\vert\,\, S \mathrm{\,is\, a\, set},\, \mX_s\in\obj(\cC)\,\mathrm{\,for\,all\,} s\in S\bigg\rbrace.
	\end{align}
	The morphisms are:
	\begin{align}
	\HHom_{\cC_\oplus}\bigg(\bigoplus_{s\in S}\mX_s,\bigoplus_{t\in T}\mY_t\bigg)
	=\left\lbrace  \left( \alpha, \{f_{s,t}\}_{s\in S, t\in \alpha(s)} \right) \right\rbrace / \sim,
	\end{align}
	with the following definitions.
	\begin{enumerate}
		\item Let $\cP_{\fin}(S)$ denote the set of finite subsets of a set $S$. Then 
		$\alpha : \cP_{\fin}(S)\rightarrow \cP_{\fin}(T)$ is a function that commutes with unions. By abuse of notation, we write $\alpha(s)=\alpha(\{s\})$ for all $s\in S$. 
    Since it is enough to specify $\alpha$ on singletons, we will often do so.
		Sometimes, $\alpha$ will map singletons to singletons, in which case, we shall simply write $\alpha(s) = t$ (or $\alpha: s\mapsto t$) if $\alpha(\{s\})=\{t\}$ and
		$\{f_s\}_{s\in S}$ in place of $\{f_{s,t}\}_{s\in S, t\in T}$.
		
		\item $f_{s,t}\in\HHom_{\cC}(\mX_s,\mY_t)$ for all $s\in S, t\in \alpha(s)$.
		\item $\sim$ is an equivalence relation defined by:
		\begin{align}
		\left( \alpha, \{f_{s,t}\}_{s\in S, t\in \alpha(s)} \right) \sim
		\left( \beta, \{g_{s,t}\}_{s\in S, t\in \beta(s)} \right) 
		\end{align}
		if and only if all of the following are satisfied:
		\begin{enumerate}
			\item $f_{s,t}=0$ if $t\in \alpha(s)\backslash \beta(s)$,
			\item $f_{s,t}=g_{s,t}$ if $t\in \alpha(s)\cap \beta(s)$,
			\item $g_{s,t}=0$ if $t\in \beta(s)\backslash \alpha(s)$.
		\end{enumerate}
	\end{enumerate}	
	The identity morphism on $\bigoplus_{s\in S} \mX_s$ is given by $\left( \Id_{\cP_{\fin}(S)},\{ \Id_{\mX_s}\}_{s\in S}\right)$. Note that we can also characterize morphism spaces as follows:
	\begin{equation*}
	 \HHom_{\cC_\oplus}\bigg(\bigoplus_{s\in S}\mX_s,\bigoplus_{t\in T}\mY_t\bigg) \subseteq\prod_{s\in S, t\in T} \HHom_\cC(\mX_s,\mY_t)
	\end{equation*}
is the subset of tuples $(f_{s,t})_{s\in S, t\in T}$ such that for any fixed $s\in S$, $f_{s,t}=0$ for all but finitely many $t\in T$.

	There are natural candidates for $\CC$-vector space structure on morphism spaces, 
	for a zero object, zero morphisms, and direct sums. With these, it was shown in \cite{AR} that $\cC_\oplus$ is again a $\CC$-linear additive category.
	There is also a fully faithful functor $\cI:\cC\rightarrow \cC_\oplus$ as follows:
	\begin{align*}
		\mX &\longmapsto \bigoplus\limits_{i\in \{0\}} \mX_i\,\,\mathrm{with}\,\, \mX_0:=\mX\\
		f &\longmapsto \left(\Id_{\{0\}}, \{f_{s,t}=f\}_{s\in \{0\}, t\in\{0\}}\right).
	\end{align*}
	We will sometimes abuse the notation and write $\mX=\cI(\mX)$.

	If $\cC$ is a tensor category, the tensor product bifunctor on $\cC_\oplus$ is defined by:
	\begin{align}
\bigoplus\limits_{s\in S} \mX_s\,\,\otimes\,\,\bigoplus_{t\in T} \mY_t&=	\bigoplus\limits_{(s,t)\in S\times T} \mX_s\otimes \mY_t,\\
\left(\alpha, \{f_{s,s'}\}_{s\in S, s'\in \alpha(s)} \right)\otimes 
\left(\beta, \{g_{t,t'}\}_{t\in T, t'\in \beta(t)} \right)
&= \left(\alpha \times \beta, \{f_{s,s'}\otimes g_{t,t'}\}_{(s,t)\in S\times T, (s',t')\in \alpha(s)\times \beta(t)} \right).
	\end{align}
	The unit object of $\cC_\oplus$ is 
  \begin{align}
  \cI(\one_{\cC})=\bigoplus_{s\in \{0\}}\one_{\cC}. 
  \end{align}
	The structure morphisms are defined as follows. Let 
	\begin{align}
	\overline{\mX}=\bigoplus\limits_{s\in S} \mX_s,\,\,	\overline{\mY}=\bigoplus\limits_{t\in T} \mY_t,\,\,
	\overline{\mZ}=\bigoplus\limits_{u\in U} \mZ_u.
	\end{align}
	If $\cC$ is rigid, let
	\begin{align}
	\overline{\mX}^\ast&=\bigoplus\limits_{s\in S}\mX_s^\ast.
	\end{align}
	Then define
	\begin{align}
	l_{\overline{\mX}}& = (\alpha:(0,s)\mapsto s, \{f_{(0,s)}=l_{\mX_s}\}_{s\in S}): 
  \cI(\one_{\cC})\otimes \overline{\mX} \rightarrow \overline{\mX},\\
	r_{\overline{\mX}} &= (\alpha:(s,0)\mapsto s, \{f_{(s,0)}=r_{\mX_s}\}_{s\in S}): 
\overline{\mX}\otimes \cI(\one_{\cC}) \rightarrow \overline{\mX},	\\
\cA_{\overline{\mX},\overline{\mY},\overline{\mZ}}&=
( \alpha: (s,(t,u)) \mapsto ((s,t),u), \{f_{(s,(t,u))}=\cA_{\mX_s,\mY_t,\mZ_u}\}_{(s,(t,u))\in S\times(T\times U)})\nonumber\\
&\quad\quad:
\overline{\mX}\otimes(\overline{\mY}\otimes \overline{\mZ}) \rightarrow 
(\overline{\mX}\otimes \overline{\mY})\otimes \overline{\mZ},\\
\cR_{\overline{\mX},\overline{\mY}}&=
( \alpha: (s,t) \mapsto (t,s), \{f_{(s,t)}=\cR_{\mX_s,\mY_t}\}_{(s,t)\in S\times T}):
\overline{\mX}\otimes \overline{\mY}\rightarrow \overline{\mY}\otimes \overline{\mX},\\
e_{\overline{\mX}}&= (\alpha: (s',s) \mapsto 0, \{f_{(s',s)}=\delta_{s',s}e_{\mX_s}\}_{(s',s)\in S\times S}):
\overline{\mX}^\ast\otimes \overline{\mX}\rightarrow \cI(\one_{\cC}),\\
\theta_{\overline{\mX}}&= 
( \alpha: s \mapsto s, \{f_s=\theta_{\mX_s}\}_{s\in S}):
\overline{\mX}\rightarrow \overline{\mX}.
	\end{align}
	These definitions give requisite structures on $\cC_{\oplus}$, except for rigidity. In particular, $\theta_{\overline{\mX}}$ is proved to satisfy the balancing axiom in  \cite{AR}.

\end{document}